\documentclass{article}
\usepackage{alphabeta,amsfonts,amsmath,amssymb,amsthm,authblk}
\usepackage[english]{babel}
\usepackage{booktabs}
\usepackage[makeroom]{cancel}
\usepackage{comment}
\usepackage{enumerate}
\usepackage[shortlabels]{enumitem}
\usepackage[margin = 1 in]{geometry}
\usepackage{graphicx}
\usepackage[usenames,dvipsnames]{xcolor}
\usepackage[linktocpage,colorlinks=true,linkcolor=Fuchsia!70!Magenta,urlcolor=blue!60!cyan,citecolor=Green!70!PineGreen,
	hypertexnames=false]{hyperref}
\usepackage[utf8]{inputenc}
\usepackage{leftidx}
\usepackage{mathrsfs,mathtools}
\usepackage{relsize}
\usepackage{scalerel,stackengine}
\usepackage{tikz,tikz-cd}
\usepackage[titles]{tocloft}
\usepackage{xspace}

\newcommand{\N}{\mathbb{N}}
\newcommand{\Z}{\mathbb{Z}}
\newcommand{\Q}{\mathbb{Q}}
\newcommand{\R}{\mathbb{R}}
\newcommand{\C}{\mathbb{C}}
\newcommand{\F}{\mathbb{F}}

\newcommand{\cA}{\mathcal{A}}

\newcommand{\cB}{\mathcal{B}}

\newcommand{\cF}{\mathcal{F}}
\newcommand{\sF}{\mathscr{F}}

\newcommand{\sH}{\mathscr{H}}
\newcommand{\cI}{\mathcal{I}}
\newcommand{\cK}{\mathcal{K}}
\newcommand{\cM}{\mathcal{M}}
\newcommand{\cN}{\mathcal{N}}
\newcommand{\cS}{\mathcal{S}}

\newcommand{\e}{\varepsilon}
\newcommand{\Om}{\Omega} 
\newcommand{\om}{\omega}
\newcommand{\blambda}{\boldsymbol{\lambda}}

\DeclareMathOperator{\Tr}{Tr}

\DeclareMathOperator{\spn}{span}
\DeclareMathOperator{\id}{id}
\DeclareMathOperator{\im}{im}

\newcommand{\sni}{\unlhd_{\mathrm{s}}}
\newcommand{\dom}{\mathrm{dom}}

\newcommand{\la}{\langle}
\newcommand{\ra}{\rangle}

\newcommand{\sh}{\text{\smaller $\#$}}

\DeclareMathOperator{\sa}{sa}
\newcommand{\vertiii}[1]{{\left\vert\kern-0.25ex\left\vert\kern-0.25ex\left\vert #1 
    \right\vert\kern-0.25ex\right\vert\kern-0.25ex\right\vert}}
\DeclareMathOperator{\supp}{supp}

\newcommand{\iotimes}{\hat{\otimes}_i}

\newcommand{\wh}{\widehat}
\newcommand\pprec{\prec\mkern-5mu\prec}
\stackMath
\newcommand\wch[1]{%
\savestack{\tmpbox}{\stretchto{%
  \scaleto{%
    \scalerel*[\widthof{\ensuremath{#1}}]{\kern-.6pt\bigwedge\kern-.6pt}%
    {\rule[-\textheight/2]{1ex}{\textheight}}
  }{\textheight}%
}{0.5ex}}%
\stackon[1pt]{#1}{\scalebox{-1}{\tmpbox}}%
}

\newcommand{\aff}{%
  \mathrel{\raisebox{0.75pt}{\text{\scalebox{0.85}{$\eta$}}}}%
}
\newcommand\numberthis{\addtocounter{equation}{1}\tag{\theequation}}
\makeatletter
\newcommand{\oset}[3][0ex]{%
  \mathrel{\mathop{#3}\limits^{
    \vbox to#1{\kern-2\ex@
    \hbox{$\scriptstyle#2$}\vss}}}}
\makeatother

\theoremstyle{plain}
\newtheorem{prop}{Proposition}[subsection]
\theoremstyle{plain}

\theoremstyle{plain}
\newtheorem{lem}[prop]{Lemma}
\theoremstyle{plain}

\theoremstyle{plain}
\newtheorem{thm}[prop]{Theorem}
\theoremstyle{plain}

\theoremstyle{plain}
\newtheorem{cor}[prop]{Corollary}
\theoremstyle{plain}

\theoremstyle{plain}
\newtheorem{fact}[prop]{Fact}
\theoremstyle{plain}

\theoremstyle{definition}
\newtheorem{defi}[prop]{Definition}
\theoremstyle{definition}

\theoremstyle{definition}
\newtheorem{nota}[prop]{Notation}
\theoremstyle{definition}
\newtheorem*{notan}{Notation}
\theoremstyle{definition}
\newtheorem{conv}[prop]{Convention}
\theoremstyle{definition}

\theoremstyle{definition}
\newtheorem{ex}[prop]{Example}
\theoremstyle{definition}

\theoremstyle{definition}

\theoremstyle{definition}

\theoremstyle{definition}
\newtheorem{rem}[prop]{Remark}
\theoremstyle{definition}

\theoremstyle{definition}

\theoremstyle{definition}

\theoremstyle{definition}
\newtheorem*{ack}{Acknowledgements}

\makeatletter
\renewenvironment{proof}[1][\proofname]{%
  \par\pushQED{\qed}\normalfont%
  \topsep6\p@\@plus6\p@\relax
  \trivlist\item[\hskip\labelsep\bfseries#1\@addpunct{.}]%
  \ignorespaces
}{%
  \popQED\endtrivlist\@endpefalse
}
\makeatother
\begin{document}

\title{Higher derivatives of operator functions \\
in ideals of von Neumann algebras}
\author{Evangelos A. Nikitopoulos\thanks{Supported by NSF grant DGE 2038238}}
\affil{Department of Mathematics, University of California San Diego\protect\\
\noindent 9500 Gilman Drive, La Jolla, CA 92093-0112 (USA)\protect\\
Email: {\tt \href{mailto:enikitop@ucsd.edu}{enikitop@ucsd.edu}}}
\date{\vspace{-5ex}}

\maketitle

\vspace{-1.2mm}\begin{abstract}
Let $\mathcal{M}$ be a von Neumann algebra and $a$ be a self-adjoint operator affiliated with $\mathcal{M}$.
We define the notion of an ``integral symmetrically normed ideal" of $\mathcal{M}$ and introduce a space $OC^{[k]}(\mathbb{R}) \subseteq C^k(\mathbb{R})$ of functions $\mathbb{R} \to \mathbb{C}$ such that the following result holds:
for any integral symmetrically normed ideal $\mathcal{I}$ of $\mathcal{M}$ and any $f \in OC^{[k]}(\mathbb{R})$, the operator function $\mathcal{I}_{\mathrm{sa}} \ni b \mapsto f(a+b)-f(a) \in \mathcal{I}$ is $k$-times continuously Fr\'{e}chet differentiable, and the formula for its derivatives may be written in terms of multiple operator integrals.
Moreover, we prove that if $f \in \dot{B}_1^{1,\infty}(\mathbb{R}) \cap \dot{B}_1^{k,\infty}(\mathbb{R})$ and $f'$ is bounded, then $f \in OC^{[k]}(\mathbb{R})$.
Finally, we prove that all of the following ideals are integral symmetrically normed:
$\mathcal{M}$ itself, separable symmetrically normed ideals, Schatten $p$-ideals, the ideal of compact operators, and --- when $\mathcal{M}$ is semifinite --- ideals induced by fully symmetric spaces of measurable operators.

$\,$\vspace{-0.5mm}


$\,$\vspace{-0.5mm}

\end{abstract}
\tableofcontents
\clearpage

\section{Introduction}\label{sec.intro}

\begin{notan}
Let $X$ be a topological space, $V,W$ be normed vector spaces, and $H$ be a complex Hilbert space.
\begin{enumerate}[label=(\alph*),leftmargin=2\parindent]
    \item $\cB_X$ is the Borel $\sigma$-algebra on $X$.
    \item $B(V;W)$ is the space of bounded linear maps $V \to W$ with operator norm $\|\cdot\| = \|\cdot\|_{V \to W}$. Also, $B(V) \coloneqq B(V;V)$. Finally, $B(H)_{\sa} \coloneqq \{A \in B(H) : A^*=A\}$.
    \item If $A$ is a (possibly unbounded) self-adjoint operator on $H$, then $P^A \colon \cB_{\sigma(A)} \to B(H)$ is its projection-valued spectral measure. If $f \colon \R \to \C$ is Borel measurable, then we define $f(A) \coloneqq \int_{\sigma(A)} f(\lambda)\,P^A(d\lambda)$. (Please see Section \ref{sec.SpecTh}.)
\end{enumerate}
\end{notan}

\subsection{Known results}\label{sec.knownres}

Let $H$ be a complex Hilbert space.
Given an appropriately regular scalar function $f \colon \R \to \C$, one of the goals of perturbation theory is to ``Taylor expand" the \textit{operator function} that takes a self-adjoint operator $A$ on $H$ and maps it to the operator $f(A)$.
This delicate problem has its beginnings in the work of Yu. L. Daletskii and S. G. Krein.
In their seminal paper \cite{daletskiikrein}, they proved that if $f \colon \R \to \C$ is $2k$-times continuously differentiable and $A,B \in B(H)_{\sa}$, then the curve $\R \ni t \mapsto f(A+tB) \in B(H)$ is $k$-times differentiable in the operator norm and
\[
\frac{d^k}{dt^k}\Big|_{t=0} f(A+tB) = k!\underbrace{\int_{\sigma(A)}\cdots\int_{\sigma(A)}}_{k+1 \; \mathrm{times}} f^{[k]}(\lambda_1,\ldots,\lambda_{k+1})\,P^A(d\lambda_1)\,B\cdots P^A(d\lambda_k)\,B\,P^A(d\lambda_{k+1}), \numberthis\label{eq.formalopder}
\]
where $f^{[k]} \colon \R^{k+1} \to \C$ is the $k^{\text{th}}$ \textit{divided difference} (Section \ref{sec.divdiffandpert}) of $f$, defined recursively as
\[
f^{[0]} \coloneqq f \; \text{ and } \; f^{[k]}(\lambda_1,\ldots,\lambda_{k+1}) \coloneqq \frac{f^{[k-1]}(\lambda_1,\ldots,\lambda_k)-f^{[k-1]}(\lambda_1,\ldots,\lambda_{k-1},\lambda_{k+1})}{\lambda_k-\lambda_{k+1}} \numberthis\label{eq.divdiffintro}
\]
for $(\lambda_1,\ldots,\lambda_{k+1}) \in \R^{k+1}$.
The reader might be (rightly) puzzled by the multiple integral in \eqref{eq.formalopder}, since standard projection-valued measure theory only allows for the integration of \textit{scalar-valued} functions.
Indeed, while the innermost integral $\int_{\sigma(A)} f^{[k]}(\lambda_1,\cdots,\lambda_{k+1})\,P^A(d\lambda_1)$ makes sense using standard theory, it is already unclear how to integrate the map
\[
\lambda_2 \mapsto \int_{\sigma(A)} f^{[k]}(\lambda_1,\lambda_2,\ldots,\lambda_{k+1})\,P^A(d\lambda_1)\,B
\]
with respect to $P^A$.
Daletskii and Krein dealt with this by using a Riemann--Stieltjes-type construction to define $\int_s^t \Phi(r) \,P^A(dr)$ for certain \textit{operator-valued} functions $\Phi \colon [s,t] \to  B(H)$, where $\sigma(A) \subseteq [s,t]$.
This approach, which requires rather stringent regularity assumptions on $\Phi$, allows one to make sense of the right hand side of \eqref{eq.formalopder} as an iterated operator-valued integral --- in other words, a multiple operator integral.

Now, for $j \in \{1,\ldots,k+1\}$, let $(\Om_j,\sF_j)$ be a measurable space and $P_j \colon \sF_j \to B(H)$ be a projection-valued measure.
Emerging naturally from the formula \eqref{eq.formalopder} is the general problem of making sense of
\[
\big(I^{P_1,\ldots,P_{k+1}} \varphi \big)[b_1,\ldots,b_k] = \int_{\Om_{k+1}}\cdots\int_{\Om_1} \varphi(\om_1,\ldots,\om_{k+1}) \, P_1(d\om_1) \, b_1 \cdots P_k(d\om_k) \, b_k \,P_{k+1}(d\om_{k+1}) \numberthis\label{eq.formalMOI}
\]
for certain functions $\varphi \colon \Om_1 \times \cdots \times \Om_{k+1} \to \C$ and operators $b_1 , \ldots, b_k \in B(H)$.
An object successfully doing so is called a \textit{multiple operator integral} (MOI).
Under the assumption that $H$ is separable, these have been studied and applied to various branches of noncommutative analysis extensively.
Please see A. Skripka and A. Tomskova's book \cite{skripka} for an excellent survey of the MOI literature and its applications.

In this paper, we shall make use of the ``separation of variables" approach to defining \eqref{eq.formalMOI}.
For separable $H$, this approach was developed by V. V. Peller \cite{peller1,peller2} and N. A. Azamov, A. L. Carey, P. G. Dodds, and F. A. Sukochev \cite{azamovetal} in order to differentiate operator functions at unbounded operators.
The present author extended the approach to the case of a non-separable Hilbert space in \cite{nikitopoulosMOI}.
We review the relevant definitions and results in Sections \ref{sec.wstarint} and \ref{sec.MOIsinI}.
Henceforth, any MOI expression we write or reference is to be interpreted in accordance with Section \ref{sec.MOIsinI} (specifically, Theorem \ref{thm.MOIsinM}).
\pagebreak

Now, we quote the best known general results on higher derivatives of operator functions.
If $\dot{B}_q^{s,p}(\R^m)$ is the homogeneous Besov space (Definition \ref{def.Besov}), then we write
\[
PB^k(\R) \coloneqq \dot{B}_1^{k,\infty}(\R) \cap \big\{f \in C^k(\R) : f^{(k)} \text{ is bounded}\big\} \numberthis\label{eq.PBspace}
\]
for the \textbf{$\boldsymbol{k^{\text{\textbf{th}}}}$ Peller--Besov space}.
It turns out that $PB^1(\R) \cap PB^k(\R) = PB^1(\R) \cap \dot{B}_1^{k,\infty}(\R)$.
(Please see the paragraph containing Equation \eqref{eq.f'series} at the end of Section \ref{sec.PellerII}.)

\begin{thm}[Peller \cite{peller1}]\label{thm.pellerder}
Let $H$ be a separable complex Hilbert space, $A$ be a self-adjoint operator on $H$, and $B \in B(H)_{\sa}$.
If $f \in PB^1(\R) \cap PB^k(\R)$, then the map $\R \ni t \mapsto f(A+tB)-f(A) \in B(H)$ is $k$-times differentiable in the operator norm, and the formula \eqref{eq.formalopder} holds.
\end{thm}

This is Theorem 5.6 in \cite{peller1}.
To quote the relevant result from \cite{azamovetal}, we need some additional terminology.
First, recall that if $H$ is a complex Hilbert space, then $\cM \subseteq B(H)$ is a \textbf{von Neumann algebra} if it is a unital $\ast$-subalgebra that is closed in the weak operator topology (WOT).
Second, suppose $\cI \subseteq \cM$ is a $\ast$-ideal with another norm $\|\cdot\|_{\cI}$ on it.
We call $\cI$ an \textbf{invariant operator ideal} if $(\cI,\|\cdot\|_{\cI})$ is a Banach space, $\|r\| \leq \|r\|_{\cI} = \|r^*\|_{\cI}$ for $r \in \cI$, and $\cI$ is \textbf{symmetrically normed}, i.e., $\|arb\|_{\cI} \leq \|a\|\,\|r\|_{\cI}\|b\|$ for $r \in \cI$ and $a,b \in \cM$.
Third, an invariant operator ideal $\cI$ has \textbf{property (F)} if whenever $(a_j)_{j \in J}$ is a net in $\cI$ such that $\sup_{j \in J}\|a_j\|_{\cI} < \infty$ and $a_j \to a \in \cM$ in the strong$^*$ operator topology (S$^*$OT), we get $a \in \cI$ and $\|a\|_{\cI} \leq \sup_{j \in J}\|a_j\|_{\cI}$.
Finally, we write $W_k(\R)$ for the $k^{\text{\textit{th}}}$ \textit{Wiener space} (Definition \ref{def.Wk}) of functions $f \colon \R \to \C$ that are Fourier transforms of complex measures with finite $k^{\text{th}}$ moment.

\begin{thm}[Azamov--Carey--Dodds--Sukochev \cite{azamovetal}]\label{thm.azamovetalder}
\hspace{-0.5mm}Let \hspace{-0.3mm}$H$\hspace{-0.4mm} be \hspace{-0.4mm}a\hspace{-0.4mm} separable \hspace{-0.4mm}complex\hspace{-0.4mm} Hilbert \hspace{-0.4mm}space,\hspace{-0.4mm} $\cM \hspace{-0.4mm}\subseteq\hspace{-0.4mm} B(H)$ be a von Neumann algebra, and $a$ be a self-adjoint operator on $H$ affiliated with $\cM$ (Definition \ref{def.aff}).
If $\cI \subseteq \cM$ is an invariant operator ideal with property (F), $\cI_{\sa} \coloneqq \{b \in \cI : b^*=b\}$, and $f \in W_{k+1}(\R)$, then the map $\cI_{\sa} \ni b \mapsto f_a(b)\coloneqq f(a+b)-f(a) \in \cI$ is $k$-times Fr\'{e}chet differentiable (Definition \ref{def.frechder}) in the $\cI$-norm $\|\cdot\|_{\cI}$ and
\[
\partial_{b_1}\cdots\partial_{b_k}f_a(0) = \sum_{\pi \in S_k}\underbrace{\int_{\sigma(a)}\cdots\int_{\sigma(a)}}_{k+1 \; \mathrm{times}} f^{[k]}(\lambda_1,\ldots,\lambda_{k+1})\,P^a(d\lambda_1)\,b_{\pi(1)}\cdots P^a(d\lambda_k)\,b_{\pi(k)}\,P^a(d\lambda_{k+1}),
\]
for all $b_1,\ldots,b_k \in \cI_{\sa}$, where $S_k$ is the symmetric group on $k$ letters.
\end{thm}

This is Theorem 5.7 in \cite{azamovetal}.
As is noted in \cite{azamovetal}, the motivating example of an invariant operator ideal with property (F) comes from the theory of \textit{symmetric operator spaces}.
Indeed, if $(E,\|\cdot\|_E)$ is a symmetric Banach function space with the Fatou property, $(\cM,\tau)$ is a semifinite von Neumann algebra (Definition \ref{def.trace}), and $(E(\tau),\|\cdot\|_{E(\tau)})$ is the symmetric space of $\tau$-measurable operators induced by $E$, then
\[
(\cI,\|\cdot\|_{\cI}) \coloneqq (E(\tau) \cap \cM, \|\cdot\|_{E(\tau) \cap \cM}) = (E(\tau) \cap \cM, \max\{\|\cdot\|_{E(\tau)},\|\cdot\|_{\cM}\}) \numberthis\label{eq.symmopideal}
\]
is an invariant operator ideal with property (F).
(Please see Section \ref{sec.symmopsp} for the meanings of the preceding terms.)
Though Theorem \ref{thm.azamovetalder} applies to this interesting general setting, the result demands much more regularity of the scalar function $f$ than Theorem \ref{thm.pellerder}.
(Indeed, $W_k(\R) \subsetneq PB^1(\R) \cap PB^k(\R)$.)
It has remained an open problem (Problem 5.3.22 in \cite{skripka}) to find less restrictive conditions for higher Fr\'{e}chet differentiability of operator functions in the symmetric operator space ideals described above.
The present paper makes substantial progress on this problem:
a corollary of our main results is that if $E$ is fully symmetric (a weaker condition than the Fatou property), then the result of Theorem \ref{thm.azamovetalder} holds for $(\cI,\|\cdot\|_{\cI})$ as in \eqref{eq.symmopideal} with $f \in PB^1(\R) \cap PB^k(\R)$.
In other words, we are able to close the regularity gap between Theorems \ref{thm.pellerder} and \ref{thm.azamovetalder} in the (fully) symmetric operator space context.
Moreover, we are able, for the first time in the literature on higher derivatives of operator functions, to remove the separability assumption on $H$ by using the MOI development from \cite{nikitopoulosMOI}.

\begin{rem}[Other related work]
The Schatten $p$-ideals have property (F), so Theorem \ref{thm.azamovetalder} applies to them when the underlying Hilbert space is separable.
There are, however, much sharper results known about differentiability of operator functions in the Schatten $p$-ideals (again, when the underlying Hilbert space is separable).
Please see \cite{lemerdyskripka,lemerdymcdonald}.
Also, there is a seminal paper of de Pagter and Sukochev \cite{depagtersukochev} that studies once Gateaux differentiability of operator functions in certain symmetric operator spaces at measurable operators;
we discuss its relation to the results in this paper in Remark \ref{rem.dePS}.
\end{rem}

\subsection{Main results}\label{sec.mainres}

Let $H$ be a complex Hilbert space, $\cM \subseteq B(H)$ be a von Neumann algebra, and $a$ be a self-adjoint operator affiliated with $\cM$ (Definition \ref{def.aff}).
We recall from the previous section that our goal is to differentiate the operator function $\cI_{\sa} \ni b \mapsto f(a+b)-f(a) \in \cI$, where $(\cI,\|\cdot\|_{\cI})$ is some normed ideal of $\cM$ and $f \in PB^1(\R) \cap PB^k(\R)$.
(Please see \eqref{eq.PBspace}.)
The ideals we consider are the \textit{integral symmetrically normed ideals} (ISNIs).
The definition of integral symmetrically normed is an ``integrated" version of the symmetrically normed condition $\|arb\|_{\cI} \leq \|a\|\,\|r\|_{\cI}\|b\|$.
Loosely speaking, $(\cI,\|\cdot\|_{\cI})$ is integral symmetrically normed if
\[
\Bigg\|\int_{\Sigma} A(\sigma) \, r \,B(\sigma) \,\rho(d\sigma)\Bigg\|_{\cI} \leq \|r\|_{\cI} \int_{\Sigma}\|A\|\,\|B\|\,d\rho, \; r \in \cI.
\]
The precise definition  (Definition \ref{def.idealproperties}\ref{item.ISNI}) is slightly technical, so we omit it from this section. Our first main result comes in the form of a list of interesting examples of ISNIs.

\begin{thm}[Examples of ISNIs]\label{thm.mainISNI}
Let $H$ be an arbitrary (not-necessarily-separable) complex Hilbert space and $\cM \subseteq B(H)$ be a von Neumann algebra.
\begin{enumerate}[label=(\roman*),font=\normalfont,leftmargin=2\parindent]
    \item The ideal $(\cM,\|\cdot\|)$ is integral symmetrically normed.\label{item.triv}
    \item If $(\cI,\|\cdot\|_{\cI})$ is a separable symmetrically normed ideal of $\cM$, then $\cI$ is integral symmetrically normed.\label{item.sep}
    \item The ideal $(\mathcal{K}(H),\|\cdot\|)$ of compact operators is an integral symmetrically normed ideal of $B(H)$.\label{item.compact}
    \item If $1 \leq p \leq \infty$, then the ideal of $(\cS_p(H),\|\cdot\|_{\cS_p})$ of Schatten $p$-class operators (Definition 2.2.1 in \cite{nikitopoulosMOI}) is an integral symmetrically normed ideal of $B(H)$.
    (For us, $\cS_{\infty}(H) = B(H)$.)\label{item.Sp}
    \item Suppose $(\cM,\tau)$ is semifinite (Definition \ref{def.trace}).
    If $(E,\|\cdot\|_E)$ is a fully symmetric space of $\tau$-measurable operators (Definition \ref{def.symmopsp}\ref{item.fullysymm}) and $(\cI,\|\cdot\|_{\cI}) \coloneqq (E \cap \cM, \|\cdot\|_{E \cap \cM}) = (E \cap \cM, \max\{\|\cdot\|_E,\|\cdot\|\})$, then $\cI$ is an integral symmetrically normed ideal of $\cM$.\label{item.fs}
\end{enumerate}
\end{thm}
\begin{proof}
Item \ref{item.triv} is Example \ref{ex.M}, item \ref{item.sep} is part of Proposition \ref{prop.sep}, item \ref{item.compact} follows from Proposition \ref{prop.compact} (or Remark \ref{rem.compact}), item \ref{item.Sp} is a special case of Example \ref{ex.Lp}, and item \ref{item.fs} is Theorem \ref{thm.FSISN}.
\end{proof}

With these in mind, we now state our second main result.
Recall that $f^{[k]} \colon \R^{k+1} \to \C$ is the $k^{\text{th}}$ divided difference of $f \colon \R \to \C$ (please see \eqref{eq.divdiffintro}), $S_k$ is the symmetric group on $k$ letters, and all multiple operator integral (MOI) expressions as in \eqref{eq.formalMOI} are to be interpreted in accordance with Section \ref{sec.MOIsinI}.

\begin{thm}[Derivatives of operator functions in ISNIs]\label{thm.mainopder}
Let $H$ be an arbitrary (not-necessarily-separable) complex Hilbert space, $\cM \subseteq B(H)$ be a von Neumann algebra, and $a$ be a self-adjoint operator on $H$ affiliated with $\cM$.
If $(\cI,\|\cdot\|_{\cI})$ is an integral symmetrically normed ideal of $\cM$ and $f \in PB^1(\R) \cap PB^k(\R)$, then the operator function $\cI_{\sa} \ni b \mapsto f_a(b)\coloneqq f(a+b)-f(a) \in \cI$ is $k$-times continuously Fr\'{e}chet differentiable (Definition \ref{def.frechder}) in the $\cI$-norm $\|\cdot\|_{\cI}$, and
\[
\partial_{b_1}\cdots\partial_{b_k}f_a(0) = \sum_{\pi \in S_k}\underbrace{\int_{\sigma(a)}\cdots\int_{\sigma(a)}}_{k+1 \; \mathrm{times}} f^{[k]}(\lambda_1,\ldots,\lambda_{k+1})\,P^a(d\lambda_1)\,b_{\pi(1)}\cdots P^a(d\lambda_k)\,b_{\pi(k)}\,P^a(d\lambda_{k+1}),
\]
for all $b_1,\ldots,b_k \in \cI_{\sa}$.
\end{thm}
\begin{proof}
Combine Theorem \ref{thm.deropfunc} and Corollary \ref{cor.PBkinOCk}.
\end{proof}

Theorems \ref{thm.mainopder} and \ref{thm.mainISNI}\ref{item.Sp} generalize the best known results, from \cite{lemerdyskripka}, on differentiability of operator functions in the ideal $(\cI,\|\cdot\|_{\cI}) = (\cS_p(H),\|\cdot\|_{\cS_p})$ to the non-separable case when $p=1$.
We do not, however, fully recover the optimal regularity on $f$, established in \cite{lemerdymcdonald}, when $p \in (1,\infty)$.
Also, to the author's knowledge, the present paper's result on the ideal of compact operators (i.e., Theorems \ref{thm.mainopder} and \ref{thm.mainISNI}\ref{item.compact}) is new even when $H$ is separable.
Finally, as promised at the end of the previous section, Theorems \ref{thm.mainopder} and \ref{thm.mainISNI}\ref{item.fs} (together with Fact \ref{fact.Etau}) make substantial progress on the open problem (Problem 5.3.22 in \cite{skripka}) of finding general conditions for higher Fr\'{e}chet differentiability of operator functions in ideals of semifinite von Neumann algebras induced by (fully) symmetric Banach function spaces.

\section{Preliminaries}\label{sec.bg}

For Section \ref{sec.bg}, fix a complex Hilbert space $(H,\la \cdot, \cdot \ra)$ and a von Neumann algebra $\cM \subseteq B(H)$.
Recall also that $S' \coloneqq \{b \in B(H) : ab=ba$, for all $a \in S\}$ is the \textbf{commutant} of a set $S \subseteq B(H)$ and that a unital $\ast$-subalgebra $\cN \subseteq B(H)$ is a von Neumann algebra if and only if $\cN = \cN'' \coloneqq (\cN')'$.
This is the well-known (von Neumann) Bicommutant Theorem.

\subsection{\texorpdfstring{Weak$^*$}{} integrals in von Neumann algebras}\label{sec.wstarint}

Following parts of Section 3.3 of \cite{nikitopoulosMOI}, we review some basics of operator-valued integrals.
We shall write
\[
\ell^2(\N;H) \coloneqq \Bigg\{(h_n)_{n \in \N} \in H^{\N} : \sum_{n=1}^{\infty} \|h_n\|^2 < \infty\Bigg\} \; \text{ and } \; \la (h_n)_{n \in \N},(k_n)_{n \in \N} \ra_{\ell^2(\N;H)} \coloneqq \sum_{n=1}^{\infty}\la h_n,k_n \ra.
\]
Also, for the duration of this section, fix a measure space $(\Sigma,\sH,\rho)$.

\begin{defi}[Weak$^*$ measurability and integrability]\label{def.wstarmeasint}
A map $F \colon \Sigma \to \cM$ is called \textbf{weak$\boldsymbol{^*}$ measurable} if $\la F(\cdot)h,k \ra \colon \Sigma \to \C$ is $(\sH,\cB_{\C})$-measurable, for all $h,k \in H$.
Now, suppose in addition that
\[
\int_{\Sigma} |\la (F(\sigma)h_n)_{n \in \N},(k_n)_{n \in \N} \ra_{\ell^2(\N;H)}|\,\rho(d\sigma) < \infty, \numberthis\label{eq.wstartintegcond}
\]
for all $(h_n)_{n \in \N},(k_n)_{n \in \N} \in \ell^2(\N;H)$.
We say that $F$ is \textbf{weak$\boldsymbol{^*}$ integrable} if for all $S \in \sH$, there exists (necessarily unique) $I_S \in \cM$ such that
\[
\la (I_Sh_n)_{n \in \N},(k_n)_{n \in \N} \ra_{\ell^2(\N;H)} = \int_S \la (F(\sigma)h_n)_{n \in \N},(k_n)_{n \in \N} \ra_{\ell^2(\N;H)}\,\rho(d\sigma),
\]
for all $(h_n)_{n \in \N},(k_n)_{n \in \N} \in \ell^2(\N;H)$.
In this case, we call $I_S$ the \textbf{weak$\boldsymbol{^*}$ integral} (over $S$) of $F$ with respect to $\rho$, and we write $\int_S F\,d\rho = \int_S F(\sigma)\,\rho(d\sigma) \coloneqq I_S$.
\end{defi}
\begin{rem}
Let $\cM_* \coloneqq \{\sigma$-WOT-continuous linear functionals $\cM \to \C\}$ be the predual of $\cM$.
Part of Theorem 3.3.6 in \cite{nikitopoulosMOI} says that $F$ is weak$^*$ measurable (respectively, integrable) in the sense of Definition \ref{def.wstarmeasint} if and only if $F$ is weakly measurable (respectively, integrable) in the weak$^*$ topology on $\cM$ induced by the usual identification $\cM \cong \cM_*^{\;*}$.
The above terminology is therefore justified.
\end{rem}

It turns out (Theorem 3.3.6(iii) in \cite{nikitopoulosMOI}) that \eqref{eq.wstartintegcond} is actually already enough to guarantee that $F$ is weak$^*$ integrable.
Since we do not need this level of generality, we shall prove a weaker statement from scratch.

\begin{defi}[Upper and lower integrals]\label{def.nonmeasint}
For a function $h \colon \Sigma \to [0,\infty]$ that is not necessarily measurable, we define
\begin{align*}
    \overline{\int_{\Sigma}} h(\sigma) \, \rho(d\sigma) & = \overline{\int_{\Sigma}} h \, d\rho \coloneqq \inf\Bigg\{\int_{\Sigma} \tilde{h} \, d\rho : h \leq \tilde{h} \;\; \rho\text{-almost everywhere}, \; \tilde{h} \colon \Sigma \to [0,\infty] \text{ measurable}\Bigg\} \text{ and} \\
    \underline{\int_{\Sigma}} h(\sigma) \, \rho(d\sigma) & = \underline{\int_{\Sigma}} h \, d\rho \coloneqq \sup\Bigg\{\int_{\Sigma} \tilde{h} \, d\rho : \tilde{h} \leq h \; \rho\text{-almost everywhere}, \; \tilde{h} \colon \Sigma \to [0,\infty] \text{ measurable}\Bigg\}
\end{align*}
to be, respectively, the \textbf{upper} and \textbf{lower integral} of $h$ with respect to $\rho$.
(Please see Section 3.1 of \cite{nikitopoulosMOI}.)
\end{defi}

\begin{prop}[Existence of weak$^*$ integrals]\label{prop.wstarintexist}
If $F \colon \Sigma \to \cM$ is weak$^*$ measurable and $\underline{\int_{\Sigma}} \|F\| \,d\rho < \infty$, then $F$ is weak$^*$ integrable and
\[
\Bigg\|\int_{\Sigma} F\,d\rho \Bigg\| \leq \underline{\int_{\Sigma}}\|F\|\,d\rho. \numberthis\label{eq.opnormtriangle}
\]
This is called the (\textbf{operator norm}) \textbf{integral triangle inequality}.
\end{prop}
\begin{rem}
First, the operator norm of a weak$^*$ measurable map is not necessarily measurable when $H$ is not separable.
This is why we need the lower integral above.
Second, one can also prove that weak$^*$ integrals are independent of the representation of $\cM$.
Please see Theorem 3.3.6(iv) in \cite{nikitopoulosMOI}.
\end{rem}
\pagebreak
\begin{proof}
Fix $S \in \sH$, and define $B_S \colon H \times H \to \C$ by $(h,k) \mapsto \int_S \la F(\sigma)h,k \ra \,\rho(d\sigma)$.
Then $B_S$ is linear in the first argument and conjugate-linear in the second argument.
Also,
\[
|B_S(h,k)| \leq \int_S |\la F(\sigma)h,k \ra| \,\rho(d\sigma) \leq \int_{\Sigma}|\la F(\sigma)h,k \ra| \,\rho(d\sigma) \leq \underline{\int_{\Sigma}} \|F(\sigma)h\|\,\|k\|\,\rho(d\sigma) \leq  \|h\|\,\|k\|\underline{\int_{\Sigma}} \|F\|\,d\rho,
\]
for all $h,k \in H$.
In particular, $B_S$ is bounded.
By the Riesz Representation Theorem, there exists unique $\rho_S(F) \in B(H)$ such that $\la \rho_S(F)h,k \ra = B_S(h,k)$, for all $h,k \in H$.
Moreover,
\[
\|\rho_S(F)\| = \sup\{|B_S(h,k)| : h,k \in H, \; \|h\|,\|k\| \leq 1\} \leq \underline{\int_{\Sigma}}\|F\|\,d\rho.
\]
If we can show $\rho_S(F) = \int_S F \,d\rho$, then we are done.

To this end, let $a \in \cM'$ and $h,k \in H$.
If $\sigma \in \Sigma$, then $\la F(\sigma)\,ah,k \ra = \la a\,F(\sigma)h,k \ra = \la F(\sigma)h,a^*k \ra$.
Thus $B_S(ah,k) = B_S(h,a^*k)$.
But then $\la a\,\rho_S(F)h,k \ra = \la \rho_S(F)h,a^*k \ra = B_S(h,a^*k) = B_S(ah,k) = \la \rho_S(F)\,ah,k \ra$.
Since $h,k \in H$ were arbitrary, $a\,\rho_S(F) = \rho_S(F)\,a$.
Since $a \in \cM'$ was arbitrary, $\rho_S(F) \in \cM'' = \cM$ by the Bicommutant Theorem.
Next, let $(h_n)_{n \in \N},(k_n)_{n \in \N} \in \ell^2(\N;H)$.
Then
\[
\int_{\Sigma} \sum_{n=1}^{\infty}|\la F(\sigma)h_n,k_n\ra|\,\rho(d\sigma) \leq \|(h_n)_{n \in \N}\|_{\ell^2(\N;H)}\|(k_n)_{n \in \N}\|_{\ell^2(\N;H)}\underline{\int_{\Sigma}}\|F\|\,d\rho < \infty
\]
by the Cauchy--Schwarz Inequality (twice).
Therefore, by the Dominated Convergence Theorem,
\[
\int_S \sum_{n=1}^{\infty}\la F(\sigma)h_n,k_n \ra\,\rho(d\sigma) = \sum_{n=1}^{\infty}\int_S \la F(\sigma)h_n,k_n \ra\,\rho(d\sigma) = \sum_{n=1}^{\infty}\la \rho_S(F)h_n,k_n \ra,
\]
as desired.
\end{proof}

Other than basic algebraic properties of the weak$^*$ integral, which are usually clear from the definition, the most important fact about weak$^*$ integrals that we shall use is an operator-valued Dominated Convergence Theorem, which we prove from scratch for the convenience of the reader.
(But we remark that it follows from Proposition 3.2.3(v) and Theorem 3.3.6 in \cite{nikitopoulosMOI}.)
First, we note that one can do better than the operator norm triangle inequality.
Indeed, retaining the notation from the proof above, we have
\[
\Bigg\|\Bigg(\int_{\Sigma} F\,d\rho\Bigg)h\Bigg\| = \sup\{|B_{\Sigma}(h,k)| : \|k\| \leq 1\}  \leq \sup\Bigg\{\int_{\Sigma}|\la F(\sigma)h,k \ra| \,\rho(d\sigma) : \|k\| \leq 1\Bigg\} \leq \underline{\int_{\Sigma}} \|F(\sigma)h\|\,\rho(d\sigma),
\]
for all $h \in H$.

\begin{lem}[Nonmeasurable Dominated Convergence Theorem]\label{lem.poorDCT}
Suppose $(g_n)_{n \in \N}$ is a sequence of functions $\Sigma \to [0,\infty]$ such that $g_n \to 0$ pointwise $\rho$-almost everywhere as $n \to \infty$.
If
\[
\overline{\int_{\Sigma}}\sup_{n \in \N} g_n \, d\rho < \infty,
\]
then $\underline{\int_{\Sigma}}g_n \, d\rho \to 0$ as $n \to \infty$.
\end{lem}
\begin{proof}
By definition of the upper integral, there is some measurable $g \colon \Sigma \to [0,\infty]$ such that $\int_{\Sigma} g \, d\rho < \infty$ and $\sup_{n \in \N}g_n \leq g$ $\rho$-almost everywhere.
Now, if $n \in \N$, then, by definition of the lower integral, there exists a measurable $\tilde{g}_n \colon \Sigma \to [0,\infty]$ such that $0 \leq \tilde{g}_n \leq g_n$ $\rho$-almost everywhere and
\[
\underline{\int_{\Sigma}}g_n \, d\rho - \frac{1}{n} < \int_{\Sigma} \tilde{g}_n \, d\rho.
\]
Since $g_n \to 0$ $\rho$-almost everywhere, and $0 \leq \tilde{g}_n \leq g_n$ $\rho$-almost everywhere, we also have $\tilde{g}_n \to 0$ $\rho$-almost everywhere as $n \to \infty$.
Also, $\tilde{g}_n \leq g_n \leq g$ $\rho$-almost everywhere, for all $n \in \N$.
Therefore, by the Dominated Convergence Theorem,
\[
0 \leq \liminf_{n \to \infty}\underline{\int_{\Sigma}}g_n \, d\rho \leq \limsup_{n \to \infty}\underline{\int_{\Sigma}}g_n \, d\rho = \limsup_{n \to \infty}\Bigg(\underline{\int_{\Sigma}}g_n \, d\rho - \frac{1}{n}\Bigg) \leq \limsup_{n \to \infty}\int_{\Sigma} \tilde{g}_n \, d\rho = 0,
\]
as desired.
\end{proof}

\begin{prop}[Operator-valued Dominated Convergence Theorem]\label{prop.opDCT}
Let $(F_n)_{n \in \N}$ be a sequence of weak$^*$ integrable maps $\Sigma \to \cM$, and suppose $F \colon \Sigma \to \cM$ is such that $F_n \to F$ pointwise in the weak, strong, or strong$^*$ operator topology as $n \to \infty$.
If
\[
\overline{\int_{\Sigma}} \sup_{n \in \N}\|F_n\|\,d\rho < \infty, \numberthis\label{eq.opnormbd}
\]
then $F \colon \Sigma \to \cM$ is weak$^*$ integrable and $\int_{\Sigma} F_n\,d\rho \to \int_{\Sigma} F\, d\rho$ in, respectively, the weak, strong, or strong$^*$ operator topology as $n \to \infty$.
\end{prop}
\begin{proof}
Fix $h,k \in H$.
In all cases, $F_n \to F$ pointwise in the WOT as $n \to \infty$. Therefore, $\la F(\cdot)h,k \ra$ is the pointwise limit of $(\la F_n(\cdot)h,k \ra)_{n \in \N}$.
Consequently, $F$ is weak$^*$ measurable.
Also, $\|F\| \leq \sup_{n \in \N} \|F_n\|$, so $F$ is weak$^*$ integrable by \eqref{eq.opnormbd} and Proposition \ref{prop.wstarintexist}.
Now, \eqref{eq.opnormbd} also gives
\[
\int_{\Sigma}\sup_{n \in \N}|\la F_n(\sigma)h,k \ra|\,\rho(d\sigma) \leq \|h\|\,\|k\|\underline{\int_{\Sigma}}\sup_{n \in \N}\|F_n\|\,d\rho < \infty.
\]
Therefore, by the Dominated Convergence Theorem,
\[
\Bigg\la \Bigg(\int_{\Sigma} F_n\,d\rho\Bigg)h,k \Bigg\ra = \int_{\Sigma} \la F_n(\sigma)h,k \ra\,\rho(d\sigma) \to \int_{\Sigma} \la F(\sigma)h,k \ra\,\rho(d\sigma) = \Bigg\la \Bigg(\int_{\Sigma} F\,d\rho\Bigg)h,k \Bigg\ra
\]
as $n \to \infty$.
Thus $\int_{\Sigma}F_n\,d\rho \to \int_{\Sigma} F\,d\rho$ in the WOT as $n \to \infty$.
Assume now that $F_n \to F$ pointwise in the strong operator topology (SOT) as $n \to \infty$, and write $T_n \coloneqq \int_{\Sigma} F_n\,d\rho$ and $T \coloneqq \int_{\Sigma} F \,d\rho$.
Then
\[
\|T_nh-Th\| = \Bigg\|\Bigg(\int_{\Sigma}(F_n-F)\,d\rho\Bigg)h\Bigg\| \leq \underline{\int_{\Sigma}} \|(F_n(\sigma)-F(\sigma))h\|\,\rho(d\sigma) \to 0
\]
as $n \to \infty$, by the integral triangle inequality observed above and Lemma \ref{lem.poorDCT}, which applies because of \eqref{eq.opnormbd} and the fact that $\sup_{n \in \N} \|(F_n-F)h\| \leq 2\|h\|\sup_{n \in \N}\|F_n\|$.
Finally, the S$^*$OT case follows from the SOT case because $(F_n^*)_{n \in \N}$ and $F^*$ satisfy the same hypotheses as $(F_n)_{n \in \N}$ and $F$, and the adjoint is easily seen to commute with the weak$^*$ integral.
\end{proof}

We end this section by stating an additional property of weak$^*$ integrals when $\cM$ is \textit{semifinite}.
Before stating the result, we recall some notation and terminology.
For $a,b \in B(H)$, we write $a \leq b$ or $b \geq a$ to mean $\la (b-a)h,h \ra \geq 0$, for all $h \in H$. 
Also, we write
\[
\cM_+ \coloneqq \{a \in \cM : a \geq 0\}.
\]
It is easy to see that $\cM_+$ is closed in the WOT.
Also, if $(a_j)_{j \in J}$ is a net in $\cM_+$ that is bounded (there exists $b \in B(H)$ such that $a_j \leq b$, for all $j \in J$) and increasing ($j_1 \leq j_2 \Rightarrow  a_{j_1} \leq a_{j_2}$), then $\sup_{j \in J} a_j$ exists in $B(H)_+$ and belongs to $\cM_+$.
(Please see Proposition 43.1 in \cite{conwayop}.)
This is often known as Vigier's Theorem.

\begin{defi}[Trace]\label{def.trace}
A function $\tau \colon \cM_+ \to [0,\infty]$ is called \textbf{trace} on $\cM$ if
\begin{enumerate}[label=(\alph*),leftmargin=2\parindent]
    \item $\tau(a+b) = \tau(a)+\tau(b)$,\label{item.add}
    \item $\tau(\lambda a) = \lambda\,\tau(a)$,\label{item.poshom} and
    \item $\tau(c^*c) = \tau(cc^*)$,\label{item.trace}
\end{enumerate}
for all $a,b \in \cM_+$, $c \in \cM$, and $\lambda \in \R_+$.
A trace $\tau$ on $\cM$ is called
\begin{enumerate}[label=(\alph*),leftmargin=2\parindent]
    \item \textbf{normal} if $\tau(\sup_{j \in J}a_j)  = \sup_{j \in J}\tau(a_j)$ whenever $(a_j)_{j \in J}$ is a bounded and increasing net in $\cM_+$,
    \item \textbf{faithful} if $a \in \cM_+$ and $\tau(a) = 0$ imply $a=0$, and
    \item \textbf{semifinite} if $\tau(a) = \sup\{\tau(b) : a \geq b \in \cM_+, \, \tau(b) < \infty\}$, for all $a \in \cM_+$.
\end{enumerate}
If $\tau$ is a normal, faithful, semifinite trace on $\cM$, then $(\cM,\tau)$ is called a \textbf{semifinite von Neumann algebra}.
\end{defi}
\begin{rem}
In the presence of \ref{item.add} and \ref{item.poshom}, condition \ref{item.trace} is equivalent to $\tau(u^*au) = \tau(a)$ for all $a \in \cM_+$ and all unitaries $u$ belonging to $\cM$.
This is Corollary 1 in Section I.6.1 of \cite{dixmier}.
\end{rem}
\pagebreak

For basic properties of traces on von Neumann algebras, please see Chapter I.6 of \cite{dixmier} or Section V.2 of \cite{takesaki}.
Motivating examples of semifinite von Neumann algebras are $(B(H),\Tr)$ and $(L^{\infty}(\Om,\mu),\int_{\Om} \boldsymbol{\cdot}\,d\mu)$, where $(\Om,\sF,\mu)$ is a $\sigma$-finite measure space and $L^{\infty}(\Om,\mu)$ is represented as multiplication operators on $L^2(\Om,\mu)$.

\begin{nota}\label{nota.ncLp}
Let $(\cM,\tau)$ be a semifinite von Neumann algebra and $1 \leq p < \infty$.
Write\vspace{-0.5mm}
\[
\|a\|_{L^p(\tau)} \coloneqq \tau(|a|^p)^{\frac{1}{p}} \in [0,\infty],\vspace{-0.5mm}
\]
for all $a \in \cM$, and $\mathcal{L}^p(\tau) \coloneqq \{a \in \cM : \|a\|_{L^p(\tau)}^p = \tau(|a|^p) < \infty\}$.
Also, we take $(\mathcal{L}^{\infty}(\tau),\|\cdot\|_{L^{\infty}(\tau)}) \coloneqq (\cM, \|\cdot\|)$.
\end{nota}

It turns out that if $1 \leq p \leq \infty$, then $\|\cdot\|_{L^p(\tau)}$ is a norm on $\mathcal{L}^p(\tau)$.
(Please see \cite{dixmierLp} or \cite{dasilva}.)
The completion $(L^p(\tau),\|\cdot\|_{L^p(\tau)})$ of $(\mathcal{L}^p(\tau),\|\cdot\|_{L^p(\tau)})$ is called the \textbf{noncommutative $\boldsymbol{L^p}$ space associated to $\boldsymbol{(\cM,\tau)}$}.
We shall see another perspective on $L^p(\tau)$ in Section \ref{sec.idealex2}.

\begin{thm}[Noncommutative Minkowski Inequality for Integrals \cite{nikitopoulosMOI}]\label{thm.Lpinteg}
Let $(\cM,\tau)$ be a semifinite von Neumann algebra.
If $F \colon \Sigma \to \cM$ is weak$^*$ integrable, then\vspace{-0.5mm}
\[
\Bigg\|\int_{\Sigma} F \,d\rho \Bigg\|_{L^p(\tau)} \leq \underline{\int_{\Sigma}}\|F\|_{L^p(\tau)}\,d\rho,\vspace{-0.5mm}
\]
for all $p \in [1,\infty]$.
In particular, if the right hand side is finite, then $\int_{\Sigma} F \,d\rho \in \mathcal{L}^p(\tau)$.
\end{thm}

This is Theorem 3.4.7 in \cite{nikitopoulosMOI}, and the motivation for its name is the classical Minkowski Inequality for Integrals (e.g., 6.19 in \cite{folland}).
In view of the name of \eqref{eq.opnormtriangle}, an equally sensible name for Theorem \ref{thm.Lpinteg} is the ``noncommutative $L^p$-norm integral triangle inequality."

\subsection{Unbounded operators and spectral theory}\label{sec.SpecTh}

In this section, we provide the information about unbounded operators, projection-valued measures, and the Spectral Theorem that is necessary for this paper.
Please see Chapters 3--6 of \cite{birmansolomyak5} or Chapters IX and X of \cite{conwayfunc} for more information and proofs of the facts we state (without proof) in this section.

An (\textbf{unbounded linear}) \textbf{operator} $A$ on $H$ is a linear subspace $\dom(A) \subseteq H$ (the \textbf{domain} of $A$) and a linear map $A \colon \dom(A) \to H$. We may identify $A$ with its graph $\Gamma(A) = \{(h,Ah) : h \in \dom(A)\} \subseteq H \times H$.
We say $A$ is \textbf{densely defined} if $\dom(A) \subseteq H$ is dense, \textbf{closable} if the closure of $\Gamma(A)$ in $H \times H$ is the graph of some operator $\overline{A}$ (called the \textbf{closure} of $A$) on $H$, and \textbf{closed} if $\Gamma(A) \subseteq H \times H$ is closed (i.e., $\overline{A} = A$).
If $B$ is another unbounded operator on $H$, then the sum $A+B$ has domain $\dom(A+B) \coloneqq \dom(A) \cap \dom(B)$, and the product $AB$ has domain $\dom(AB) \coloneqq \{h \in \dom(B) : B h \in \dom(A)\} = B^{-1}(\dom(A))$.
In addition, we write $A \subseteq B$ if $\Gamma(A) \subseteq \Gamma(B)$, i.e., $\dom(A) \subseteq \dom(B)$ and $Ah=Bh$, for all $h \in \dom(A)$;
and $A=B$ if $A \subseteq B$ and $B \subseteq A$, i.e., $\dom(A) = \dom(B)$ and $Ah=Bh$ for $h$ in this common domain. 

Given a densely defined operator $A$ on $H$, we may form the \textbf{adjoint} $A^*$ of $A$ as follows.
First, let\vspace{-0.5mm}
\[
\dom(A^*) \coloneqq \{k \in H : \dom(A) \ni h \mapsto \la Ah,k \ra \in \C \text{ is bounded}\}.\vspace{-0.5mm}
\]
Then, for $k \in \dom(A^*)$, let $A^*k \in H$ be the unique vector in $H$ such that $\la Ah,k \ra = \la h,A^*k \ra$, for all $h \in \dom(A)$.
One can show that if $A$ is densely defined and closed, then so is $A^*$, and $A=(A^*)^*$.
A closed, densely defined operator $A$ on $H$ is \textbf{normal} if $A^*A=AA^*$ and \textbf{self-adjoint} if $A^*=A$.
Finally, a self-adjoint operator $A$ on $H$ is called \textbf{positive}, written $A \geq 0$, if $\la Ah,h \ra \geq 0$ whenever $h \in \dom(A)$.

\begin{nota}
Write $C(H)$ for the set of closed, densely defined linear operators on $H$.
Also, write\vspace{-0.5mm}
\[
C(H)_{\nu} \coloneqq \{A \in C(H) : A^*A=AA^*\}, \; C(H)_{\sa} \coloneqq \{A \in C(H) : A=A^*\}, \; C(H)_+ \coloneqq \{A \in C(H) : A \geq 0\}\vspace{-0.5mm}
\]
for the set of normal, self-adjoint, and positive operators on $H$, respectively.
\end{nota}

Next, we recall basic definitions and facts about integration with respect to a projection-valued measure.

\begin{defi}[Projection-valued measure]
Let $(\Om,\sF)$ be a measurable space and $P \colon \sF \to B(H)$.
We call $P$ a \textbf{projection-valued measure} if $P(\Om) = \id_H$, $P(G)^2 = P(G) = P(G)^*$ whenever $G \in \sF$, and\vspace{-0.5mm}
\[
P\Bigg(\bigcup_{n \in \N}G_n\Bigg) = \text{WOT-}\sum_{n=1}^{\infty}P(G_n) \numberthis\label{eq.WOTcountadd}\vspace{-0.5mm}
\]
whenever $(G_n)_{n \in \N} \in \sF^{\N}$ is a sequence of disjoint measurable sets.
In this case, we call the quadruple $(\Om,\sF,H,P)$ a \textbf{projection-valued measure space}.
\end{defi}
\pagebreak
\begin{rem}
To be clear, it is implicitly required in \eqref{eq.WOTcountadd} that the series on the right hand side converges in the WOT.
It actually follows from the above definition that $P(\emptyset) = 0$ and $P(G_1 \cap G_2) = P(G_1)\,P(G_2)$ whenever $G_1,G_2 \in \sF$.
(Please see Theorem 1 in Section 5.1.1 of \cite{birmansolomyak5}.)
These are often added to the definition of a projection-valued measure because they guarantee that the series in \eqref{eq.WOTcountadd} converges in the SOT.
\end{rem}

\begin{nota}
If $\varphi \colon \Om \to \C$ is a function, then $\|\varphi\|_{\ell^{\infty}(\Om)} \coloneqq \sup_{\om \in \Om} |\varphi(\om)| \in [0,\infty]$.
Also, we write $\ell^0(\Om,\sF) \coloneqq \{(\sF,\cB_{\C})$-measurable functions $\Om \to \C\}$ and $\ell^{\infty}(\Om,\sF) \coloneqq \{\varphi \in \ell^0(\Om,\sF) : \|\varphi\|_{\ell^{\infty}(\Om)} < \infty\}$.
Finally, if $\mu$ is a complex measure on $(\Om,\sF)$, then we write $|\mu|$ for the total variation measure of $\mu$ and $\|\mu\|_{\text{var}} \coloneqq |\mu|(\Om)$ for the total variation norm of $\mu$.
\end{nota}

\begin{prop}[Integration with respect to a projection-valued measure]\label{prop.integP}
$\,$Let $(\Om,\,\sF,\,H,\,P)$ be a projection-valued measure space and $\varphi,\psi \in \ell^0(\Om,\sF)$.
\begin{enumerate}[label=(\roman*),font=\normalfont,leftmargin=2\parindent]
    \item Fix $h,k \in H$.
    If $P_{h,k}(G) \coloneqq \la P(G)h,k\ra$, for all $G \in \sF$, then $P_{h,k}$ is a complex measure such that $\|P_{h,k}\|_{\mathrm{var}} \leq \|h\|\,\|k\|$.
    Also, $P_{h,k} \ll P$ in the sense that if $P(G) = 0$, then $P_{h,k}(G_0) = 0$ whenever $\sF \ni G_0 \subseteq G$, i.e., $|P_{h,k}|(G) = 0$.
    Finally, $P_{h,h}$ is a (finite) positive measure.\label{item.Phk}
    \item Let $\dom(P(\varphi)) \coloneqq \{h \in H : \int_{\Om}|\varphi|^2\,dP_{h,h} < \infty\}$ and $h \in \dom(P(\varphi))$.
    If $k \in H$, then $\varphi \in L^1(\Om,|P_{h,k}|)$, and there exists a unique $h_{\varphi} \in H$ such that $\la h_{\varphi},k \ra_H = \int_{\Om}\varphi \,dP_{h,k}$, for all $k \in K$.
    If we define $P(\varphi)h \coloneqq h_{\varphi}$, then $P(\varphi) \in C(H)_{\nu}$.\label{item.defofPint}
    \item We have $P(\varphi)^* = P(\overline{\varphi})$, $\dom(P(\varphi)P(\psi)) = \dom(P(\psi)) \cap \dom(P(\varphi\psi))$, and $P(\varphi)P(\psi) \subseteq P(\varphi\psi)$.
    In particular, $P(\varphi)^*P(\varphi) = P(|\varphi|^2)$, and if $\psi \in \ell^{\infty}(\Om,\sF)$, then $P(\varphi)P(\psi) = P(\varphi\psi)$.\label{item.algpropPint}
    \item The map $\ell^{\infty}(\Om,\sF) \ni \varphi \mapsto P(\varphi) \in B(H)$ is a (contractive) $\ast$-homomorphism.\label{item.bddPint}
    \item Let $(\varphi_n)_{n \in \N} \in \ell^{\infty}(\Om,\sF)^{\N}$ be a sequence.
    If $\sup_{n \in \N}\|\varphi_n\|_{\ell^{\infty}(\Om)} < \infty$ and $\varphi_n \to \varphi \in \ell^{\infty}(\Om,\sF)$ pointwise (i.e., $\varphi_n \to \varphi$ \textbf{boundedly}), then $P(\varphi_n) \to P(\varphi)$ in the strong$^*$ operator topology as $n \to \infty$.\label{item.PintDCT}
\end{enumerate}
We shall often write $P(\varphi) = \int_{\Om} \varphi \, dP = \int_{\Om} \varphi(\om)\,P(d\om)$ and call this the \textbf{integral} of $\varphi$ with respect to $P$.
\end{prop}

The reason projection-valued measures are relevant for us is the Spectral Theorem, which we now recall.
If $A \in C(H)$, then the \textbf{resolvent set} $\rho(A) \subseteq \C$ of $A$ is the set of $\lambda \in \C$ such that $A-\lambda \id_H : \dom(A) \to H$ is a bijection with bounded inverse.
The \textbf{spectrum} $\sigma(A) \subseteq \C$ of $A$ is the complement of $\rho(A)$.
The resolvent set $\rho(A)$ is open in $\C$, and the spectrum $\sigma(A)$ is closed in $\C$. Also, a normal operator $A \in C(H)_{\nu}$ is self-adjoint if and only if $\sigma(A) \subseteq \R$. Finally, $A \in C(H)_+$ if and only if $A \in C(H)_{\sa}$ and $\sigma(A) \subseteq [0,\infty)$.

\begin{thm}[Spectral Theorem for normal operators]\label{thm.spec}
If $A \in C(H)_{\nu}$, then there exists a unique projection-valued measure $P^A \colon \cB_{\sigma(A)} \to B(H)$ such that $A = \int_{\sigma(A)} \lambda\,P^A(d\lambda)$.
We call $P^A$ the \textbf{projection-valued spectral measure} of $A$.
We shall frequently abuse notation and consider $P^A$ to be a projection-valued measure defined on $\cB_{\C}$ or, when $A \in C(H)_{\sa}$, on $\cB_{\R}$.
\end{thm}

The Spectral Theorem leads to the usual definition of functional calculus.
If $A \in C(H)_{\nu}$ and $f \colon \sigma(A) \to \C$ is Borel measurable, then we define
\[
f(A) \coloneqq P^A(f) = \int_{\sigma(A)}f\,dP^A \in C(H)_{\nu}.
\]
This definition enjoys the property that if $f,g \colon \C \to \C$ are Borel measurable, then $f(g(A)) = (f \circ g)(A)$.
(Please see Corollary 5.6.29 of \cite{kadisonringrose1}.)

Now, if $A \in C(H)$ is arbitrary, then $A^*A \in C(H)_+$.
(This result is often called von Neumann's Theorem.)
In particular, $\sigma(A^*A) \subseteq [0,\infty)$, so we may define the \textbf{absolute value} $|A| \coloneqq (A^*A)^{\frac{1}{2}} \in C(H)_+$ of $A$ via functional calculus.
Also, there exists a unique partial isometry $U \in B(H)$ with initial space $\overline{\im |A|} = \overline{\im (A^*)}$ and final space $\overline{\im A}$ such that $A=U|A|$.
(In particular, $\dom (A) = \dom (|A|)$.)
This is called the \textbf{polar decomposition} of $A$.
(Please see Section 8.1, particularly Theorems 2 and 3, of \cite{birmansolomyak5}.)

We end this section with a review of the concept of an operator affiliated with a von Neumann algebra.

\begin{defi}[Affiliated operators]\label{def.aff}
An operator $a \in C(H)$ is said to be \textbf{affiliated with $\boldsymbol{\cM}$} if $u^*au = a$, for all unitaries $u$ belonging to $\cM'$.
In this case, we write $a \aff \cM$.
If in addition $a$ is normal (respectively, self-adjoint), then we write $a \aff \cM_{\nu}$ (respectively, $a \aff \cM_{\sa}$).
\end{defi}

Here are some properties of affiliated operators.
\pagebreak

\begin{prop}\label{prop.aff}
Let $(\Om,\sF,H,P)$ be a projection-valued measure space.
\begin{enumerate}[label=(\roman*),font=\normalfont,leftmargin=2\parindent]
    \item If $a \in B(H)$, then $a \aff \cM$ if and only if $a \in \cM$.\label{item.bddaffinM}
    \item If $P(G) \in \cM$, for all $G \in \sF$, then $P(\varphi) \aff \cM$, for all $\varphi \in \ell^0(\Om,\sF)$.
    In particular, by item \ref{item.bddaffinM}, if $\varphi \in \ell^{\infty}(\Om,\sF)$, then $P(\varphi) \in \cM$.\label{item.intPaff}
    \item If $a \in C(H)_{\nu}$, then $a \aff \cM$ if and only if $P^a(G) \in \cM$, for all $G \in \cB_{\sigma(a)}
    $;
    and in this case $f(a) \aff \cM$, for all $f \in \ell^0(\sigma(a),\cB_{\sigma(a)})$.
    In particular, $f(a) \in \cM$, for all $f \in \ell^{\infty}(\sigma(a),\cB_{\sigma(a)})$.\label{item.funcofaffisaff}
    \item If $a \in C(H)$ and $a=u|a|$ is its polar decomposition, then $a \aff \cM$ if and only if $u \in \cM$ and $P^{|a|}(G) \in \cM$, for all $G \in \cB_{\sigma(|a|)}$.\label{item.affandpolar}
\end{enumerate}
\end{prop}

We sketch the proofs for the reader's convenience.
As we shall see, the first three properties follow without much difficulty from the definitions, the Bicommutant Theorem, and the Spectral Theorem.
For the difficult part of item \ref{item.affandpolar}, please see also Lemma 4.4.1 in \cite{murrayvonneumann}.

\begin{proof}[Sketch of proof]
We take each item in turn.

\ref{item.bddaffinM} Let $a \in B(H)$ and $u \in \cM'$ be a unitary.
If $a \in \cM$, then of course $u^*au = u^*ua = a$.
Now, if $a \aff \cM$, then $au = uu^*au=ua$.
Since all $C^*$-algebras are spanned by their unitaries, we conclude that $ab=ba$, for all $b \in \cM'$.
Thus $a \in \cM'' = \cM$ by the Bicommutant Theorem.

\ref{item.intPaff} Suppose that $P(G) \in \cM$, for all $G \in \sF$.
If $h,k \in H$ and $u \in \cM'$ is a unitary, then it is easy to see that $P_{uh,uk} = P_{h,k}$.
Unraveling the definition of $P(\varphi)$ then gives $u^*P(\varphi)u = P(\varphi)$.
Thus $P(\varphi) \aff \cM$.
It is worth mentioning that one can prove much more directly --- without knowing anything about unbounded operators or the Bicommutant Theorem --- that if $P(G) \in \cM$, for all $G \in \sF$, then $P(\varphi) \in \cM$, for all $\varphi \in \ell^{\infty}(\Om,\sF)$.
Please see Lemma 4.2.16 in \cite{nikitopoulosMOI}.

\ref{item.funcofaffisaff} If $P^a(G) \in \cM$, for all $G \in \cB_{\sigma(a)}$, then $a = \int_{\sigma(a)} \lambda \, P^a(d\lambda) \aff \cM$ by the previous item.
Now, suppose that $a \aff \cM_{\nu}$, and let $u \in \cM'$ be a unitary.
Note that $Q^a \coloneqq u^*P^a(\cdot)u \colon \cB_{\sigma(a)} \to B(H)$ is a projection-valued measure, and it is easy to see from the Spectral Theorem and definition of $Q^a$ that $u^*au = \int_{\sigma(a)}\lambda \, Q^a(d\lambda)$.
But $u^*au = a$ by assumption, so the uniqueness part of the Spectral Theorem forces $P^a = Q^a = u^*P^a(\cdot)u$.
In other words, $P^a(G) \aff \cM$ and thus, by item \ref{item.bddaffinM}, $P^a(G) \in \cM$, for all $G \in \cB_{\sigma(a)}$.

\ref{item.affandpolar} Let $a \in C(H)$, $a=u|a|$ be the polar decomposition of $a$, and $v \in \cM'$ be a unitary.
If $P^{|a|}(G) \in \cM$ whenever $G \in \cB_{\sigma(|a|)}$, then $|a| \aff \cM$ by the previous item.
If in addition $u \in \cM$, then we have that $v^*av = v^*u|a|v = v^*uvv^*|a|v = u|a| = a$.
Thus $a \aff \cM$.
Conversely, if $a \aff \cM$, then $|a| = |v^*av| = v^*|a|v$.
Thus $|a| \aff \cM$, and, by the previous item, $P^{|a|}(G) \in \cM$ whenever $G \in \cB_{\sigma(|a|)}$.
Next, notice that $v^*uv$ is a partial isometry, and $a = v^*av = v^*u|a|v = v^*uv|a|$ by what just proved.
Finally, $|a| \aff \cM$ implies that $v^*uv$ has initial space $\overline{\im|a|}$, and $a \aff \cM$ implies that $v^*uv$ has final space $\overline{\im a}$.
We conclude that $v^*uv = u$ by the uniqueness of the polar decomposition.
Thus $u \aff \cM$ and so, by item \ref{item.bddaffinM}, $u \in \cM$.
\end{proof}

\subsection{Symmetric operator spaces}\label{sec.symmopsp}

In Section \ref{sec.idealex2}, we shall make use of the theory of \textit{symmetric operator spaces}.
In the present section, we review the notation, terminology, and results from this theory that are necessary for our purposes.
We refer the reader to \cite{doddsNKS} for extra exposition, examples, and a thorough list of references.
(The reader who is uninterested in Section \ref{sec.idealex2} may safely skip at this point to Section \ref{sec.idealprop}.)
For the duration of this section, suppose $(\cM,\tau)$ is a semifinite von Neumann algebra.

Write $\mathrm{Proj}(\cM) \coloneqq \{p \in \cM : p^2=p=p^*\}$ for the lattice of (orthogonal) projections in $\cM$.
An operator $a \aff \cM$ is called \textbf{$\boldsymbol{\tau}$-measurable} if there exists some $s \geq 0$ such that $\tau(P^{|a|}((s,\infty))) < \infty$.
Write
\[
S(\tau) \coloneqq \{a \aff \cM : a \text{ is }\tau\text{-measurable}\},
\]
and let $a,b \in S(\tau)$.
Then $a+b$ is closable, and $\overline{a+b} \in S(\tau)$;
$ab$ is closable, and $\overline{ab} \in S(\tau)$;
and $a^*,|a| \in S(\tau)$.
Moreover, $S(\tau)$ is a $\ast$-algebra under the adjoint, strong sum (closure of sum), and strong product (closure of product) operations;
we shall therefore omit the closures from strong sums and products in the future.
For the preceding facts (and more) about $\tau$-measurable operators, please see \cite{nelson,terp}.

Fix $a \in S(\tau)$.
For $s \geq 0$, define
\[
d_s(a) \coloneqq \tau\big(P^{|a|}((s,\infty))\big) \in [0,\infty].\pagebreak
\]
By definition of $\tau$-measurability, $d_s(a) < \infty$ for sufficiently large $s$.
The function $d(a) = d_{\boldsymbol{\cdot}}(a)$ is called the (\textbf{noncommutative}) \textbf{distribution function} of $a$.
Now, for $t > 0$, define
\[
\mu_t(a) \coloneqq \inf\{s \geq 0 : d_s(a) \leq t\} \in [0,\infty).
\]
The function $\mu(a) = \mu_{\boldsymbol{\cdot}}(a)$ is called the (\textbf{generalized}) \textbf{singular value function} or (\textbf{noncommutative}) \textbf{decreasing rearrangement} of $a$, and $\mu(a)$ is decreasing and right-continuous.
For properties of $d(a)$ and $\mu(a)$, please see \cite{fackkosaki}.
Now, let $S(\tau)_+ \coloneqq S(\tau) \cap C(H)_+$.
If $a \in \cM_+ = S(\tau)_+ \cap \cM$, then we have the identity
\[
\tau(a) = \int_0^{\infty} \mu_t(a)\,dt.
\]
We therefore extend $\tau$ to $S(\tau)_+$ via the formula above;
this extension is still notated $\tau \colon S(\tau)_+ \to [0,\infty]$.
Finally, if $a,b \in S(\tau)$, then we write
\[
a \pprec b \; \text{ if } \; \int_0^t \mu_s(a)\,ds \leq \int_0^t \mu_s(b)\,ds, \text{ for all } t \geq 0.
\]
In this case, we say that $a$ is \textbf{submajorized by} $b$ or that $b$ \textbf{submarjorizes} $a$ (in the ``noncommutative" sense of Hardy--Littlewood--P\'{o}lya).
We now define symmetric operator spaces.

\begin{defi}[Symmetric operator spaces]\label{def.symmopsp}
Let $E \subseteq S(\tau)$ be a linear subspace and $\|\cdot\|_E$ be a norm on $E$ such that $(E,\|\cdot\|_E)$ is a Banach space.
We call $(E,\|\cdot\|_E)$
\begin{enumerate}[label=(\alph*),leftmargin=2\parindent]
    \item a \textbf{symmetric} (or \textbf{rearrangement-invariant}) \textbf{space of $\boldsymbol{\tau}$-measurable operators} --- a \textbf{symmetric space}\footnote{Beware: This has nothing to do with the notion of a (Riemannian) symmetric space from geometry.} for short --- if $a \in S(\tau)$, $b \in E$, and $\mu(a) \leq \mu(b)$ imply $a \in E$ and $\|a\|_E \leq \|b\|_E$;
    \item a \textbf{strongly symmetric space of $\boldsymbol{\tau}$-measurable operators} --- a \textbf{strongly symmetric space} for short --- if it is a symmetric space, and $a, b \in E$ and $a \pprec b$ imply $\|a\|_E \leq \|b\|_E$;
    and
    \item a \textbf{fully symmetric space of $\boldsymbol{\tau}$-measurable operators} --- a \textbf{fully symmetric space} for short --- if $a \in S(\tau)$, $b \in E$, and $a \pprec b$ imply $a \in E$ and $\|a\|_E \leq \|b\|_E$.\label{item.fullysymm}
\end{enumerate}
If $(E,\|\cdot\|_E)$ is a symmetric space, then we define
\[
\mathrm{Proj}(E) \coloneqq E \cap \mathrm{Proj}(\cM) \; \text{ and } \; c_E \coloneqq \sup \mathrm{Proj}(E) \in \mathrm{Proj}(\cM)
\]
and call $c_E$ the \textbf{carrier projection} of $E$.
\end{defi}

Next, we describe a large class of examples of symmetric spaces.
Let $m$ be the Lebesgue measure on the positive halfline $(0,\infty)$ and $(\mathcal{N},\eta) = (L^{\infty}((0,\infty),m),\int_0^{\infty} \boldsymbol{\cdot} \,dm)$, where $L^{\infty}((0,\infty),m)$ is represented as multiplication operators on $L^2((0,\infty),m)$.
Then the set of densely defined, closed operators affiliated with $\mathcal{N}$ is precisely $L^0((0,\infty),m)$, i.e., the space of $m$-almost-everywhere equivalence classes of measurable functions $(0,\infty) \to \C$, viewed as unbounded multiplication operators on $L^2((0,\infty),m)$;
and
\[
S(\eta) = \{f \in L^0((0,\infty),m) : d_s(f) = m(\{x \in (0, \infty) : |f(x)| > s\}) < \infty \text{ for some } s \geq 0\}.
\]
For proofs of these facts, please see Section 2.3 of \cite{cz}.
A (strongly, fully) symmetric space of $\eta$-measurable operators is called a (\textbf{strongly}, \textbf{fully}) \textbf{symmetric Banach function space}.
For the classical theory of such spaces, please see Chapter II of \cite{kreininterp}.

\begin{fact}\label{fact.Etau}
Let $(E \subseteq L^0((0,\infty),m),\|\cdot\|_E)$ be a (strongly, fully) symmetric Banach function space.
If
\[
E(\tau) \coloneqq \{a \in S(\tau) : \mu(a) \in E\} \; \text{ and } \; \|a\|_{E(\tau)} \coloneqq \|\mu(a)\|_E \; \text{ for } \; a \in E(\tau),
\]
then $(E(\tau),\|\cdot\|_{E(\tau)})$ is a (strongly, fully) symmetric space of $\tau$-measurable operators. 
\end{fact}

For the strongly/fully symmetric cases, please see Section 9.1 of \cite{doddsNKS}.
For the (highly nontrivial) case of an arbitrary symmetric space, please see \cite{kaltonsukochev}.
When $1 \leq p \leq \infty$ and $E = L^p \coloneqq L^p((0,\infty),m)$, then $L^p(\tau)$ as defined using the construction in Fact \ref{fact.Etau} is a concrete description of the abstract (completion-based) definition from Section \ref{sec.wstarint}.
When $p=\infty$, this follows from Lemma 2.5(i) in \cite{fackkosaki};
when $p < \infty$, it follows from Lemma 2.5(iv) in \cite{fackkosaki} and Proposition 2.8 in \cite{doddsNKD}.
Moreover, we have
\[
(L^p(\tau),\|\cdot\|_{L^p(\tau)}) = \big(\{a \in S(\tau) : \tau(|a|^p) < \infty\},\tau(|\cdot|^p)^{\frac{1}{p}}\big)\vspace{2mm}\pagebreak
\]
when $1 \leq p < \infty$.
As a result,
\[
(L^p \cap L^{\infty})(\tau) = L^p(\tau) \cap L^{\infty}(\tau) = L^p(\tau) \cap \cM = \mathcal{L}^p(\tau)
\]
with equality of norms (if we give $\mathcal{L}^p(\tau)$ the norm $\max\{\|\cdot\|_{L^p(\tau)},\|\cdot\|\}$).
It is also true that
\[
(L^1+L^{\infty})(\tau) = L^1(\tau)+L^{\infty}(\tau) = L^1(\tau)+\cM
\]
with equality of norms.
(This follows from Proposition 2.5 in \cite{doddsNKD}.)
To be clear, if $Z$ is a vector space and $X,Y \subseteq Z$ are normed linear subspaces with respective norms $\|\cdot\|_X$ and $\|\cdot\|_Y$, then the subspace $X \cap Y \subseteq Z$ is given the norm $\|\cdot\|_{X \cap Y} \coloneqq \max\{\|\cdot\|_X,\|\cdot\|_Y\}$, and the subspace $X+Y \subseteq Z$ is given the norm $\|z\|_{X+Y} \coloneqq \inf\{\|x\|_X+\|y\|_Y : x \in X, \; y \in Y, \; z=x+y\}$.

In general, if $(E,\|\cdot\|_E)$ is a strongly symmetric space of $\tau$-measurable operators, then $E \subseteq L^1(\tau)+\cM$ with continuous inclusion, and $c_E = 1 \iff L^1(\tau) \cap \cM \subseteq E$ with continuous inclusion.
This is Lemma 25 in \cite{doddsNKS} (combined with the last paragraph of the proof of Lemma 3.4.6 in \cite{nikitopoulosMOI}).
By Theorem 4.1 in Section II.4.1 of \cite{kreininterp}, if $(\tilde{E},\|\cdot\|_{\tilde{E}})$ is a nonzero symmetric Banach function space, then $L^1 \cap L^{\infty} \subseteq \tilde{E} \subseteq L^1+L^{\infty}$ with continuous inclusions, i.e., $c_{\tilde{E}} = 1$.

Finally, we discuss K\"{o}the duals.
For a symmetric space $(E,\|\cdot\|_E)$, define
\begin{align*}
    E^{\times} & \coloneqq \{a \in S(\tau) : ab \in L^1(\tau), \, \text{ for all } \, b \in E\} \; \text{ and} \\
    \|a\|_{E^{\times}} & \coloneqq \sup\{\tau(|ab|) : b \in E, \; \|b\|_E \leq 1\} \, \text{ for } \, a \in S(\tau).
\end{align*}
Of course, $\|a\|_{E^{\times}}$ could be infinite.

\begin{fact}[K\"{o}the dual]\label{fact.Kothedual}
Let $(E,\|\cdot\|_E)$ be a strongly symmetric space of $\tau$-measurable operators with $c_E = 1$.
If $a \in S(\tau)$, then $\|a\|_{E^{\times}} =  \sup\{\tau(|ab|) : b \in \mathcal{L}^1(\tau) = L^1(\tau) \cap \cM, \; \|b\|_E \leq 1\}$.
Moreover, $a \in E^{\times}$ if and only if $\|a\|_{E^{\times}} < \infty$.
Finally, $\|\cdot\|_{E^{\times}}$ is a norm on $E^{\times}$ such that $(E^{\times},\|\cdot\|_{E^{\times}})$ is a fully symmetric space with $c_{E^{\times}} = 1$.
We call $E^{\times}$ the \textbf{K\"{o}the dual} of $E$.
\end{fact}
\begin{rem}
In the classical case of symmetric Banach function spaces, the K\"{o}the dual of $E$ is called the \textit{associate space of} $E$ or the \textit{space associated with} $E$.
\end{rem}

For a proof of this fact, please see Section 5 of \cite{doddsNKD} or Sections 5.2 and 6 of \cite{doddsNKS}.
Now, let $(E,\|\cdot\|_E)$ be a strongly symmetric space of $\tau$-measurable operators with $c_E = 1$. Since $E^{\times}$ is fully symmetric and $c_{E^{\times}} = 1$, we can consider the \textbf{K\"{o}the bidual} $(E^{\times\times},\|\cdot\|_{E^{\times\times}}) = ((E^{\times})^{\times},\|\cdot\|_{(E^{\times})^{\times}})$ of $E$ as a (fully) symmetric space.
It is always the case that $E \subseteq E^{\times \times}$ and $\|\cdot\|_{E^{\times\times}} \leq \|\cdot\|_E$ on $E$.
If $E=E^{\times\times}$ and $\|\cdot\|_E = \|\cdot\|_{E^{\times\times}}$ on $E$, then we call $E$ \textbf{K\"{o}the reflexive}.
(This term is not standard; a more common term is \textit{maximal}.)
Note that, by Fact \ref{fact.Kothedual}, if $E$ is K\"{o}the reflexive, then $E$ is automatically fully symmetric.

The following is a celebrated equivalent characterization of K\"{o}the reflexivity.
It is stated and proven as Proposition 5.14 in \cite{doddsNKD} and Theorem 32 in \cite{doddsNKS}.

\begin{thm}[Noncommutative Lorentz--Luxemburg]\label{thm.ncll}
Let $(E,\|\cdot\|_E)$ be a strongly symmetric space of $\tau$-measurable operators with $c_E=1$.
Then $E$ is K\"{o}the reflexive if and only if $E$ has the \textbf{Fatou property}:
whenever $(a_j)_{j \in J}$ is an increasing net ($j_1 \leq j_2 \Rightarrow a_{j_2} - a_{j_1} \in S(\tau)_+$) in $E \cap S(\tau)_+$ with $\sup_{j \in J}\|a_j\|_E < \infty$, we have that $\sup_{j \in J}a_j$ exists in $E \cap S(\tau)_+$ and $\|\sup_{j \in J}a_j\|_E = \sup_{j \in J}\|a_j\|_E$.
\end{thm}

The definition of the Fatou property involves rather arbitrary nets.
It is therefore reasonable to be concerned that verifying the Fatou property in classical situations might be quite difficult.
However, as we explain shortly, the sequence formulation of the Fatou property is equivalent in classical situations.
Let $(E \subseteq L^0((0,\infty),m),\|\cdot\|_E)$ be a symmetric Banach function space.
We say that $E$ has the \textbf{classical Fatou property} if whenever $(f_n)_{n \in \N}$ is an increasing sequence of nonnegative functions in $E$ such that $\sup_{n \in \N} \|f_n\|_E < \infty$, we have $\sup_{n \in \N} f_n \in E$ and $\|\sup_{n \in \N} f_n\|_E = \sup_{n \in \N} \|f_n\|_E$.
It turns out (Theorem 4.6 in Section 2.4 of \cite{bensharp}) that if $E$ has the classical Fatou property, then $E$ is fully symmetric, so we may speak of its K\"{o}the dual as a (fully) symmetric Banach function space when $E$ is nonzero.
The classical Lorentz--Luxemburg Theorem (e.g., Theorem 1 in Section 71 of \cite{zaanen}) says that a nonzero symmetric Banach function space has the classical Fatou property if and only if it is (strongly symmetric and) K\"{o}the reflexive.
In particular, by the Noncommutative Lorentz--Luxemburg Theorem, a symmetric Banach function space has the Fatou property if and only if it has the classical Fatou property.

\begin{ex}\label{ex.Fatou}
Let $(E,\|\cdot\|_E)$ be a nonzero strongly symmetric Banach function space (which implies $c_E = 1$ as noted above).
By Theorem 5.6 in \cite{doddsNKD},
\[
\big(E(\tau)^{\times},\|\cdot\|_{E(\tau)^{\times}}\big) = \big(E^{\times}(\tau),\|\cdot\|_{E^{\times}(\tau)}\big).
\]
In particular, if $E$ is K\"{o}the reflexive (i.e., has the classical Fatou property), then $E(\tau)$ is K\"{o}the reflexive (i.e., has the Fatou property) as well.
\end{ex}

\begin{rem}
Let $E$ be a symmetric Banach function space.
By Theorem 3 in Section 65 of \cite{zaanen}, $E$ has the classical Fatou property if and only if whenever $(f_n)_{n \in \N}$ is a sequence of nonnegative functions in $E$ with $\liminf_{n \to \infty} \|f_n\|_E < \infty$, we have $\liminf_{n \to \infty} f_n \in E$ and $\big\|\liminf_{n \to \infty} f_n\big\|_E \leq \liminf_{n \to \infty} \|f_n\|_E$, i.e, Fatou's Lemma holds for $\|\cdot\|_E$.
Hence the property's name.
\end{rem}

\section{Ideals of von Neumann algebras}\label{sec.ideals}

For Section \ref{sec.ideals}, fix a complex Hilbert space $(H,\la \cdot,\cdot \ra)$ and a von Neumann algebra $\cM \subseteq B(H)$.

\subsection{Properties to request of ideals}\label{sec.idealprop}

In this section, we introduce some abstract properties of ideals of $\cM$ that are useful in the study of MOIs and their applications to the differentiation of operator functions.
In Section \ref{sec.idealex1}, we give several classes of examples that do not require the theory of symmetric operator spaces to understand.
In Section \ref{sec.idealex2}, we give a large class of additional examples using the theory of symmetric operator spaces.

\begin{defi}[Symmetrically normed ideals]\label{def.sni}
Let $\cA$ be a Banach algebra and $\mathcal{J} \subseteq \cA$ be an ideal, i.e., a linear subspace such that $arb \in \mathcal{J}$ whenever $a,b \in \cA$ and $r \in \mathcal{J}$;
in this case, we write $\mathcal{J} \unlhd \cA$.
Suppose we have another norm $\|\cdot\|_{\mathcal{J}}$ on $\mathcal{J}$.
We call $(\mathcal{J},\|\cdot\|_{\mathcal{J}})$ a \textbf{Banach ideal} of $\cA$ if $(\mathcal{J},\|\cdot\|_{\mathcal{J}})$ is a Banach space and the inclusion $\iota_{\mathcal{J}} \colon (\mathcal{J},\|\cdot\|_{\mathcal{J}}) \hookrightarrow (\cA,\|\cdot\|_{\cA})$ is bounded;
in this case, we write
\[
(\mathcal{J},\|\cdot\|_{\mathcal{J}}) \unlhd \cA \; \text{ and } \; C_{\mathcal{J}} \coloneqq \|\iota_{\mathcal{J}}\|_{\mathcal{J} \to \cA} \in [0,\infty).
\]
If in addition $a,b \in \cA$ and $r \in \mathcal{J}$ imply
\[
\|arb\|_{\mathcal{J}} \leq \|a\|_{\cA}\|r\|_{\mathcal{J}}\|b\|_{\cA},
\]
then we call $(\mathcal{J},\|\cdot\|_{\mathcal{J}})$ a \textbf{symmetrically normed ideal} of $\cA$ and write
\[
(\mathcal{J},\|\cdot\|_{\mathcal{J}}) \sni \cA
\]
or $\mathcal{J} \sni \cA$ when confusion is unlikely.
\end{defi}
\begin{rem}
Beware: Definitions of a symmetrically normed ideal vary in the literature.
Sometimes it is required that $C_{\mathcal{J}} = 1$.
Sometimes $\cA$ is required to be a von Neumann or $C^*$-algebra and $\mathcal{J}$ is required to be a $\ast$-ideal with $\|r^*\|_{\mathcal{J}} = \|r\|_{\mathcal{J}}$, for all $r \in \mathcal{J}$.
Sometimes even more requirements are imposed.
We take the above minimal definition because it is all we need.
\end{rem}

We now define two additional properties one can demand of Banach or symmetrically normed ideals of a von Neumann algebra.
Before doing so, however, we make an observation.
Let $(\Sigma,\sH,\rho)$ be a measure space, $(\cI,\|\cdot\|_{\cI}) \unlhd \cM$ be a Banach ideal, and $F \colon \Sigma \to \cI \subseteq \cM$ be weak$^*$ measurable.
By definition,
\[
\underline{\int_{\Sigma}} \|F\| \, d\rho \leq C_{\cI}\underline{\int_{\Sigma}}\|F\|_{\cI} \, d\rho.
\]
In particular, if $\underline{\int_{\Sigma}} \|F\|_{\cI} \, d\rho < \infty$, then Proposition \ref{prop.wstarintexist} says that $F \colon \Sigma \to \cM$ is weak$^*$ integrable.

\begin{defi}[Properties of Banach ideals of $\cM$]\label{def.idealproperties}
Fix $(\cI,\|\cdot\|_{\cI}) \unlhd \cM$.
\begin{enumerate}[label=(\alph*),leftmargin=2\parindent]
    \item $\cI$ has the \textbf{Minkowski integral inequality property} --- or \textbf{property (M)} for short --- if whenever $(\Sigma,\sH,\rho)$ is a measure space and $F \colon \Sigma \to \cI \subseteq \cM$ is weak$^*$ measurable with $\underline{\int_{\Sigma}} \|F\|_{\cI} \, d\rho < \infty$, we have\label{item.Minkowski}
    \[
    \int_{\Sigma} F \, d\rho \in \cI \; \text{ and } \; \Bigg\|\int_{\Sigma} F \, d\rho \Bigg\|_{\cI} \leq \underline{\int_{\Sigma}} \|F\|_{\cI} \, d\rho.\pagebreak
    \]
    \item $\cI$ is \textbf{integral symmetrically normed} if whenever $(\Sigma,\sH,\rho)$ is a measure space, $A,B \colon \Sigma \to \cM$ are weak$^*$ measurable, $A(\cdot)\,c\,B(\cdot) \colon \Sigma \to \cM$ is weak$^*$ measurable whenever $c \hspace{-0.2mm}\in\hspace{-0.2mm} \cM$, and $\underline{\int_{\Sigma}}\|A\|\,\|B\|\,d\rho \hspace{-0.2mm}<\hspace{-0.2mm} \infty$, it follows that 
    \[
    \int_{\Sigma}A(\sigma)\,r\,B(\sigma)\,\rho(d\sigma) \in \cI \; \text{ and } \; \Bigg\|\int_{\Sigma}A(\sigma)\,r\,B(\sigma)\,\rho(d\sigma)\Bigg\|_{\cI} \leq \|r\|_{\cI}\underline{\int_{\Sigma}}\|A\|\,\|B\|\,d\rho,
    \]
    for all $r \in \cI$.\label{item.ISNI}
\end{enumerate}
\end{defi}
\begin{rem}
First, the name for property (M) is inspired by Theorem \ref{thm.Lpinteg}.
However, inequalities like the one required in \ref{item.Minkowski} are called triangle inequalities in the theory of vector-valued integrals.
Therefore, it would also be appropriate to name \ref{item.Minkowski} the ``integral triangle inequality property."
However, this would lead naturally to the abbreviation ``property (T)," which is already decidedly taken.
Second, if $H$ is separable, then one can show that the pointwise product of weak$^*$ measurable maps $\Sigma \to \cM$ is itself weak$^*$ measurable.
In particular, the requirement in \ref{item.ISNI} that ``$A(\cdot)\,c\,B(\cdot) \colon \Sigma \to \cM$ is weak$^*$ measurable whenever $c \in \cM$" is redundant when $H$ is separable.
\end{rem}

By testing the definition on the one-point probability space, we see that an integral symmetrically normed ideal is symmetrically normed.
We also have the following.

\begin{prop}\label{prop.idealproperties}
Let $(\cI,\|\cdot\|_{\cI}) \sni \cM$. If $\cI$ has property (M), then $\cI$ is integral symmetrically normed.
\end{prop}
\begin{proof}
Suppose that $\cI \sni \cM$ has property (M).
Let $A,B \colon \Sigma \to \cM$ be as in \ref{def.idealproperties}\ref{item.ISNI}, and fix $r \in \cI$.
Since $\cI$ is symmetrically normed, $\|A(\sigma)\,r\,B(\sigma)\|_{\cI} \leq \|r\|_{\cI}\|A(\sigma)\|\,\|B(\sigma)\|$ whenever $\sigma \in \Sigma$.
Applying the definition of property (M) to $F \coloneqq A(\cdot)\,r\,B(\cdot)$, we conclude $\int_{\Sigma}A(\sigma)\,r\,B(\sigma)\,\rho(d\sigma) \in \cI$ and
\[
\Bigg\|\int_{\Sigma}A(\sigma)\,r\,B(\sigma)\,\rho(d\sigma) \Bigg\|_{\cI} \leq \underline{\int_{\Sigma}}\|A(\sigma)\,r\,B(\sigma)\|_{\cI} \, \rho(d\sigma) \leq \|r\|_{\cI} \underline{\int_{\Sigma}}\|A\|\,\|B\|\,d\rho.
\]
Thus $\cI$ is integral symmetrically normed.
\end{proof}

\subsection{Examples of ideals I}\label{sec.idealex1}

In this section, we exhibit several examples of ideals with property (M), namely the trivial ideals, the noncommutative $L^p$-ideals, separable ideals, and the ideal of compact operators.

\begin{ex}[Trivial ideals]\label{ex.M}
The trivial symmetrically normed ideals $\cI = \{0\}$ and $\cI = \cM$ both have property (M).
The latter follows, of course, from Proposition \ref{prop.wstarintexist}.
\end{ex}

\begin{ex}[Noncommutative $L^p$ ideals]\label{ex.Lp}
Suppose $\cM$ is semifinite with normal, faithful, semifinite trace $\tau$.
If $1 \leq p < \infty$ and $\mathcal{L}^p(\tau)$ is given the norm $\|\cdot\|_{\mathcal{L}^p(\tau)} \coloneqq \max\{\|\cdot\|_{L^p(\tau)},\|\cdot\|\}$, then $(\mathcal{L}^p(\tau),\|\cdot\|_{\mathcal{L}^p(\tau)}) \sni \cM$ by Noncommutative H\"{o}lder's Inequality (Th\'{e}or\`{e}me 6 in \cite{dixmierLp}) and the completeness of $(L^p(\tau),\|\cdot\|_{L^p(\tau)})$ and $(\cM,\|\cdot\|)$.
If we combine Example \ref{ex.M} with Theorem \ref{thm.Lpinteg}, then we conclude $\mathcal{L}^p(\tau)$ has property (M) and is therefore integral symmetrically normed by Proposition \ref{prop.idealproperties}.
Note that if $(\cM,\tau) = (B(H),\Tr)$, then $(\mathcal{L}^p(\Tr),\|\cdot\|_{\mathcal{L}^p(\Tr)}) = (\cS_p(H),\|\cdot\|_{\cS_p})$ is the ideal of Schatten $p$-class operators on $H$.
\end{ex}

Notice that the ideal of compact operators is left out of the above examples.
To include it in the mix, we first prove that separable ideals have property (M).

\begin{prop}[Separable ideals]\label{prop.sep}
Fix $(\cI,\|\cdot\|_{\cI}) \unlhd \cM$.
If $(\cI,\|\cdot\|_{\cI})$ is separable, then $\cI$ has property (M).
In particular, if $(\cI,\|\cdot\|_{\cI}) \sni \cM$ is separable, then $\cI$ is integral symmetrically normed.
\end{prop}
\begin{proof}
To prove this, we make use of the basic theory of the Bochner integral;
please see, for instance, Appendix E of \cite{cohn} for the relevant background. 

Let $(\Sigma,\sH,\rho)$ be a measure space, $F \colon \Sigma \to \cI \subseteq \cM$ be weak$^*$ measurable, and $h,k \in H$.
Now, define $\ell_{h,k} \colon \cI \to \C$ by $r \mapsto \la rh,k \ra$.
Since the inclusion $\iota_{\cI} \colon (\cI,\|\cdot\|_{\cI}) \hookrightarrow (\cM,\|\cdot\|)$ is bounded, $\ell_{h,k}$ is a continuous function $\cI \to \C$.
Also, $\ell_{h,k} \circ F = \la F(\cdot)h,k \ra \colon \Sigma \to \C$ is measurable by assumption.
Since the collection $\{\ell_{h,k} : h,k \in H\}$ clearly separates points, we conclude from the (completeness and) separability of $\cI$ and Proposition 1.10 in Chapter I of \cite{vakhania} that $F \colon \Sigma \to (\cI,\|\cdot\|_{\cI})$ is Borel measurable.
Using again the separability of $\cI$, this implies $F \colon \Sigma \to (\cI,\|\cdot\|_{\cI})$ is strongly (or ``Bochner") measurable.
Therefore, if in addition $\underline{\int_{\Sigma}} \|F\|_{\cI} \, d\rho = \int_{\Sigma} \|F\|_{\cI} \,d\rho < \infty$, then $F \colon \Sigma \to (\cI,\|\cdot\|_{\cI})$ is also Bochner integrable, and --- by applying $\ell_{h,k}$ to the Bochner integral --- the Bochner and weak$^*$ integrals of $F$ agree.
Thus $\int_{\Sigma} F \, d\rho \in \cI$ and $\big\|\int_{\Sigma} F \, d\rho\big\|_{\cI} \leq \int_{\Sigma}\|F\|_{\cI} \, d\rho$, by the triangle inequality for Bochner integrals.
This completes the proof.
\end{proof}

In particular, if $H$ is separable, then the ideal $\cK(H) \sni B(H)$ of compact operators $H \to H$ has property (M).
Actually, this also implies the non-separable case by an argument suggested by J. Jeon.

\begin{lem}\label{lem.compact}
For a closed linear subspace $K \subseteq H$, write $\iota_K \colon K \to H$ and $\pi_K \colon H \to K$ for, respectively, the inclusion of and the orthogonal projection onto $K$.
Fix $A \in B(H)$.
Then $A \in \cK(H)$ if and only if $A_K \coloneqq \pi_K A \iota_K \in \cK(K)$, for all closed, separable linear subspaces $K \subseteq H$.
\end{lem}
\begin{proof}
The ``only if" direction is clear.
For the ``if" direction, suppose that $A_K = \pi_K A \iota_K \in \cK(K)$, for all closed, separable linear subspaces $K \subseteq H$.
If $(h_n)_{n \in \N}$ is a bounded sequence in $H$, then set
\[
K \coloneqq \overline{\spn}\{A^kh_n : k \in \N_0,\,n \in \N\}.
\]
Of course, $K$ is a separable, closed linear subspace of $H$ that contains $\{h_n : n \in \N\}$ and is invariant under $A$.
Since $A_K$ is compact, there is a subsequence $(h_{n_k})_{k \in \N}$ such that $(A_Kh_{n_k})_{k \in \N}$ converges.
But
\[
A_Kh_{n_k} = \pi_KAh_{n_k} = Ah_{n_k},
\]
for all $k \in \N$, since $K$ is $A$-invariant.
We conclude that $A \in \cK(H)$.
\end{proof}

\begin{prop}[Compact operators]\label{prop.compact}
$(\cK(H),\|\cdot\|) \sni \cM$ has property (M).
\end{prop}
\begin{proof}
Let $(\Sigma,\sH,\rho)$ be a measure space and $F \colon \Sigma \to \cK(H) \subseteq B(H)$ be weak$^*$ measurable with
\[
\underline{\int_{\Sigma}}\|F\|\,d\rho < \infty.
\]
Since we already know the triangle inequality for the operator norm, it suffices to prove $\int_{\Sigma} F \, d\rho \in \cK(H)$.
To this end, let $K \subseteq H$ be a closed, separable linear subspace.
Then, in the notation of Lemma \ref{lem.compact}, $F_K = \pi_KF(\cdot) \iota_K \colon \Sigma \to \cK(K) \subseteq B(K)$ is weak$^*$ measurable and
\[
\underline{\int_{\Sigma}}\|F_K\|\,d\rho \leq \underline{\int_{\Sigma}}\|F\|\,d\rho < \infty.
\]
Since $\cK(K)$ is separable, Proposition \ref{prop.sep} gives $\int_{\Sigma} F_K \, d\rho \in \cK(K)$.
Since
\[
\Bigg(\int_{\Sigma} F\,d\rho \Bigg)_K = \pi_K\Bigg(\int_{\Sigma} F\,d\rho \Bigg) \iota_K = \int_{\Sigma} \pi_KF(\sigma)\,\iota_K \, \rho(d\sigma) = \int_{\Sigma} F_K\,d\rho \in \cK(K),
\]
we conclude from Lemma \ref{lem.compact} that $\int_{\Sigma} F \, d\rho \in \cK(H)$.
\end{proof}
\begin{rem}\label{rem.compact}
In case one only wants to know $\cK(H)$ is integral symmetrically normed, there is a different proof available that does not go through the separable case first.
Indeed, let $(\Sigma,\sH,\rho)$ be a measure space and $A,B \colon \Sigma \to B(H)$ be as in \ref{def.idealproperties}\ref{item.ISNI}. To prove the claim, it suffices to show that if $c \in \cK(H)$, then $\int_{\Sigma}A(\sigma)\,c\,B(\sigma)\,\rho(d\sigma) \in \cK(H)$.
First, suppose $c$ has finite rank.
Then $c \in \cS_1(H)$.
Since $(\cS_1(H),\|\cdot\|_{\cS_1}) \unlhd B(H)$ is integral symmetrically normed, $\int_{\Sigma} A(\sigma)\,c\,B(\sigma) \,\rho(d\sigma) \in \cS_1(H) \subseteq \cK(H)$.
Now, if $c \in \cK(H)$ is arbitrary, then --- using, for instance, the singular value decomposition --- there is a sequence $(c_n)_{n \in \N}$ of finite-rank linear operators $H \to H$ such that $\|c_n - c\| \to 0$ as $n \to \infty$.
But then, by the operator norm triangle inequality, $\int_{\Sigma} A(\sigma)\,c_n\,B(\sigma)\,\rho(d\sigma) \to \int_{\Sigma} A(\sigma)\,c\,B(\sigma) \,\rho(d\sigma)$ in the operator norm topology as $n \to \infty$.
Since this exhibits $\int_{\Sigma}A(\sigma)\,c\,B(\sigma)\,\rho(d\sigma)$ as an operator norm limit of compact operators, we conclude it is compact, as desired.
\end{rem}

\subsection{Examples of ideals II}\label{sec.idealex2}

At this point, we shall make heavy use of the theory reviewed in Section \ref{sec.symmopsp}.
For the duration of this section, suppose $(\cM,\tau)$ is a semifinite von Neumann algebra.
To begin, we note that if $(E,\|\cdot\|_E)$ is a symmetric space of $\tau$-measurable operators, then
\[
(\mathcal{E},\|\cdot\|_{\mathcal{E}}) \coloneqq (E \cap \cM , \|\cdot\|_{E \cap \cM}) = (E \cap \cM , \max\{\|\cdot\|_E,\|\cdot\|\})  \sni \cM.
\]
This follows from Proposition 17 in \cite{doddsNKS}.
We call $\mathcal{E}$ the \textbf{ideal induced by $\boldsymbol{E}$}.
In this section, we prove that ideals induced by fully symmetric spaces are integral symmetrically normed and that ideals induced by symmetric spaces with the Fatou property have property (M).

\begin{thm}[Fully symmetric $\Rightarrow$ integral symmetrically normed]\label{thm.FSISN}
If $(E,\|\cdot\|_E)$ is a fully symmetric space, then $(\mathcal{E},\|\cdot\|_{\mathcal{E}}) \coloneqq (E \cap \cM , \|\cdot\|_{E \cap \cM}) \sni \cM$ is integral symmetrically normed.
\end{thm}
\begin{proof}
Let $(\Sigma,\sH,\rho)$ be a measure space and $A,B \colon \Sigma \to \cM$ be as in \ref{def.idealproperties}\ref{item.ISNI}.
Define $T_{\infty} \colon \cM \to \cM$ by $\cM \ni c \mapsto \int_{\Sigma} A(\sigma)\,c\,B(\sigma) \,\rho(d\sigma) \in \cM$.
Then $\|T_{\infty}\|_{\cM \to \cM} \leq \underline{\int_{\Sigma}}\|A\|\,\|B\|\,d\rho$ by the operator norm triangle inequality.
Also, if $c \in L^1(\tau) \cap \cM$, then
\[
\|T_{\infty}c\|_{L^1(\tau)} \leq \underline{\int_{\Sigma}}\|A(\sigma)\,c\,B(\sigma)\|_{L^1(\tau)} \, \rho(d\sigma) \leq \|c\|_{L^1(\tau)} \underline{\int_{\Sigma}}\|A\|\,\|B\|\,d\rho
\]
by Theorem \ref{thm.Lpinteg}.
Since $L^1(\tau) \cap \cM$ is dense in $L^1(\tau)$ (Proposition 2.8 in \cite{doddsNKD}), we get that $T_{\infty}|_{L^1(\tau) \cap \cM}$ extends uniquely to a bounded linear map $T_1 \colon L^1(\tau) \to L^1(\tau)$ with $\|T_1\|_{L^1(\tau)\to L^1(\tau)} \leq \underline{\int_{\Sigma}}\|A\|\,\|B\|\,d\rho$.
Since $T_{\infty}$ and $T_1$ agree on $L^1(\tau) \cap \cM$, we obtain a well-defined linear map $T \colon L^1(\tau) + \cM \to L^1(\tau) + \cM$ by setting $T(x+y) \coloneqq T_1x+T_{\infty}y$ for $x \in L^1(\tau)$ and $y \in \cM$.
Moreover,
\[
\|T\|_{L^1(\tau)+\cM \to L^1(\tau)+\cM} \leq \max\{\|T_1\|_{L^1(\tau)\to L^1(\tau)},\|T_{\infty}\|_{\cM \to \cM}\} \leq \underline{\int_{\Sigma}}\|A\|\,\|B\|\,d\rho.
\]
By Proposition 4.1 in \cite{doddsNKD}, this implies
\[
Tc \pprec \Bigg(\underline{\int_{\Sigma}}\|A\|\,\|B\|\,d\rho\Bigg)c, \, \text{ for all } c \in L^1(\tau)+\cM.
\]
In particular, if $c \in E \subseteq L^1(\tau)+\cM$, then
\[
Tc \in E \; \text{ and } \; \|Tc\|_E \leq \|c\|_E\underline{\int_{\Sigma}}\|A\|\,\|B\|\,d\rho
\]
because $E$ is fully symmetric;
in other words, $T$ restricts to a bounded linear map $T_E \colon E \to E$ with $\|T_E\|_{E \to E} \leq \underline{\int_{\Sigma}}\|A\|\,\|B\|\,d\rho$.
We conclude that if $c \in \mathcal{E} = E \cap \cM$, then
\[
\int_{\Sigma} A(\sigma) \, c \, B(\sigma) \,\rho(d\sigma) = T_{\infty}c = T_Ec \in \mathcal{E} \; \text{ and } \; \Bigg\|\int_{\Sigma} A(\sigma) \, c \, B(\sigma) \, \rho(d\sigma) \Bigg\|_{\mathcal{E}} \leq \|c\|_{\mathcal{E}}\underline{\int_{\Sigma}}\|A\|\,\|B\|\,d\rho.
\]
Thus $\mathcal{E}$ is integral symmetrically normed.
\end{proof}
\begin{rem}
The argument above is inspired in part by Section 4.4 of \cite{doddssub2}.
\end{rem}

The second main result of this section upgrades Theorem \ref{thm.FSISN} when the symmetric space in question is a K\"{o}the dual.
(It also generalizes Theorem \ref{thm.Lpinteg}.)

\begin{thm}[K\"{o}the duals and property (M)]\label{thm.KothePropM}
Let $(E,\|\cdot\|_E)$ be a strongly symmetric space with $c_E = 1$.
If $(\Sigma,\sH,\rho)$ is a measure space and $F \colon \Sigma \to \cM$ is weak$^*$ integrable, then
\[
\Bigg\|\int_{\Sigma} F\,d\rho \Bigg\|_{E^{\times}} \leq \underline{\int_{\Sigma}} \|F\|_{E^{\times}} \,d\rho.
\]
In particular, $(\mathcal{E}^{\times},\|\cdot\|_{\mathcal{E}^{\times}}) \coloneqq (E^{\times} \cap \cM,\|\cdot\|_{E^{\times}\cap \cM}) \sni \cM$ has property (M).
\end{thm}
\begin{proof}
Let $a \coloneqq \int_{\Sigma} F \,d\rho \in \cM$.
By Fact \ref{fact.Kothedual} (twice) and Theorem \ref{thm.Lpinteg}, we have\vspace{-0.4mm}
\begin{align*}
    \Bigg\|\int_{\Sigma} F\,d\rho\Bigg\|_{E^{\times}} & = \|a\|_{E^{\times}} = \sup\{\tau(|ab|) : b \in \mathcal{L}^1(\tau), \; \|b\|_E \leq 1\} \\[-0.4mm]
    & = \sup\Bigg\{\Bigg\|\int_{\Sigma} F(\sigma)\,b\,\rho(d\sigma) \Bigg\|_{L^1(\tau)} : b \in \mathcal{L}^1(\tau), \; \|b\|_E \leq 1\Bigg\} \\[-0.4mm]
    & \leq \sup\Bigg\{\underline{\int_{\Sigma}} \|F(\sigma)\,b\|_{L^1(\tau)}\,\rho(d\sigma) : b \in \mathcal{L}^1(\tau), \; \|b\|_E \leq 1\Bigg\} \leq \underline{\int_{\Sigma}}\|F\|_{E^{\times}}\,d\rho,\vspace{-0.4mm}
\end{align*}
as desired.
\end{proof}

\begin{cor}\label{cor.KbidualMink}
Let $(E,\|\cdot\|_E)$ be a strongly symmetric space with $c_E = 1$.
If $(\Sigma,\sH,\rho)$ is a measure space, $F \colon \Sigma \to \cM$ is weak$^*$ integrable, and $F(\Sigma) \subseteq E \cap \cM$, then\vspace{-0.4mm}
\[
\Bigg\|\int_{\Sigma} F\,d\rho \Bigg\|_{E^{\times \times}} \leq \underline{\int_{\Sigma}} \|F\|_E \,d\rho.\vspace{-0.4mm}
\]
In particular, by Fact \ref{fact.Kothedual}, if the right hand side is finite, then $\int_{\Sigma} F \,d\rho \in E^{\times\times}$.
\end{cor}
\begin{proof}
Applying Theorem \ref{thm.KothePropM} to the space $E^{\times\times} = (E^{\times})^{\times}$ and using that $\|\cdot\|_{E^{\times\times}} \leq \|\cdot\|_E$ on $E$, we get\vspace{-0.4mm}
\[
\Bigg\|\int_{\Sigma} F\,d\rho \Bigg\|_{E^{\times \times}} \leq \underline{\int_{\Sigma}} \|F\|_{E^{\times\times}}\,d\rho \leq \underline{\int_{\Sigma}} \|F\|_E \,d\rho,\vspace{-0.4mm}
\]
as desired.
\end{proof}
\begin{rem}
Please see equation (0.5) in Section II.0.3 of \cite{kreininterp} for a classical analog of this Minkowski-type integral inequality. 
\end{rem}

Combining the Noncommutative Lorentz--Luxemburg Theorem with Corollary \ref{cor.KbidualMink}, we get the following.

\begin{thm}[Fatou property $\Rightarrow$ property (M)]\label{thm.KRpropM}
Let $(E,\|\cdot\|_E)$ be a strongly symmetric space with $c_E = 1$.
Suppose $(\Sigma,\sH,\rho)$ is a measure space, $F \colon \Sigma \to \cM$ is weak$^*$ integrable, and $F(\Sigma) \subseteq E \cap \cM$.
If $E$ has the Fatou property and $\underline{\int_{\Sigma}} \|F\|_E \,d\rho < \infty$, then\vspace{-0.4mm}
\[
\int_{\Sigma} F \,d\rho \in E \; \text{ and } \; \Bigg\|\int_{\Sigma} F\,d\rho \Bigg\|_E \leq \underline{\int_{\Sigma}} \|F\|_E \,d\rho.\vspace{-0.4mm}
\]
In particular, $(\mathcal{E},\|\cdot\|_{\mathcal{E}}) \coloneqq (E \cap \cM,\|\cdot\|_{E \cap \cM}) \sni \cM$ has property (M).
\end{thm}
\begin{proof}
By the Noncommutative Lorentz--Luxemburg Theorem (Theorem \ref{thm.ncll}), $(E,\|\cdot\|_E) = (E^{\times\times},\|\cdot\|_{E^{\times\times}})$.
Therefore, by Corollary \ref{cor.KbidualMink}, we know $\int_{\Sigma} F \,d\rho \in E^{\times\times} = E$ and\vspace{-0.4mm}
\[
\Bigg\|\int_{\Sigma} F\,d\rho \Bigg\|_E = \Bigg\|\int_{\Sigma} F\,d\rho \Bigg\|_{E^{\times \times}} \leq \underline{\int_{\Sigma}} \|F\|_E \,d\rho,\vspace{-0.4mm}
\]
as desired.
\end{proof}

\subsection{Comments about property (F)}\label{sec.propF}

A Banach ideal $(\cI,\|\cdot\|_{\cI}) \unlhd \cM$ has (the \textbf{sequential}) \textbf{property (F)} if whenever $a \in \cM$ and $(a_j)_{j \in J}$ is a net (sequence) in $\cI$ such that $\sup_{j \in J}\|a_j\|_{\cI} < \infty$ and $a_j \to a$ in the S$^*$OT, we have $a \in \cI$ and $\|a\|_{\cI} \leq \sup_{j \in J}\|a_j\|_{\cI}$.
In \cite{azamovetal}, certain multiple operator integrals in invariant operator ideals with property (F) are considered.
We now take some time to discuss the relationship between properties (M) and (F).
First, there are certainly ideals with property (M) that do not have property (F), e.g., the ideal of compact operators (Proposition \ref{prop.compact}).
Second, as mentioned in \cite{azamovetal}, the motivating example of an invariant operator ideal with property (F) is an ideal induced via Fact \ref{fact.Etau} by a (nonzero) symmetric Banach function space with the Fatou property.
By Theorem \ref{thm.KRpropM} and Example \ref{ex.Fatou}, such ideals have property (M).
Third, the author is unaware of an example of a symmetrically normed ideal with property (F) that does not have property (M).
It would be interesting to know if such an ideal exists.
\pagebreak

In this context, it is worth discussing a technical issue in \cite{azamovetal} with its treatment of operator-valued integrals.
For the rest of this section, assume $H$ is separable.
It is implicitly assumed in the proof of (the second sentence of) Lemma 4.6 in \cite{azamovetal} that at least some form of the integral triangle inequality holds for the $\cI$-norm $\|\cdot\|_{\cI}$ when $\cI$ has property (F).
Specifically, it seems to be assumed that if $(\Sigma,\sH,\rho)$ is a finite measure space and $F \colon \Sigma \to \cI \subseteq \cM$ is $\|\cdot\|_{\cI}$-bounded and weak$^*$ measurable, then
\[
\int_{\Sigma} F\,d\rho \in \cI \; \text{ and } \; \Bigg\|\int_{\Sigma} F\,d\rho\Bigg\|_{\cI} \leq \int_{\Sigma}\|F\|_{\cI}\,d\rho
\]
(ignoring that $\|F\|_{\cI}$ may not be measurable).
Let us call this the \textbf{finite property (M)}.
Then we may rephrase the implicit claim as ``property (F) implies the finite property (M)."
As far as the author can tell, the arguments in \cite{azamovetal} are only sufficient to prove 
\[
\int_{\Sigma} F\,d\rho \in \cI \; \text{ and } \; \Bigg\|\int_{\Sigma} F\,d\rho\Bigg\|_{\cI} \leq \rho(\Sigma)\,\sup_{\sigma \in \Sigma}\|F(\sigma)\|_{\cI}.
\]
Indeed, the authors of \cite{azamovetal} prove that $\cI$ has property (F) if and only if $\cI_1 = \{r \in \cI : \|r\|_{\cI} \leq 1\}$ is a complete, separable metric space in the strong$^*$ operator topology and then apply Propositions 1.9--1.10 in Chapter I of \cite{vakhania} to approximate $F$ by simple functions in the strong$^*$ operator topology.
Crucially, Propositions 1.9--1.10 in \cite{vakhania} only guarantee the existence of a sequence $(F_n)_{n \in \N}$ of simple functions $\Sigma \to \cI$ such that $\sup_{\sigma \in \Sigma} \|F_n(\sigma)\|_{\cI} \leq \sup_{\sigma \in \Sigma}\|F(\sigma)\|_{\cI}$, for all $n \in \N$, and $F_n \to F$ pointwise in the strong$^*$ operator topology as $n \to \infty$.
Now, by Proposition \ref{prop.opDCT}, $\int_{\Sigma} F_n \, d\rho \to \int_{\Sigma} F\,d\rho$ in the strong$^*$ operator topology as $n \to \infty$.
Also, by the (obvious) triangle inequality for integrals of simple functions, if $k \in \N$, then
\[
\sup_{n \geq k}\Bigg\|\int_{\Sigma} F_n\,d\rho\Bigg\|_{\cI} \leq \sup_{n \geq k}\int_{\Sigma} \|F_n\|_{\cI} \,d\rho \leq \int_{\Sigma} \sup_{n \geq k} \|F_n\|_{\cI} \, d\rho \leq \rho(\Sigma) \,\sup_{\sigma \in \Sigma}\|F(\sigma)\|_{\cI}.
\]
Thus (the sequential) property (F) and the Dominated Convergence Theorem give
\[
\int_{\Sigma} F \, d\rho \in \cI \; \text{ and } \; \Bigg\|\int_{\Sigma} F\,d\rho \Bigg\|_{\cI} \leq \int_{\Sigma} \limsup_{n \to \infty} \|F_n\|_{\cI} \,d\rho \leq \rho(\Sigma) \,\sup_{\sigma \in \Sigma}\|F(\sigma)\|_{\cI}. \numberthis\label{eq.propFestim}
\]
The definition of property (F) does \textit{not} guarantee that $\|F_n(\sigma)\|_{\cI} \to \|F(\sigma)\|_{\cI}$ as $n \to \infty$, so we cannot evaluate the limit superior above much further without an upgraded version of property (F).
(Interestingly, this does not damage the applications in \cite{azamovetal}, since it seems only the estimate \eqref{eq.propFestim} is used seriously.)
It therefore seems that property (F) \textit{almost} implies some weaker form of property (M) \hspace{-0.45mm}---\hspace{-0.45mm} but \hspace{-0.2mm}perhaps\hspace{-0.2mm} not \hspace{-0.2mm}quite.

\begin{rem}
Though we centered the discussion above on the ``finite property (M)," it is worth pointing out that in order to prove Lemma 4.6 in \cite{azamovetal}, it would actually be sufficient to know the following ``finite integral symmetrically normed" condition:
for every finite measure space $(\Sigma,\sH,\rho)$ and $\|\cdot\|$-bounded, weak$^*$ measurable $A,B \colon \Sigma \to \cM$, we have $\int_{\Sigma}A(\sigma)\,r\,B(\sigma) \,\rho(d\sigma) \in \cI$ and $\big\|\int_{\Sigma}A(\sigma)\,r\,B(\sigma)\,\rho(d\sigma)\big\|_{\cI} \leq \|r\|_{\cI}\underline{\int_{\Sigma}}\|A\|\,\|B\|\,d\rho$, for all $r \in \cI$.
As mentioned, in the presence of property (F), we would already know $\int_{\Sigma} A(\sigma)\,r\,B(\sigma)\,\rho(d\sigma) \in \cI$, so --- as was the case above --- it is really only the integral triangle inequality that is potentially missing.
\end{rem}

\section{Differentiating operator functions}\label{sec.derivopfunc}

\subsection{Multiple operator integrals (MOIs) in ideals}\label{sec.MOIsinI}

We begin with a review of some information from \cite{nikitopoulosMOI} about (a simplified version of) the ``separation of variables" approach to defining multiple operator integrals.
For the duration of this section, fix $k \in \N$ and, for each $j \in \{1,\ldots,k+1\}$, a projection-valued measure space $(\Om_j,\sF_j,H,P_j)$ such that $P_j(G) \in \cM$ whenever $G \in \sF_j$.
Also, write
\[
(\Om,\sF) \coloneqq (\Om_1 \times \cdots \times \Om_{k+1},\sF_1 \otimes \cdots \otimes \sF_{k+1})
\]
for the product measurable space.

\begin{defi}[Integral projective tensor products I]\label{def.babyIPTP}
Let $\varphi \colon \Om \to \C$ be a function.
A $\boldsymbol{\ell^{\infty}}$\textbf{-integral projective decomposition} ($\ell^{\infty}$-IPD) of $\varphi$ is a choice $(\Sigma,\rho,\varphi_1,\ldots,\varphi_{k+1})$ of a $\sigma$-finite measure space $(\Sigma,\sH,\rho)$ and measurable functions $\varphi_1 \colon \Om_1 \times \Sigma \to \C,\ldots,\varphi_{k+1} \colon \Om_{k+1} \times \Sigma \to \C$ such that 
\begin{enumerate}[label=(\alph*),leftmargin=2\parindent]
    \item $\varphi_j(\cdot,\sigma) \in \ell^{\infty}(\Om_j,\sF_j)$, for all $j \in \{1,\ldots,k+1\}$ and $\sigma \in \Sigma$; \label{item.nullset}
    \item $\overline{\int_{\Sigma}} \,\|\varphi_1(\cdot,\sigma)\|_{\ell^{\infty}(\Om_1)} \cdots \|\varphi_{k+1}(\cdot,\sigma)\|_{\ell^{\infty}(\Om_{k+1})} \, \rho(d\sigma) < \infty$; and\label{item.upperintegofphis}
    \item $\varphi(\boldsymbol{\om}) = \int_{\Sigma} \varphi_1(\om_1,\sigma) \cdots \varphi_{k+1}(\om_{k+1},\sigma) \, \rho(d\sigma)$, for all $\boldsymbol{\om} = (\om_1,\ldots,\om_{k+1}) \in \Om$.\label{item.equality}
\end{enumerate}
Now, define
\begin{align*}
    \|\varphi\|_{\ell^{\infty}(\Om_1,\sF_1) \iotimes \cdots \iotimes \ell^{\infty}(\Om_{k\hspace{-0.1mm}+\hspace{-0.1mm}1},\sF_{k\hspace{-0.1mm}+\hspace{-0.1mm}1})} &\hspace{-0.25mm} \coloneqq\hspace{-0.15mm} \inf\hspace{-0.75mm}\Bigg\{\overline{\int_{\Sigma}}\hspace{-0.5mm} \|\varphi_1(\cdot,\sigma)\|_{\ell^{\infty}(\Om_1)}\hspace{-0.5mm}\cdots\hspace{-0.5mm}\|\varphi_{k+1}(\cdot,\sigma)\|_{\ell^{\infty}(\Om_{k\hspace{-0.1mm}+\hspace{-0.1mm}1})} \rho(d\sigma) : (\Sigma,\rho,\varphi_1,\ldots,\varphi_{k\hspace{-0.1mm}+\hspace{-0.1mm}1}\hspace{-0.1mm})\\
    & \; \; \; \; \; \; \; \; \; \; \; \; \; \; \; \; \; \; \; \; \; \; \; \; \; \; \; \text{ is a } \ell^{\infty}\text{-integral projective decomposition of } \varphi\Bigg\},
\end{align*}
where $\inf \emptyset \coloneqq \infty$.
Finally, we define
\[
\ell^{\infty}(\Om_1,\sF_1) \iotimes \cdots \iotimes \ell^{\infty}(\Om_{k+1},\sF_{k+1}) \coloneqq \{\varphi \in \ell^{\infty}(\Om,\sF) : \|\varphi\|_{\ell^{\infty}(\Om_1,\sF_1) \iotimes \cdots \iotimes \ell^{\infty}(\Om_{k+1},\sF_{k+1})} < \infty\}
\]
to be the \textbf{integral projective tensor product of} $\boldsymbol{\ell^{\infty}(\Om_1,\sF_1),\ldots,\ell^{\infty}(\Om_{k+1},\sF_{k+1})}$.
\end{defi}

It is easy to see that $\|\cdot\|_{\ell^{\infty}(\Om)} \leq \|\cdot\|_{\ell^{\infty}(\Om_1,\sF_1) \iotimes \cdots \iotimes \ell^{\infty}(\Om_{k+1},\sF_{k+1})}$.
Also,
\[
\big(\ell^{\infty}(\Om_1,\sF_1) \iotimes \cdots \iotimes \ell^{\infty}(\Om_{k+1},\sF_{k+1}), \|\cdot\|_{\ell^{\infty}(\Om_1,\sF_1) \iotimes \cdots \iotimes \ell^{\infty}(\Om_{k+1},\sF_{k+1})}\big)
\]
is a Banach $\ast$-algebra under pointwise operations.
This is a special case of Proposition 4.1.4 in \cite{nikitopoulosMOI}.
(Please see Example 4.1.5 in \cite{nikitopoulosMOI} as well.)

\begin{thm}[Definition of MOIs \cite{nikitopoulosMOI}]\label{thm.MOIsinM}
If $(\Sigma,\rho,\varphi_1,\ldots,\varphi_{k+1})$ is a $\ell^{\infty}$-integral projective decomposition of $\varphi\in \ell^{\infty}(\Om_1,\sF_1) \iotimes \cdots \iotimes \ell^{\infty}(\Om_{k+1},\sF_{k+1})$ and $b_1,\ldots,b_k \in \cM$, then the map
\[
\Sigma \ni \sigma \mapsto P_1(\varphi_1(\cdot,\sigma))\,b_1 \cdots P_k(\varphi_k(\cdot,\sigma))\,b_k \,P_{k+1}(\varphi_{k+1}(\cdot,\sigma)) \in \cM
\]
is weak$^*$ integrable, and
\begin{align*}
    \big(I^{P_1,\ldots,P_{k+1}}\varphi\big)[b_1,\ldots,b_k] &  = \int_{\Om_{k+1}}\cdots\int_{\Om_1}\varphi(\om_1,\ldots,\om_{k+1})\,P_1(d\om_1)\,b_1\cdots P_k(d\om_k)\,b_k\,P_{k+1}(d\om_{k+1}) \\
    & \coloneqq \int_{\Sigma}P_1(\varphi_1(\cdot,\sigma))\,b_1 \cdots P_k(\varphi_k(\cdot,\sigma))\,b_k \,P_{k+1}(\varphi_{k+1}(\cdot,\sigma))\,\rho(d\sigma) \in \cM
\end{align*}
is independent of the chosen decomposition $(\Sigma,\rho,\varphi_1,\ldots,\varphi_{k+1})$ of $\varphi$.
Moreover,
\[
\big\|\big(I^{P_1,\ldots,P_{k+1}}\varphi\big)[b_1,\ldots,b_k]\big\| \leq \|\varphi\|_{\ell^{\infty}(\Om_1,\sF_1) \iotimes \cdots \iotimes \ell^{\infty}(\Om_{k+1},\sF_{k+1})}\|b_1\|\cdots\|b_k\|,
\]
for all $b_1,\ldots,b_k \in \cM$.
Writing $\boldsymbol{P} = (P_1,\ldots,P_{k+1})$, we call $I^{\boldsymbol{P}}\varphi \colon \cM^k \to \cM$ the \textbf{multiple operator integral} (MOI) of $\varphi$ with respect to $\boldsymbol{P}$.
\end{thm}

This follows from Theorem 1.1.3 (and Equation (4.14)) in \cite{nikitopoulosMOI}.
We shall also need to know that, in general, if $\varphi_j \colon \Om_j \times \Sigma \to \C$ is measurable and $\varphi_j(\cdot,\sigma) \in \ell^{\infty}(\Om_j,\sF_j)$ for $j \in \{1,\ldots,k+1\}$ and $\sigma \in \Sigma$, then
\[
\Sigma \ni \sigma \mapsto P_1(\varphi_1(\cdot,\sigma))\,b_1\cdots P_k(\varphi_k(\cdot,\sigma))\,b_k\,P_{k+1}(\varphi_{k+1}(\cdot,\sigma)) \in \cM
\]
is weak$^*$ measurable, for all $b_1,\ldots,b_k \in \cM$.
This follows from a repeated application of Proposition 4.2.3 in \cite{nikitopoulosMOI}, and we shall use it in the sequel without further comment.

\begin{nota}
If $a_1,\ldots,a_{k+1} \aff \cM_{\nu}$ and $\varphi \in \ell^{\infty}(\sigma(a_1),\cB_{\sigma(a_1)})\iotimes \cdots \iotimes \ell^{\infty}(\sigma(a_{k+1}),\cB_{\sigma(a_{k+1})})$, we write
\[
\varphi(a_1,\ldots,a_{k+1})\sh [b_1,\ldots,b_k] = \varphi(\boldsymbol{a}) \sh b = \big(I^{\boldsymbol{a}}\varphi\big)[b] \coloneqq \big(I^{P^{a_1},\ldots,P^{a_{k+1}}}\varphi\big)[b_1,\ldots,b_k] \in \cM,
\]
for all $b = (b_1,\ldots,b_k) \in \cM^k$, where $\boldsymbol{a} \coloneqq (a_1,\ldots,a_{k+1})$.
\end{nota}
\begin{rem}
Please see Remark 4.2.14 in \cite{nikitopoulosMOI} for an explanation of the use of the $\sh$ symbol above.
Also, please be aware that $T_{\varphi}^{a_1,\ldots,a_{k+1}}$ is a common alternative way to notate $I^{a_1,\ldots,a_{k+1}}\varphi$.
\end{rem}
\pagebreak

The following are two algebraic properties of MOIs that we shall use.
They are proven as part of Proposition 4.3.1 in \cite{nikitopoulosMOI}.

\begin{prop}[Algebraic properties of MOIs]\label{prop.linandmult}
Suppose $1 \leq m \leq k$.
\begin{enumerate}[label=(\roman*),font=\normalfont,leftmargin=2\parindent]
    \item If $\varphi,\psi \in \ell^{\infty}(\Om_1,\sF_1) \iotimes \cdots \iotimes \ell^{\infty}(\Om_{k+1},\sF_{k+1})$ and $\alpha \in \C$, then $I^{\boldsymbol{P}}(\varphi+\alpha\,\psi) = I^{\boldsymbol{P}}\varphi+\alpha\,I^{\boldsymbol{P}}\psi$.\label{item.lin}
    \item If $\psi \in \ell^{\infty}(\Om_m,\sF_m) \iotimes \ell^{\infty}(\Om_{m+1},\sF_{m+1})$, $\tilde{\psi}(\om_1,\ldots,\om_{k+1}) \coloneqq \psi(\om_m,\om_{m+1})$ for $(\om_1,\ldots,\om_{k+1}) \in \Om$, and $\varphi \in \ell^{\infty}(\Om_1,\sF_1) \iotimes \cdots \iotimes \ell^{\infty}(\Om_{k+1},\sF_{k+1})$, then
    \[
    \big(I^{P_1,\ldots,P_{k+1}}(\varphi\tilde{\psi})\big)[b_1,\ldots,b_k] = \big(I^{P_1,\ldots,P_{k+1}}\varphi\big)\big[b_1,\ldots,b_{m-1},\big(I^{P_m,P_{m+1}}\psi\big)[b_m],b_{m+1},\ldots,b_k\big],
    \]
    for all $(b_1,\ldots,b_k) \in \cM^k$.\label{item.mult2}
\end{enumerate}
\end{prop}

Finally, we restrict the MOI in Theorem \ref{thm.MOIsinM} to certain ideals of $\cM$.

\begin{defi}[MOI-friendly ideals]\label{def.MOIfriendly}
Fix $(\cI,\|\cdot\|_{\cI}) \unlhd \cM$.
We say that $\cI$ is \textbf{MOI-friendly} if whenever we are in the setup of Theorem \ref{thm.MOIsinM} and $j \in \{1,\ldots,k\}$, the MOI $I^{\boldsymbol{P}}\varphi \colon \cM^k \to \cM$ restricts to a bounded $k$-linear map (Notation \ref{nota.bddmultilin})
\[
(\cM,\|\cdot\|)^{j-1} \times (\cI,\|\cdot\|_{\cI}) \times (\cM,\|\cdot\|)^{k-j} \to (\cI,\|\cdot\|_{\cI})
\]
with operator norm at most $\|\varphi\|_{\ell^{\infty}(\Om_1,\sF_1) \iotimes \cdots \iotimes \ell^{\infty}(\Om_{k+1},\sF_{k+1})}$.
In this case, $I^{\boldsymbol{P}}\varphi$ also restricts to a bounded $k$-linear map $(\cI,\|\cdot\|_{\cI})^k \to (\cI,\|\cdot\|_{\cI})$ with operator norm at most $C_{\cI}^{k-1}\|\varphi\|_{\ell^{\infty}(\Om_1,\sF_1) \iotimes \cdots \iotimes \ell^{\infty}(\Om_{k+1},\sF_{k+1})}$.
\end{defi}

This definition may seem contrived, but the following shows that all the examples of ideals from Sections \ref{sec.idealex1} and \ref{sec.idealex2} are MOI-friendly.

\begin{prop}\label{prop.MOIfriendly}
If $(\cI,\|\cdot\|_{\cI}) \unlhd \cM$ is integral symmetrically normed, then $\cI$ is MOI-friendly.
\end{prop}
\begin{proof}
Suppose that $\cI$ is integral symmetrically normed and that we are in the setup of Theorem \ref{thm.MOIsinM}.
Let $j \in \{1,\ldots,k\}$, $b = (b_1,\ldots,b_k) \in \cM^{j-1} \times \cI \times \cM^{k-j}$, and $(\Sigma,\rho,\varphi_1,\ldots,\varphi_{k+1})$ be a $\ell^{\infty}$-integral projective decomposition of $\varphi \in \ell^{\infty}(\Om_1,\sF_1) \iotimes \cdots \iotimes \ell^{\infty}(\Om_{k+1},\sF_{k+1})$.
Now, we apply the definition of integral symmetrically normed with
\begin{align*}
    A(\sigma) & \coloneqq \Bigg(\prod_{j_1=1}^{j-1}P_{j_1}(\varphi_{j_1}(\cdot,\sigma))\,b_{j_1}\Bigg)P_j(\varphi_j(\cdot,\sigma)) \, \text{ and} \\
    B(\sigma) & \coloneqq P_{j+1}(\varphi_{j+1}(\cdot,\sigma))\prod_{j_2=j+2}^{k+1}b_{j_2-1}P_{j_2}(\varphi_{j_2}(\cdot,\sigma)),
\end{align*}
where empty products are the identity.
This yields $\big(I^{\boldsymbol{P}}\varphi\big)[b] = \int_{\Sigma} A(\sigma)\,b_j\,B(\sigma)\,\rho(d\sigma) \in \cI$ and
\[
\big\|\big(I^{\boldsymbol{P}}\varphi\big)[b]\big\|_{\cI} \leq \|b_j\|_{\cI}\underline{\int_{\Sigma}}\|A\|\,\|B\|\,d\rho \leq \|b_j\|_{\cI}\prod_{p \neq j}\|b_p\|\underline{\int_{\Sigma}}\|\varphi_1(\cdot,\sigma)\|_{\ell^{\infty}(\Om_1)}\cdots\|\varphi_{k+1}(\cdot,\sigma)\|_{\ell^{\infty}(\Om_{k+1)}}\,\rho(d\sigma).
\]
Using $\underline{\int_{\Sigma}} \leq \overline{\int_{\Sigma}}$ and taking the infimum over all $\ell^{\infty}$-IPDs $(\Sigma,\rho,\varphi_1,\ldots,\varphi_{k+1})$ of $\varphi$ gives the desired result.
\end{proof}

\subsection{Divided differences and perturbation formulas}\label{sec.divdiffandpert}

Our goal is to differentiate operator functions in integral symmetrically normed ideals.
As is common practice, we begin by proving ``perturbation formulas."
To do so, we shall use a generalization of the argument from the proof of Theorem 1.2.3 in \cite{peller2}.
First, we review divided differences.

\begin{defi}[Divided differences]\label{def.divdiff}
Let $f \colon \R \to \C$ be a function.
Define $f^{[0]} \coloneqq f$ and, for $k \in \N$ and distinct $\lambda_1,\ldots,\lambda_{k+1} \in \R$, recursively define
\[
f^{[k]}(\lambda_1,\ldots,\lambda_{k+1}) \coloneqq \frac{f^{[k-1]}(\lambda_1,\ldots,\lambda_k) - f^{[k-1]}(\lambda_1,\ldots,\lambda_{k-1},\lambda_{k+1})}{\lambda_k-\lambda_{k+1}}.
\]
We call $f^{[k]}$ the $\boldsymbol{k^{\textbf{th}}}$ \textbf{divided difference} of $f$.
\end{defi}

It is easy to prove by induction that if $\lambda_1,\ldots,\lambda_{k+1} \in \R$ are distinct, then
\[
f^{[k]}(\lambda_1,\ldots,\lambda_{k+1}) = \sum_{j=1}^{k+1}\frac{f(\lambda_j)}{\prod_{\ell \neq j}(\lambda_j-\lambda_{\ell})}.
\]
In particular, $f^{[k]}$ is symmetric in its arguments.
As we shall see shortly, if in addition $f \in C^k(\R)$, then $f^{[k]}$ extends uniquely to a continuous function defined on all of $\R^{k+1}$.

\begin{nota}
For $m \in \N$, define
\begin{align*}
     \Sigma_m & \coloneqq \Bigg\{(s_1,\ldots,s_m) \in \R^m : s_j \geq 0 \text{ for } 1 \leq j \leq m \text{ and } \sum_{j=1}^m s_j \leq 1\Bigg\} \; \text{ and} \\
     \Delta_m & \coloneqq \Bigg\{(t_1,\ldots,t_{m+1}) \in \R^{m+1} : t_j \geq 0 \text{ for } 1 \leq j \leq m+1 \text{ and } \sum_{j=1}^{m+1}t_j = 1\Bigg\}.
\end{align*}
Denote by $\rho_m$ the pushforward of the restriction to $\Sigma_m$ of the $m$-dimensional Lebesgue measure by the homeomorphism $\Sigma_m \ni (s_1,\ldots,s_m) \mapsto (s_1,\ldots,s_m,1-s_1-\cdots-s_m) \in \Delta_m$.
\end{nota}

Explicitly, $\rho_m$ is the finite Borel measure on $\Delta_m$ such that
\[
\int_{\Delta_m} \varphi(\boldsymbol{t}) \, \rho_m(d\boldsymbol{t}) = \int_{\Sigma_m} \varphi(s_1,\ldots,s_m,1-s_1-\cdots-s_m) \, ds_1\cdots ds_m,
\]
for all $\varphi \in \ell^{\infty}(\Delta_m,\cB_{\Delta_m})$.
In particular, $\rho_m(\Delta_m) = \frac{1}{m!}$.
The following is proven using the Fundamental Theorem of Calculus and induction.

\begin{prop}\label{prop.divdiffCk}
If $k \in \N$, $f \in C^k(\R)$, and $\lambda_1,\ldots,\lambda_{k+1} \in \R$ are distinct, then
\[
f^{[k]}(\blambda) = \int_{\Delta_k}f^{(k)}(\boldsymbol{t} \cdot \blambda) \, \rho_k(d\boldsymbol{t}),
\]
where $\blambda \coloneqq (\lambda_1,\ldots,\lambda_{k+1})$ and $\cdot$ is the Euclidean dot product.
In particular, $f^{[k]}$ extends uniquely to a symmetric continuous function on all of $\R^{k+1}$.
We shall use the same notation for this extended function.
\end{prop}

With this under our belts, we move on to proving perturbation formulas.
For the rest of this section, fix a complex Hilbert space $(H,\la \cdot,\cdot \ra)$.

\begin{lem}\label{lem.approxbybdd}
Fix $k \in \N$ and $a_1,\ldots,a_{k+1} \in C(H)_{\sa}$.
For $j \in \{1,\ldots,k+1\}$ and $n \in \N$, define
\[
a_{j,n} \coloneqq a_jP^{a_j}([-n,n]) = \chi_n(a_j) \in B(H)_{\sa},
\]
where $\chi_n(t) \coloneqq t\,1_{[-n,n]}(t)$ for $t \in \R$.
If $\varphi \in \ell^{\infty}(\R,\cB_{\R})^{\iotimes(k+1)}$ and $(b_{\cdot,n})_{n \in \N} = (b_{1,n},\ldots,b_{k,n})_{n \in \N}$ is a sequence in $B(H)^k$ converging in the (product) SOT to $b = (b_1,\ldots,b_k) \in B(H)^k$, then
\[
\varphi(a_{1,n},\ldots,a_{k+1,n})\sh b_{\cdot,n} \to \varphi(a_1,\ldots,a_{k+1}) \sh b
\]
in the SOT as $n \to \infty$.
\end{lem}
\begin{proof}
First, fix $j \in \{1,\ldots,k+1\}$ and $n \in \N$.
If $f \in \ell^0(\R,\cB_{\R})$, then $f(a_{j,n}) = f(\chi_n(a_j)) = (f \circ \chi_n)(a_j)$.
Now, if $f$ is also bounded, then $\sup_{n \in \N} \|f\circ\chi_n\|_{\ell^{\infty}(\R)} \leq \|f\|_{\ell^{\infty}(\R)} < \infty$ and $f \circ \chi_n \to f \circ \id_{\R} = f$ pointwise as $n \to \infty$.
Therefore, by Proposition \ref{prop.integP}\ref{item.PintDCT}, $f(a_{j,n}) \to f(a_j)$ in the S$^*$OT as $n \to \infty$.

Next, let $(\Sigma,\rho,\varphi_1,\ldots,\varphi_{k+1})$ be a $\ell^{\infty}$-IPD of $\varphi$.
By definition,
\[
\varphi(a_{1,n},\ldots,a_{k+1,n})\sh b = \int_{\Sigma} \varphi_1(a_{1,n},\sigma)\,b_1\cdots\varphi_k(a_{k,n},\sigma)\,b_k\,\varphi_{k+1}(a_{k+1,n},\sigma)\, \rho(d\sigma)
\]
whenever $b = (b_1,\ldots,b_k) \in B(H)^k$.
By the previous paragraph's observations,
\[
\varphi_1(a_{1,n},\sigma)\,b_{1,n}\cdots\varphi_k(a_{k,n},\sigma)\,b_{k,n}\,\varphi_{k+1}(a_{k+1,n},\sigma) \to \varphi_1(a_1,\sigma)\,b_1\cdots\varphi_k(a_k,\sigma)\,b_k\,\varphi_{k+1}(a_{k+1},\sigma)\pagebreak
\]
in the SOT as $n \to \infty$, for all $\sigma \in \Sigma$.
Since
\begin{align*}
    & \overline{\int_{\Sigma}}\sup_{n \in \N} \|\varphi_1(a_{1,n},\sigma)\,b_{1,n}\cdots\varphi_k(a_{k,n},\sigma)\,b_{k,n}\,\varphi_{k+1}(a_{k+1,n},\sigma)\|\,\rho(d\sigma)\\
    & \; \; \; \; \; \; \; \; \; \; \; \leq \sup_{n\in \N}(\|b_{1,n}\|\cdots\|b_{k,n}\|)\overline{\int_{\Sigma}}\|\varphi_1(\cdot,\sigma)\|_{\ell^{\infty}(\R)}\cdots\|\varphi_{k+1}(\cdot,\sigma)\|_{\ell^{\infty}(\R)}\,\rho(d\sigma) < \infty,
\end{align*}
the desired result follows from Proposition \ref{prop.opDCT} and the definition of $\varphi(a_1,\ldots,a_{k+1})\sh b$.
\end{proof}

Before stating and proving our perturbation formulas, we make an observation that we shall use repeatedly.
If $f \colon \R \to \C$ is Lipschitz, then there are constants $C_1,C_2 \geq 0$ such that $|f(\lambda)| \leq C_1|\lambda|+C_2$, for all $\lambda \in \R$.
In particular, it follows from the definition of functional calculus and the Spectral Theorem that
\[
\dom(a) \subseteq \dom(f(a)), \numberthis\label{eq.dom}
\]
for all $a \in C(H)_{\sa}$.

\begin{nota}\label{nota.list}
Let $S$ be a set, $m \in \N$, and $s = (s_1,\ldots,s_m) \in S^m$.
If $j \in \{1,\ldots,m+1\}$, then
\[
s_{j-} \coloneqq (s_1,\ldots,s_{j-1}) \in S^{j-1} \; \text{ and } \; s_{j+} \coloneqq (s_j,\ldots,s_m) \in S^{m+1-j},
\]
where $s_{1-}$ and $s_{(m+1)+}$ are both the empty list.
\end{nota}

\begin{thm}[Perturbation formulas]\label{thm.perturb}
Fix $c \in B(H)_{\sa}$ and $a \in C(H)_{\sa}$.
If $f \in C^1(\R)$ is such that $f^{[1]} \in \ell^{\infty}(\R,\cB_{\R}) \iotimes \ell^{\infty}(\R,\cB_{\R})$, then
\[
f(a+c)-f(a) = f^{[1]}(a+c,a)\sh c. \numberthis\label{eq.perturb0}
\]
More precisely, $f(a+c)-f(a)$ is densely defined and bounded, and $f^{[1]}(a+c,a)\sh c$ is its unique bounded linear extension.
Now, suppose $k \geq 2$, $b=(b_1,\ldots,b_{k-1}) \in B(H)_{\sa}^{k-1}$, and $\vec{a} = (a_1,\ldots,a_{k-1}) \in C(H)_{\sa}^{k-1}$. If $f \in C^k(\R)$ is such that $f^{[k-1]} \in \ell^{\infty}(\R,\cB_{\R})^{\iotimes k}$ and $f^{[k]} \in \ell^{\infty}(\R,\cB_{\R})^{\iotimes (k+1)}$, then
\[
f^{[k-1]}\big(\vec{a}_{j-},a+c,\vec{a}_{j+}\big)\sh b - f^{[k-1]}\big(\vec{a}_{j-},a,\vec{a}_{j+}\big)\sh b = f^{[k]}\big(\vec{a}_{j-},a+c,a,\vec{a}_{j+}\big)\sh[b_{j-},c,b_{j+}], \numberthis\label{eq.perturbk}
\]
for all $j \in \{1,\ldots,k\}$.
\end{thm}
\begin{proof}
We first make an important observation.
Fix $a \in C(H)_{\sa}$, $c \in B(H)_{\sa}$, and $n \in \N$.
Now, define $p_n \coloneqq P^a([-n,n])$, $q_n \coloneqq P^{a+c}([-n,n])$, $a_n \coloneqq a\,p_n = \chi_n(a)$, and $d_n \coloneqq (a+c)\,q_n= \chi_n(a+c)$ in the notation of Lemma \ref{lem.approxbybdd}.
If $\psi_1,\psi_2 \in \ell^{\infty}(\R,\cB_{\R})$, then
\begin{align*}
    q_n\,\psi_1(d_n)(d_n-a_n)\psi_2(a_n)\,p_n & = 1_{[-n,n]}(a+c)\,(\psi_1 \circ \chi_n)(a+c)\,(d_n-a_n)\,(\psi_2 \circ \chi_n)(a)\,1_{[-n,n]}(a) \\
    & = ((\psi_1 \circ \chi_n)1_{[-n,n]})(a+c)\,(d_n-a_n)\,((\psi_2 \circ \chi_n)1_{[-n,n]})(a) \\
    & = (\psi_1 \,1_{[-n,n]})(a+c)\,(d_n-a_n)\,(\psi_2 \,1_{[-n,n]})(a) \\
    & = \psi_1(a+c)\,q_n\,(d_n-a_n)\,p_n\,\psi_2(a) = \psi_1(a+c)\,q_n\,c\,p_n\,\psi_2(a),
\end{align*}
where $q_n(d_n-a_n)p_n = q_nd_np_n - q_na_np_n = q_n(a+c)p_n - q_nap_n = q_ncp_n$ because $\im p_n \subseteq \dom(a) = \dom(a+c)$.

We now begin in earnest.
Let $f \in C^1(\R)$ be such that $f^{[1]} \in \ell^{\infty}(\R,\cB_{\R}) \iotimes \ell^{\infty}(\R,\cB_{\R})$. Then
\[
f(\lambda)-f(\mu) = f^{[1]}(\lambda,\mu)(\lambda-\mu), \numberthis\label{eq.keyperturbid1}
\]
for all $\lambda,\mu \in \R$.
Now, let $a,c \in B(H)_{\sa}$.
Since $\sigma(a)$ and $\sigma(a+c)$ are compact and $f \in C(\R)$, the functions
\begin{align*}
    \sigma(a+c)\times \sigma(a) \ni (\lambda,\mu) & \mapsto \psi(\lambda,\mu) = \lambda-\mu \in \C  \, \text{ and} \\
    \sigma(a+c)\times \sigma(a) \ni (\lambda,\mu) & \mapsto \varphi(\lambda,\mu) = f(\lambda)-f(\mu) \in \C
\end{align*}
belong to $\ell^{\infty}(\sigma(a+c),\cB_{\sigma(a+c)}) \iotimes \ell^{\infty}(\sigma(a),\cB_{\sigma(a)})$.
By Proposition \ref{prop.linandmult}\ref{item.mult2} and \eqref{eq.keyperturbid1}, we have
\[
I^{a+c,a}\varphi = \big(I^{a+c,a}f^{[1]}\big)\circ \big(I^{a+c,a}\psi\big).
\]
Applying this to the identity $1 = \id_H \in B(H)$, we conclude
\[
f(a+c) - f(a) = \big(I^{a+c,a}\varphi\big)[1] = \big(I^{a+c,a}f^{[1]}\big)\big[ \big(I^{a+c,a}\psi\big)[1]\big] = \big(I^{a+c,a}f^{[1]}\big)[c] = f^{[1]}(a+c,a)\sh c,
\]
as desired.
\pagebreak

For general $a \in C(H)_{\sa}$, we begin by showing $f(a+c)-f(a)$ is densely defined;
specifically, we show $\dom(a) \subseteq \dom(f(a+c)-f(a))$.
Indeed, since $f^{[1]} \in \ell^{\infty}(\R^2,\cB_{\R^2})$, $f$ is Lipschitz on $\R$.
By \eqref{eq.dom}, we have $\dom(a_0) \subseteq \dom(f(a_0))$, for all $a_0 \in C(H)_{\sa}$.
In particular, since $\dom(a) = \dom(a+c)$, we get
\[
\dom(a) = \dom(a) \cap \dom(a+c) \subseteq \dom(f(a+c))\cap \dom(f(a)) = \dom(f(a+c)-f(a)),
\]
as desired.
Next, let $p_n$, $q_n$, $a_n$, and $d_n$ be as in the first paragraph.
If $(\Sigma,\rho,\varphi_1,\varphi_2)$ is a $\ell^{\infty}$-IPD of $f^{[1]}$, then the results of the previous two paragraphs and Lemma \ref{lem.approxbybdd} give
\begin{align*}
    q_n(f(d_n)-f(a_n))p_n & = q_n\,f^{[1]}(d_n,a_n)\sh (d_n-a_n)\,p_n = \int_{\Sigma} q_n\,\varphi(d_n,\sigma)(d_n-a_n)\varphi_2(a_n,\sigma)\,p_n\,\rho(d\sigma) \\
    & = \int_{\Sigma} \varphi(a+c,\sigma)\,q_n\,c\,p_n\,\varphi_2(a,\sigma)\,\rho(d\sigma) = f^{[1]}(a+c,c)\sh[q_n\,c\,p_n] \to f^{[1]}(a+c,a)\sh c
\end{align*}
in the SOT as $n \to \infty$, since $q_n \to 1$ and $p_n \to 1$ in the SOT as $n \to \infty$.
But now, notice
\[
f(a_n)\,p_n = (f \circ \chi_n)(a)\,1_{[-n,n]}(a) = ((f \circ \chi_n)\,1_{[-n,n]})(a) = (f\,1_{[-n,n]})(a) = f(a)\,p_n
\]
and similarly $q_nf(d_n)p_n = q_nf(a+c)p_n$.
(For the latter, we use that $\im p_n \subseteq \dom(a) \subseteq \dom(f(a+c))$.)
It follows that if $m \in \N$, $h \in \im p_m$, and $n \geq m$, then
\begin{align*}
    q_n(f(d_n)-f(a_n))p_nh & = q_n(f(a+c)-f(a))p_nh \\
    & = q_n(f(a+c)-f(a))p_np_mh \\
    & = q_n(f(a+c)-f(a))p_mh \\
    & \to (f(a+c)-f(a))p_mh = (f(a+c)-f(a))h
\end{align*}
in $H$ as $n \to \infty$.
We have now proven that
\[
(f(a+c)-f(a))h = (f^{[1]}(a+c,c)\sh c)h
\]
for $h \in \im p_m$.
Since $\bigcup_{m \in \N} \im p_m \subseteq H$ is a dense linear subspace, we are done with the first part.

Next, let $k \geq 2$ and $f \in C^k(\R)$ be such that $f^{[k-1]} \in \ell^{\infty}(\R,\cB_{\R})^{\iotimes k}$ and $f^{[k]} \in \ell^{\infty}(\R,\cB_{\R})^{\iotimes (k+1)}$.
By definition and symmetry of divided differences, if $j \in \{1,\ldots,k\}$ and $\lambda,\mu \in \R$, then
\[
f^{[k-1]}\big(\vec{\lambda}_{j-},\lambda,\vec{\lambda}_{j+}\big) - f^{[k-1]}\big(\vec{\lambda}_{j-},\mu,\vec{\lambda}_{j+}\big) = f^{[k]}\big(\vec{\lambda}_{j-},\lambda,\mu,\vec{\lambda}_{j+}\big)(\lambda-\mu), \numberthis\label{eq.keyperturbid2}
\]
for $\vec{\lambda} = (\lambda_1,\ldots,\lambda_{k-1}) \in \R^{k-1}$.
Now, suppose $a,c \in B(H)_{\sa}$ again, and fix $\vec{a} = (a_1,\ldots,a_{k-1}) \in C(H)_{\sa}^{k-1}$ and $b = (b_1,\ldots,b_{k-1}) \in B(H)^{k-1}$.
Since $\sigma(a)$ and $\sigma(a+c)$ are compact and $f^{[k-1]} \in \ell^{\infty}(\R,\cB_{\R})^{\iotimes k}$, we have that both of the functions
\begin{align*}
    \R^{j-1} \times \sigma(a+c)\times \sigma(a) \times \R^{k-j} \ni (u,\lambda,\mu,v) & \overset{\psi}{\mapsto} \lambda-\mu \in \C  \, \text{ and} \\
    \R^{j-1} \times \sigma(a+c)\times \sigma(a) \times \R^{k-j} \ni (u,\lambda,\mu,v) & \overset{\varphi}{\mapsto}  f^{[k-1]}(u,\lambda,v) - f^{[k-1]}(u,\mu,v) \in \C
\end{align*}
belong to
\[
\ell^{\infty}(\R,\cB_{\R})^{\iotimes(j-1)}\iotimes \ell^{\infty}(\sigma(a+c),\cB_{\sigma(a+c)}) \iotimes \ell^{\infty}(\sigma(a),\cB_{\sigma(a)}) \iotimes \ell^{\infty}(\R,\cB_{\R})^{\iotimes(k-j)}.
\]
This allows us to apply $I^{\vec{a}_{j-},a+c,a,\vec{a}_{j+}}$ to \eqref{eq.keyperturbid2}, which may be rewritten
\[
\varphi = f^{[k]}\psi.
\]
If we do so and then plug $(b_{j-},1,b_{j+})$ into the result, then we get
\begin{align*}
    f^{[k-1]}\big(\vec{a}_{j-},a+c,\vec{a}_{j+}\big)\sh b - f^{[k-1]}\big(\vec{a}_{j-},a,\vec{a}_{j+}\big)\sh b & = \varphi\big(\vec{a}_{j-},a+c,a,\vec{a}_{j+}\big)\sh[b_{j-},1,b_{j+}] \\
    & = \big(f^{[k]}\psi\big)\big(\vec{a}_{j-},a+c,a,\vec{a}_{j+}\big)\sh[b_{j-},1,b_{j+}] \\
    & = f^{[k]}\big(\vec{a}_{j-},a+c,a,\vec{a}_{j+}\big)\sh[b_{j-},c,b_{j+}],
\end{align*}
where in the last line we used Proposition \ref{prop.linandmult}\ref{item.mult2} and the definition of $\psi$.
\pagebreak

Finally, for general $a \in C(H)_{\sa}$, let $p_n$, $q_n$, $a_n$, and $d_n$ be as in the first paragraph.
If $1 < j < k$, then we also let $b_{\cdot,n} \coloneqq (b_{(j-1)-},b_{j-1}q_n,p_nb_j,b_{(j+1)+})$.
Since $p_n \to 1$ and $q_n \to 1$ in the SOT, Lemma \ref{lem.approxbybdd} gives
\begin{align*}
    f^{[k-1]}\big(\vec{a}_{j-},d_n,\vec{a}_{j+}\big)\sh b_{\cdot,n} & \to f^{[k-1]}\big(\vec{a}_{j-},a+c,\vec{a}_{j+}\big)\sh b \; \text{ and} \\
    f^{[k-1]}\big(\vec{a}_{j-},a_n,\vec{a}_{j+}\big)\sh b_{\cdot,n} & \to f^{[k-1]}\big(\vec{a}_{j-},a,\vec{a}_{j+}\big)\sh b
\end{align*}
in the SOT as $n \to \infty$.
Now, let $(\Sigma,\rho,\varphi_1,\ldots,\varphi_{k+1})$ be a $\ell^{\infty}$-IPD of $f^{[k]}$.
Then
\begin{align*}
    T_{j,n} & \coloneqq f^{[k]}\big(\vec{a}_{j-},d_n,a_n,\vec{a}_{j+}\big)\sh[(b_{\cdot,n})_{j-},d_n-a_n,(b_{\cdot,n})_{j+}] \\
    & = \int_{\Sigma} \Bigg(\prod_{j_1=1}^{j-1}\varphi(a_{j_1},\sigma)\,b_{j_1}\Bigg)q_n\,\varphi_j(d_n,\sigma)(d_n-a_n)\varphi_{j+1}(a_n,\sigma)\,p_n\Bigg(\prod_{j_2=j}^{k-1}b_{j_2}\,\varphi(a_{j_2+2},\sigma)\Bigg)\,\rho(d\sigma) \\
    & = \int_{\Sigma} \Bigg(\prod_{j_1=1}^{j-1}\varphi(a_{j_1},\sigma)\,b_{j_1}\Bigg)\varphi_j(a+c,\sigma)\,q_n\,c\,p_n\,\varphi_{j+1}(a,\sigma)\Bigg(\prod_{j_2=j}^{k-1}b_{j_2}\,\varphi(a_{j_2+2},\sigma)\Bigg)\,\rho(d\sigma) \\
    & = f^{[k]}\big(\vec{a}_{j-},a+c,a,\vec{a}_{j+}\big)\sh[b_{j-},q_n\,c\,p_n,b_{j+}] \\
    & \to f^{[k]}\big(\vec{a}_{j-},a+c,a,\vec{a}_{j+}\big)\sh[b_{j-},c,b_{j+}]
\end{align*}
in the SOT as $n \to \infty$, by the observation from the first paragraph and Lemma \ref{lem.approxbybdd}.
(Above, empty products are declared to be $1$.)
Since we already know from the previous paragraph that
\[
f^{[k-1]}\big(\vec{a}_{j-},d_n,\vec{a}_{j+}\big)\sh b_{\cdot,n} -  f^{[k-1]}\big(\vec{a}_{j-},a_n,\vec{a}_{j+}\big)\sh b_{\cdot,n} = f^{[k]}\big(\vec{a}_{j-},d_n,a_n,\vec{a}_{j+}\big)\sh[(b_{\cdot,n})_{j-},d_n-a_n,(b_{\cdot,n})_{j+}],
\]
for all $n \in \N$, this completes the proof when $1 < j < k$.
For the cases $j \in \{1,k\}$, we redefine
\[
b_{\cdot,n} \coloneqq (p_nb_1,b_{2+}) \; \text{ and } \; \tilde{b}_{\cdot,n} \coloneqq (b_{(k-1)-},b_{k-1}q_n).
\]
Then we use an argument similar to the one above to see that
\begin{align*}
    q_n(f^{[k-1]}(d_n,\vec{a}) \sh b_{\cdot,n} - f^{[k-1]}(a_n,\vec{a}) \sh b_{\cdot,n}) & = q_n f^{[k]}(d_n,a_n,\vec{a})\sh [d_n-a_n,b_{\cdot,n}] \\
    & = f^{[k]}(a+c,a,\vec{a})\sh [q_n\,c\,p_n,b]
\end{align*}
and
\begin{align*}
    \big(f^{[k-1]}(\vec{a},d_n) \sh \tilde{b}_{\cdot,n} - f^{[k-1]}(\vec{a},a_n) \sh \tilde{b}_{\cdot,n}\big)p_n & = \big(f^{[k]}(\vec{a},d_n,a_n)\sh \big[\tilde{b}_{\cdot,n},d_n-a_n\big]\big)p_n \\
    & = f^{[k]}(\vec{a},a+c,a)\sh [b,q_n\,c\,p_n].
\end{align*}
Then we use Lemma \ref{lem.approxbybdd} to take $n \to \infty$.
This completes the proof.
\end{proof}

\begin{cor}\label{cor.Iperturb}
Let $\cM \subseteq B(H)$ be a von Neumann algebra and $a \aff \cM_{\sa}$.
Suppose $(\cI,\|\cdot\|_{\cI}) \unlhd \cM$ is MOI-friendly.
If $f \in C^1(\R)$ is such that $f^{[1]} \in \ell^{\infty}(\R,\cB_{\R}) \iotimes \ell^{\infty}(\R,\cB_{\R})$, then
\[
f(a+c)-f(a) \in \cI \; \text{ and } \; \|f(a+c)-f(a)\|_{\cI} \leq \big\|f^{[1]}\big\|_{\ell^{\infty}(\sigma(a+c),\cB_{\sigma(a+c)})\iotimes \ell^{\infty}(\sigma(a),\cB_{\sigma(a)})}\|c\|_{\cI}
\]
for all $c \in \cI_{\sa} \coloneqq \cI \cap \cM_{\sa}$, where $\cM_{\sa} = \{b \in \cM : b^*=b\}$.
\end{cor}
\begin{proof}
Since $a \aff \cM_{\sa}$ and $c \in \cI_{\sa} \subseteq \cM_{\sa}$, it is easy to see $a+c \aff \cM_{\sa}$ as well.
In particular, the projection-valued measures $P^{a+c}$ and $P^a$ take values in $\cM$.
It then follows from \eqref{eq.perturb0} and the definition of MOI-friendly that $f(a+c)-f(a) \in \cI$ and $\|f(a+c)-f(a)\|_{\cI} \leq \|f^{[1]}\|_{\ell^{\infty}(\sigma(a+c),\cB_{\sigma(a+c)})\iotimes \ell^{\infty}(\sigma(a),\cB_{\sigma(a)})}\|c\|_{\cI}$.
\end{proof}
\begin{rem}[Quasicommutators]
Let $f \in C^1(\R)$ be such that $f^{[1]} \in \ell^{\infty}(\R,\cB_{\R}) \iotimes \ell^{\infty}(\R,\cB_{\R})$.
One can show using essentially the same proofs that if $a,b \in C(H)_{\sa}$ and $q \in B(H)$ are such that $aq-qb \in B(H)$ (i.e., $aq-qb$ is densely defined and bounded), then $f(a)q-qf(b) \in B(H)$ and
\[
f(a)q-qf(b) = f^{[1]}(a,b)\sh [aq-qb].
\]
As a result, we get a quasicommutator estimate in MOI-friendly ideals.
Let $\cM \subseteq B(H)$ be a von Neumann algebra, and suppose $(\cI,\|\cdot\|_{\cI}) \unlhd \cM$ is MOI-friendly.
If $a,b \aff \cM_{\sa}$ and $q \in B(H)$ are such that $aq-qb \in \cI$, then $f(a)q-qf(b) \in \cI$ and $\|f(a)q-qf(b)\|_{\cI} \leq \|f^{[1]}\|_{\ell^{\infty}(\sigma(a),\cB_{\sigma(a)}) \iotimes \ell^{\infty}(\sigma(b),\cB_{\sigma(b)})}\|aq-qb\|_{\cI}$.
Such quasicommutator estimates are of interest in the study of \textit{operator Lipschitz functions}.
Please see \cite{aleksandrovOL} or \cite{peller2} for more information.
\end{rem}

\subsection{The spaces \texorpdfstring{$BOC(\R)^{\iotimes(k+1)}$}{} and \texorpdfstring{$OC^{[k]}(\R)$}{}}\label{sec.OCk}

In the following section, we prove a general result about derivatives of operator functions.
In this section, we introduce the functions whose operator functions we shall be differentiating.
Then we use Peller's work from \cite{peller1}, which we review in detail in Appendix \ref{sec.BesovPeller}, to give a large class of examples of such functions.

\begin{defi}[Operator continuity]\label{def.opcont}
Fix $f \in \ell^0(\R,\cB_{\R})$.
We say that $f$ is \textbf{operator continuous} if 
\begin{enumerate}[label=(\alph*),leftmargin=2\parindent]
    \item for every complex Hilbert space $H$, $a \in C(H)_{\sa}$, and $c \in B(H)_{\sa}$, the operator $f(a+c)-f(a)$ is densely defined and bounded;
    and
    \item for every complex Hilbert space $H$ and $a \in C(H)_{\sa}$, $f(a+c)-f(a) \to 0$ in $B(H)$ as $c \to 0$ in $B(H)_{\sa}$.
    (More precisely, for every $a \in C(H)_{\sa}$ and $\e > 0$, there is some $\delta > 0$ such that $\|f(a+c)-f(a)\| < \e$ whenever $c \in B(H)_{\sa}$ and $\|c\| < \delta$.)
\end{enumerate}
In this case, we write $f \in OC(\R)$.
If in addition $f$ is bounded, then we write $f \in BOC(\R)$.
\end{defi}

Taking $H = \C$ in the definition, it is clear that operator continuous functions are continuous.
Also, we observe that if $f,g \in BOC(\R)$, $H$ is a complex Hilbert space, $a \in C(H)_{\sa}$, and $c \in B(H)_{\sa}$, then
\[
(fg)(a+c)-(fg)(a) = (f(a+c)-f(a))g(a+c) + f(a)(g(a+c)-g(a)).
\]
But then
\[
\|(fg)(a+c)-(fg)(a)\|  \leq \|g\|_{\ell^{\infty}(\R)} \|f(a+c)-f(a)\| + \|f\|_{\ell^{\infty}(\R)}\|g(a+c)-g(a)\| \to 0
\]
as $c \to 0$ in $B(H)_{\sa}$. Thus $fg \in BOC(\R)$.
It is even easier to see $f+g \in BOC(\R)$ and $\overline{f} \in BOC(\R)$.

Next, if $(\Sigma,\sH)$ is a measurable space, $\psi \colon \R \times \Sigma \to \C$ is measurable, $\psi(\cdot,\sigma) \in C(\R)$ for $\sigma \in \Sigma$, then
\[
\|\psi(\cdot,\sigma)\|_{\ell^{\infty}(\R)} = \sup_{t \in \Q}|\psi(t,\sigma)|, \, \text{ for all } \sigma \in \Sigma.
\]
In particular, $\sigma \mapsto \|\psi(\cdot,\sigma)\|_{\ell^{\infty}(\R)}$ is measurable.
Thus the following definition makes sense (without needing to use upper or lower integrals).

\begin{defi}[Integral projective tensor products II]\label{def.BOCIPTP}
Let $\varphi \colon \R^{k+1} \to \C$ be a function.
A $\boldsymbol{BOC}$\textbf{-integral projective decomposition} (BOCIPD) of $\varphi$ is a choice $(\Sigma,\rho,\varphi_1,\ldots,\varphi_{k+1})$ of a $\sigma$-finite measure space $(\Sigma,\sH,\rho)$ and measurable functions $\varphi_1, \ldots,\varphi_{k+1} \colon \R \times \Sigma \to \C$ such that
\begin{enumerate}[label=(\alph*),leftmargin=2\parindent]
    \item $\varphi_j(\cdot,\sigma) \in BOC(\R)$, for all $j \in \{1,\ldots,k+1\}$ and $\sigma \in \Sigma$; \label{item.BOCnullset}
    \item $\int_{\Sigma} \|\varphi_1(\cdot,\sigma)\|_{\ell^{\infty}(\R)} \cdots \|\varphi_{k+1}(\cdot,\sigma)\|_{\ell^{\infty}(\R)} \, \rho(d\sigma) < \infty$; and\label{item.BOICintegofphis}
    \item $\varphi(\boldsymbol{\lambda}) = \int_{\Sigma} \varphi_1(\lambda_1,\sigma) \cdots \varphi_{k+1}(\lambda_{k+1},\sigma) \, \rho(d\sigma)$, for all $\boldsymbol{\lambda} = (\lambda_1,\ldots,\lambda_{k+1}) \in \R^{k+1}$.\label{item.BOCequality}
\end{enumerate}
Now, define
\begin{align*}
    \|\varphi\|_{BOC(\R)^{\iotimes(k+1)}} &\coloneqq \inf\Bigg\{\int_{\Sigma} \|\varphi_1(\cdot,\sigma)\|_{\ell^{\infty}(\R)} \cdots \|\varphi_{k+1}(\cdot,\sigma)\|_{\ell^{\infty}(\R)} \, \rho(d\sigma) : (\Sigma,\rho,\varphi_1,\ldots,\varphi_{k+1})\\
    & \; \; \; \; \; \; \; \; \; \; \; \; \; \; \; \; \; \; \; \; \; \; \; \; \; \; \;  \; \; \; \text{ is a } BOC\text{-integral projective decomposition of } \varphi\Bigg\},
\end{align*}
where $\inf \emptyset \coloneqq \infty$.
Finally, we define
\[
BOC(\R)^{\iotimes(k+1)} \coloneqq \big\{\varphi : \|\varphi\|_{BOC(\R)^{\iotimes(k+1)}} < \infty\big\}
\]
to be the \textbf{$\boldsymbol{(k+1)}^{\text{\textbf{st}}}$ integral projective tensor power of $\boldsymbol{BOC(\R)}$}.
\end{defi}

\begin{prop}\label{prop.BOCIPTP}
$BOC(\R)^{\iotimes(k+1)} \subseteq C(\R^{k+1}) \cap \ell^{\infty}(\R,\cB_{\R})^{\iotimes(k+1)}$ is a $\ast$-subalgebra, and
\[
\big(BOC(\R)^{\iotimes(k+1)},\|\cdot\|_{BOC(\R)^{\iotimes(k+1)}}\big)
\]
is a Banach $\ast$-algebra under pointwise operations.
\end{prop}
\pagebreak
\begin{proof}[Sketch of proof]
The containment $BOC(\R)^{\iotimes(k+1)} \subseteq C(\R^{k+1}) \cap \ell^{\infty}(\R,\cB_{\R})^{\iotimes(k+1)}$ follows from the definitions and an application of the (standard) Dominated Convergence Theorem.
The rest of the statement follows from the observation above that $BOC(\R)$ is a $\ast$-algebra and arguments similar to (but easier than) those in the proof of Proposition 4.1.4 in \cite{nikitopoulosMOI}.
\end{proof}

We now introduce the space of functions to which the main result of the following section applies.

\begin{nota}\label{nota.OCk}
For $k \in \N$ and $f \in C^k(\R)$, define
\[
[f]_{OC^{[k]}(\R)} \coloneqq \sum_{j=1}^k\big\|f^{[j]}\big\|_{BOC(\R)^{\iotimes(j+1)}} \in [0,\infty],
\]
and let $OC^{[k]}(\R) \coloneqq \{f \in C^k(\R) : [f]_{OC^{[k]}(\R)} < \infty\}$.
\end{nota}

Notice that if $f \in C^1(\R)$ and $[f]_{OC^1(\R)} = \|f^{[1]}\|_{BOC(\R) \iotimes BOC(\R)}= 0$, then $f^{[1]} \equiv 0$, so $f$ must be a constant.
In particular, $[\cdot]_{OC^{[k]}(\R)}$ is a seminorm but not quite a norm.
If we define
\[
\|f\|_{OC^{[k]},r} \coloneqq \|f\|_{\ell^{\infty}([-r,r])}+[f]_{OC^{[k]}(\R)}
\]
for $r > 0$, then it can be shown --- using standard arguments and Proposition \ref{prop.BOCIPTP} ---  that $OC^{[k]}(\R)$ is a Fr\'{e}chet space with the topology induced by the collection $\{\|\cdot\|_{OC^{[k]},r} : r > 0\}$ of seminorms.
One can even show that $OC^{[k]}(\R)$ is a $\ast$-algebra under pointwise operations, and that these operations are continuous. Since we shall not need these facts, we shall not dwell on them.
Instead, we turn to examples.

\begin{defi}[Wiener space]\label{def.Wk}
If $k \in \N$, then we define the \textbf{$\boldsymbol{k^{\text{\textbf{th}}}}$ Wiener space} to be the set of functions $f \colon \R \to \C$ such that there is a (unique) Borel complex measure $\mu$ on $\R$ satisfying $\int_{\R} |\xi|^k\,|\mu|(d\xi) < \infty$ and $f(\lambda) = \int_{\R} e^{i\lambda\xi} \,\mu(d\xi)$, for all $\lambda \in \R$.
\end{defi}

It is easy to see that if $1 \leq j \leq k$, then $W_k(\R) \subseteq W_j(\R)$ and $W_k(\R) \subseteq C^k(\R)$.
We now prove by elementary means that $W_k(\R) \subseteq OC^{[k]}(\R)$.
Then we use Peller's work from \cite{peller1} to generalize this substantially.

\begin{lem}\label{lem.expOC}
If $\xi \in \R$ and $f(\lambda) \coloneqq e^{i\lambda\xi}$, for all $\lambda \in \R$, then $f \in BOC(\R)$.
\end{lem}
\begin{proof}
Of course, $f$ is bounded and continuous.
Now, if $\lambda,\mu \in \R$, then
\[
f^{[1]}(\lambda,\mu) = \int_0^1f'(t\lambda+(1-t)\mu)\,dt = i\xi\int_0^1e^{it\lambda\xi}e^{i(1-t)\xi\mu}\,dt
\]
by Proposition \ref{prop.divdiffCk}.
This is clearly a $\ell^{\infty}$-integral projective decomposition of $f^{[1]}$ that yields
\[
\big\|f^{[1]}\big\|_{\ell^{\infty}(\R,\cB_{\R}) \iotimes \ell^{\infty}(\R,\cB_{\R})} \leq |\xi|.
\]
In particular, if $a \in C(H)_{\sa}$ and $c \in B(H)_{\sa}$, then $\|f(a+c)-f(a)\| \leq |\xi|\, \|c\|$ by Corollary \ref{cor.Iperturb}.
It follows that $f$ is operator continuous.
\end{proof}

\begin{prop}\label{prop.WkinOCk}
$W_k(\R) \subseteq OC^{[k]}(\R)$, for all $k \in \N$.
Specifically, if $f(\lambda) = \int_{\R} e^{i\lambda\xi}\,\mu(d\xi)$, where $\mu$ is a Borel complex measure on $\R$ with $\int_{\R}|\xi|^k\,|\mu|(d\xi) < \infty$, then
\[
[f]_{OC^{[k]}(\R)} \leq \sum_{j=1}^k\frac{1}{j!}\int_{\R}|\xi|^j\,|\mu|(d\xi).
\]
\end{prop}
\begin{proof}
Notice that if $f$ is as in the statement and $j \in \{1,\ldots,k\}$, then $f^{(j)}(\lambda) = \int_{\R} e^{i\lambda\xi}(i\xi)^j\,\mu(d\xi)$, for all $\lambda \in \R$.
Therefore, by Proposition \ref{prop.divdiffCk},
\[
f^{[j]}(\blambda) = \int_{\Delta_j} f^{(j)}(\boldsymbol{t} \cdot \blambda) \, \rho_j(d\boldsymbol{t}) =
\int_{\Delta_j}\int_{\R}e^{it_1\lambda_1\xi}\cdots e^{it_{j+1}\lambda_{j+1}\xi} \, (i\xi)^j\mu(d\xi) \, \rho_j(d \boldsymbol{t}).
\]
By Lemma \ref{lem.expOC}, this is (after writing $d\mu = \frac{d\mu}{d|\mu|}d|\mu|$ to match the definition) a BOCIPD of $f^{[j]}$ that yields
\[
\big\|f^{[j]}\big\|_{BOC(\R)^{\iotimes(j+1)}} \leq \int_{\Delta_j}\int_{\R} |\xi|^j\,|\mu|(d\xi)\,\rho_j(d\boldsymbol{t}) = \rho_j(\Delta_j)\int_{\R} |\xi|^j\,|\mu|(d\xi) = \frac{1}{j!}\int_{\R}|\xi|^j\,|\mu|(d\xi).
\]
Summing over $j \in \{1,\ldots,k\}$ gives the desired bound.
\end{proof}
\pagebreak
\begin{rem}
For similar reasons, if $f \in C^k(\R)$ and, for all $j \in \{1,\ldots,k\}$, $f^{(j)}$ and the Fourier transform of $f^{(j)}$ belong to $L^1(\R)$, then $f \in OC^{[k]}(\R)$.
\end{rem}

Now, we use more serious harmonic analysis done by Peller \cite{peller1} to exhibit a large class --- containing $W_k(\R)$ strictly --- of functions belonging to $OC^{[k]}(\R)$.
We begin by defining Besov spaces.

\begin{nota}\label{nota.fourier}
If $m \in \N$, then we write $\mathscr{S}(\R^m)$ for the Fr\'{e}chet space of Schwartz functions $\R^m \to \C$ and $\mathscr{S}'(\R^m) \coloneqq \mathscr{S}(\R^m)^*$ for the space of tempered distributions on $\R^m$.
Also, the conventions we use for the Fourier transform and its inverse are, respectively,
\[
\wh{f}(\xi) = (\cF f)(\xi) = \int_{\R^m} e^{-i\xi\cdot x}f(x)\,dx \; \text{ and } \; \wch{f}(x) = \frac{1}{(2\pi)^m}\int_{\R^m}e^{ix\cdot\xi}f(\xi)\,d\xi
\]
for $f \in L^1(\R)$, with corresponding extensions to $\mathscr{S}'(\R^m)$.
\end{nota}

\begin{defi}[Homogeneous Besov spaces]\label{def.Besov}
Fix $m \in \N$ and $\varphi \in C_c^{\infty}(\R^m)$ such that $0 \leq \varphi \leq 1$, $\supp \varphi \subseteq \{\xi \in \R^m : \|\xi\| \leq 2\}$, and $\varphi \equiv 1$ on $\{\xi \in \R^m : \|\xi\| \leq 1\}$.
For $j \in \Z$, define
\[
\varphi_j(\xi) \coloneqq \varphi(2^{-j}\xi) - \varphi(2^{-j+1}\xi),
\]
for all $\xi \in \R^m$.
Now, for $s \in \R$, $p,q \in [1,\infty]$, and $f \in \mathscr{S}'(\R^m)$, write
\[
\|f\|_{\dot{B}_q^{s,p}} \coloneqq \big\|\big(2^{js}\|\wch{\varphi}_j \ast f\|_{L^p}\big)_{j \in \Z}  \big\|_{\ell^q(\Z)} \in [0,\infty].
\]
We call $\dot{B}_q^{s,p}(\R^m) \coloneqq \{f \in \mathscr{S}'(\R^m) : \|f\|_{\dot{B}_q^{s,p}} < \infty\}$ the \textbf{homogeneous $\boldsymbol{(s,p,q)}$-Besov space}.
\end{defi}
\begin{rem}
First, note that $\wch{\varphi} \ast f, \wch{\varphi}_j \ast f$ have compactly supported Fourier transforms and so are smooth by the Paley--Wiener Theorem;
it therefore makes sense to apply the $L^p$-norm to them.
Second, since it is easy to show that $\|f\|_{\dot{B}_q^{s,p}} = 0$ if and only if $f$ is a polynomial, it is usually best to define $\dot{B}_q^{s,p}(\R^m)$ as a quotient space in which all polynomials are zero.
The definition above is given in Chapter 3 of \cite{peetre} and Sections 5.1.2 and 5.1.3 of \cite{triebel1}.
The definition ``modulo polynomials" is given in Section 2.4 of \cite{sawano}.
(Please see Section 1.2.5.3 of \cite{sawano} as well.)
Finally, beware that the positions of $p$ and $q$ in $\dot{B}_q^{s,p}(\R^m)$ are far from consistent in the literature.
\end{rem}

The case of interest is $m=1$ and $(s,p,q) = (k,\infty,1)$ for $k \in \N$.
As we show in Section \ref{sec.PellerII}, in this case it turns out $\dot{B}_1^{k,\infty}(\R) \subseteq C^k(\R)$.
Also, as mentioned above, if $f \in \C[\lambda]$ is a polynomial, then
\[
\|f\|_{\dot{B}_1^{k,\infty}} = 0.
\]
Therefore, if we are to prove sensible results about differentiating the operator function $c \mapsto f(a+c)-f(a)$ when $a$ is \textit{unbounded} and $f \in \dot{B}_1^{k,\infty}(\R)$, it is necessary to impose additional restrictions that exclude (at least) polynomials of degree higher than two.
We accomplish this with the following modified Besov spaces.

\begin{defi}[Peller--Besov spaces]\label{def.PellerBesov}
If $k \in \N$, then we define
\[
PB^k(\R) \coloneqq \dot{B}_1^{k,\infty}(\R) \cap \big\{f \in C^k(\R) : f^{(k)} \text{ is bounded}\big\}
\]
to be the $\boldsymbol{k^{\text{\textbf{th}}}}$ \textbf{Peller--Besov space}.
\end{defi}

The following result is a slight upgrade of Theorem 5.5 in \cite{peller1} or Theorem 2.2.1 in \cite{peller2}.

\begin{thm}[Peller \cite{peller1}]\label{thm.Peller}
If $k \in \N$, then there is a constant $c_k < \infty$ such that
\[
\big\|f^{[k]}\big\|_{BOC(\R)^{\iotimes(k+1)}} \leq \frac{1}{k!}\inf_{x \in \R}\big|f^{(k)}(x)\big| + c_k\|f\|_{\dot{B}_1^{k,\infty}},
\]
for all $f \in PB^k(\R)$;
and if $k \geq 2$, then
\[
\big\|f^{[k]}\big\|_{BOC(\R)^{\iotimes(k+1)}} \leq  c_k\|f\|_{\dot{B}_1^{k,\infty}},
\]
for all $f \in PB^1(\R) \cap \dot{B}_1^{k,\infty}(\R) = PB^1(\R) \cap PB^k(\R) = \bigcap_{j=1}^kPB^j(\R)$.
\end{thm}

The proof given in \cite{peller1} is not very detailed and is only explicit in the cases $k \in \{1,2\}$, so we present a full proof of this theorem in Appendix \ref{sec.BesovPeller}.
As a result, we obtain the following.
\pagebreak

\begin{cor}\label{cor.PBkinOCk}
$PB^1(\R) \cap PB^k(\R) = PB^1(\R) \cap \dot{B}_1^{k,\infty}(\R) \subseteq OC^{[k]}(\R)$, for all $k \in \N$.
Specifically,
\[
[f]_{OC^{[k]}(\R)} \leq \inf_{x \in \R}|f'(x)| + \sum_{j=1}^kc_j\|f\|_{\dot{B}^{j,\infty}},
\]
for all $f \in PB^1(\R) \cap PB^k(\R)$, where $c_1,\ldots,c_k$ are as in Theorem \ref{thm.Peller}. \qed
\end{cor}

Since it is easy to show that $W_k(\R) \subsetneq PB^1(\R) \cap \dot{B}_1^{k,\infty}(\R)$, for all $k \in \N$, Corollary \ref{cor.PBkinOCk} does in fact generalize Proposition \ref{prop.WkinOCk}.

\subsection{Derivatives of operator functions in ideals}

In this section, we finally differentiate operator functions in integral symmetrically normed ideals.
For the duration of this section, fix a complex Hilbert space $(H,\la \cdot,\cdot \ra)$, a von Neumann algebra $\cM \subseteq B(H)$, and $(\cI,\|\cdot\|_{\cI}) \unlhd \cM$.
Also, write $\cI_{\sa} \coloneqq \cI \cap \cM_{\sa} = \{b \in \cI : b^*=b\}$.

As a consequence of the definition of a Banach ideal, $\cI_{\sa}$ is a real Banach space when it is given (the restriction of) the $\cI$-norm $\|\cdot\|_{\cI}$.
Now, before setting up the main result of this section, we prove a key technical lemma that is the main reason \textit{integral} symmetrically normed ideals are considered in this paper.

\begin{nota}[Bounded multilinear maps]\label{nota.bddmultilin}
Let $(V_1,\|\cdot\|_{V_1}),\ldots,(V_k,\|\cdot\|_{V_k}),(W,\|\cdot\|_W)$ be normed vector spaces over $\F \in \{\R,\C\}$ and $T \colon V_1 \times \cdots \times V_k \to W$ be a $k$-linear map.
Then we write
\[
\|T\|_{B_k(V_1 \times \cdots \times V_k;W)} \coloneqq \sup\{\|T(v_1,\ldots,v_k)\|_W : v_j \in V_j, \; \|v_j\|_{V_j}\leq 1, \; 1 \leq j \leq k\}
\]
for the operator norm of $T$ and $B_k(V_1 \times \cdots \times V_k;W)$ for the space of $k$-linear maps $V_1 \times \cdots \times V_k \to W$ with finite operator norm.
\end{nota}

\begin{lem}[Continuous Perturbation]\label{lem.Icont}
If $\cI$ is integral symmetrically normed, $a_1,\ldots,a_{k+1} \aff \cM_{\sa}$, and $\varphi \in BOC(\R)^{\iotimes(k+1)}$, then the map
\[
\cI_{\sa}^{k+1} \ni (c_1,\ldots,c_{k+1}) \mapsto I^{a_1+c_1,\ldots,a_{k+1}+c_{k+1}}\varphi \in B_k(\cI^k;\cI)
\]
is continuous.
(To be clear, $\cI$ and $\cI_{\sa}$ are always endowed with the $\cI$-norm $\|\cdot\|_{\cI}$.)
\end{lem}
\begin{rem}
Recall from Proposition \ref{prop.idealproperties} that integral symmetrically normed ideals are MOI-friendly.
In particular, the map under consideration in Lemma \ref{lem.Icont} does actually make sense by definition of MOI-friendly and the fact that $a+c \aff \cM_{\sa}$ whenever $a \aff \cM_{\sa}$ and $c \in \cI_{\sa}$.
(As in the proof of Corollary \ref{cor.Iperturb}, the latter imply that $P^a$ and $P^{a+c}$ take values in $\cM$.)
\end{rem}
\begin{proof}
Write $\varphi_a \colon \cI_{\sa}^{k+1} \to B_k(\cI^k;\cI)$ for the map in question.
Now, let $c = (c_1,\ldots,c_{k+1}) \in \cI_{\sa}^{k+1}$ and $(c_{\cdot,n})_{n \in \N} = (c_{1,n},\ldots,c_{k+1,n})_{n \in \N}$ be a sequence in $\cI_{\sa}^{k+1}$ converging to $c$.
Then
\[
\varphi_a(c_{\cdot,n}) - \varphi_a(c) = \sum_{j=1}^{k+1} (\underbrace{\varphi_a(c_{1,n},\ldots,c_{j,n},c_{j+1},\ldots,c_{k+1}) - \varphi_a(c_{1,n},\ldots,c_{j-1,n},c_j,\ldots,c_{k+1})}_{\coloneqq T_{j,n}}).
\]
Fix now a BOCIPD $(\Sigma,\rho,\varphi_1,\ldots,\varphi_{k+1})$ of $\varphi$ and $b_1,\ldots,b_k \in \cI$, and write $b_{k+1} \coloneqq 1$.
By definition of the multiple operator integral, $T_{j,n}(b_1,\ldots,b_k)$ is precisely
\[
\int_{\Sigma}\Bigg(\prod_{m=1}^{j-1}\varphi_m(a_m+c_{m,n},\sigma)\,b_m\Bigg)(\varphi_j(a_j+c_{j,n},\sigma) - \varphi_j(a_j+c_j,\sigma))\,b_j\Bigg(\prod_{m=j+1}^{k+1}\varphi_m(a_m+c_m,\sigma)\,b_m\Bigg)\,\rho(d\sigma),
\]
where empty products are the identity.
Now, if $1 \leq j < k+1$ and
\begin{align*}
    A_n(\sigma) & \coloneqq \Bigg(\prod_{m=1}^{j-1}\varphi_m(a_m+c_{m,n},\sigma)\,b_m\Bigg)(\varphi_j(a_j+c_{j,n},\sigma) - \varphi_j(a_j+c_j,\sigma)) \; \text{ and} \\
    B(\sigma) & \coloneqq \prod_{m=j+1}^{k+1}\varphi_m(a_m+c_m,\sigma))\,b_m,
\end{align*}
then
\[
T_{j,n}(b_1,\ldots,b_k) = \int_{\Sigma}A_n(\sigma)\,b_j\,B(\sigma)\,\rho(d\sigma).
\]
But
\begin{align*}
    \underline{\int_{\Sigma}} \|A_n\|\,\|B\|\,d\rho & \leq \prod_{p \neq j}\|b_p\|\underline{\int_{\Sigma}} \big\|\varphi_j(a_j+c_{j,n},\sigma) - \varphi_j(a_j+c_j,\sigma)\big\|\prod_{m \neq j}\|\varphi_m(\cdot,\sigma)\|_{\ell^{\infty}(\R)}\,\rho(d\sigma) < \infty.
\end{align*}
Therefore, the definition of integral symmetrically normed gives $T_{j,n}(b_1,\ldots,b_k) \in \cI$ and
\begin{align*}
    \|T_{j,n}(b_1,\ldots,b_k)\|_{\cI} & \leq \|b_j\|_{\cI}\prod_{p \neq j}\|b_p\|\underline{\int_{\Sigma}} \big\|\varphi_j(a_j+c_{j,n},\sigma) - \varphi_j(a_j+c_j,\sigma)\big\|\prod_{m \neq j}\|\varphi_m(\cdot,\sigma)\|_{\ell^{\infty}(\R)}\,\rho(d\sigma) \\
    & \leq C_{\cI}^{k-1}\|b_1\|_{\cI}\cdots\|b_k\|_{\cI}\underline{\int_{\Sigma}} \big\|\varphi_j(a_j+c_{j,n},\sigma) - \varphi_j(a_j+c_j,\sigma)\big\|\prod_{m \neq j}\|\varphi_m(\cdot,\sigma)\|_{\ell^{\infty}(\R)}\,\rho(d\sigma).
\end{align*}
Thus
\[
\|T_{j,n}\|_{B_k(\cI^k;\cI)} \leq C_{\cI}^{k-1}\underline{\int_{\Sigma}} \big\|\varphi_j(a_j+c_{j,n},\sigma) - \varphi_j(a_j+c_j,\sigma)\|\prod_{m \neq j}\|\varphi_m(\cdot,\sigma)\|_{\ell^{\infty}(\R)}\,\rho(d\sigma). \numberthis\label{eq.Tjn}
\]
Next, fix $\sigma \in \Sigma$.
Since $\|c_{j,n}-c_j\| \leq C_{\cI}\|c_{j,n}-c_j\|_{\cI} \to 0$ as $n \to \infty$, the operator continuity of $\varphi_j(\cdot,\sigma)$ gives
\[
\|\varphi_j(a_j+c_{j,n},\sigma) - \varphi_j(a_j+c_j,\sigma)\| \to 0
\]
as $n \to \infty$.
Since
\[
\overline{\int_{\Sigma}}\sup_{n \in \N} \Bigg(\big\|\varphi_j(a_j+c_{j,n},\sigma) - \varphi_j(a_j+c_j,\sigma)\big\|\prod_{m \neq j}\|\varphi_m(\cdot,\sigma)\|_{\ell^{\infty}(\R)}\Bigg)\,\rho(d\sigma) \leq 2\int_{\Sigma}\prod_{m =1}^{k+1}\|\varphi_m(\cdot,\sigma)\|_{\ell^{\infty}(\R)}\,\rho(d\sigma),
\]
which is finite, we conclude from \eqref{eq.Tjn} and Lemma \ref{lem.poorDCT} that $\|T_{j,n}\|_{B_k(\cI^k;\cI)} \to 0$ as $n \to \infty$.
If $j=k+1$, then we run the same argument with
\begin{align*}
    A_n(\sigma) & \coloneqq \Bigg(\prod_{m=1}^{k-1}\varphi_m(a_m+c_{m,n},\sigma)\,b_m\Bigg)\,\varphi_k(a_k+c_{k,n},\sigma) \; \text{ and} \\
    B_n(\sigma) & \coloneqq \varphi_{k+1}(a_{k+1}+c_{k+1,n},\sigma) - \varphi_{k+1}(a_{k+1}+c_{k+1},\sigma)
\end{align*}
to prove that $\|T_{k+1,n}\|_{B_k(\cI^k;\cI)} \to 0$ as $n \to \infty$.
We conclude that
\[
\|\varphi_a(c_{\cdot,n}) - \varphi_a(c)\|_{B_k(\cI^k;\cI)} \leq \sum_{j=1}^{k+1} \|T_{j,n}\|_{B_k(\cI^k;\cI)} \to 0
\]
as $n \to \infty$, as claimed.
\end{proof}

Next, we recall the notion of Fr\'{e}chet differentiability of maps between normed vector spaces and then define what it means for a scalar function to be $\cI$\textit{-differentiable}.
For these purposes, note that if $V_1,\ldots,V_k,W$ are normed vector spaces, then $B_k(V_1 \times \cdots \times V_k ; W) \cong B(V_1;B_{k-1}(V_2 \times \cdots \times V_k;W))$ isometrically via
\[
T \mapsto (v_1 \mapsto ((v_2,\ldots,v_k) \mapsto T(v_1,\ldots,v_k))).
\]
We use this identification below.

\begin{defi}[Fr\'{e}chet differentiability]\label{def.frechder}
Let $V$ and $W$ be normed vector spaces, $U \subseteq V$ be open, and $F \colon U \to W$ be a map.
For $p \in U$, we say $F$ is \textbf{Fr\'{e}chet differentiable at} $\boldsymbol{p}$ if there exists (necessarily unique) $DF(p) \in B(V;W)$ such that
\[
\frac{\|F(p+h)-F(p)-DF(p)h\|_W}{\|h\|_V} \to 0
\]
as $h \to 0$ in $V$.
If $F$ is Fr\'{e}chet differentiable at all $p \in U$, then we say $F$ is \textbf{Fr\'{e}chet differentiable in} $\boldsymbol{U}$ and write $D^1F = DF \colon U \to B(V;W)$ for its \textbf{Fr\'{e}chet derivative} map $U \ni p \mapsto DF(p) \in B(V;W)$.
For $k\geq 2$, we say $F$ is $\boldsymbol{k}$\textbf{-times Fr\'{e}chet differentiable at} $\boldsymbol{p}$ if it is $(k-1)$-times Fr\'{e}chet differentiable in a\pagebreak\ neighborhood of $p$ --- say $U$ for simplicity --- and $D^{k-1}F \colon U \to B_{k-1}(V^{k-1};W)$ is Fr\'{e}chet differentiable at $p$.
In this case, we write
\[
D^kF(p) \coloneqq D(D^{k-1}F)(p) \in  B(V;B_{k-1}(V^{k-1};W)) \cong B_k(V^k;W).
\]
If $F$ is $k$-times Fr\'{e}chet differentiable at all $p \in U$, then we say $F$ is $\boldsymbol{k}$\textbf{-times Fr\'{e}chet differentiable in} $\boldsymbol{U}$ and write $D^kF \colon U \to B_k(V^k;W)$ for its $\boldsymbol{k^{\text{\textbf{th}}}}$ \textbf{Fr\'{e}chet derivative} map $U \ni p \mapsto D^kF(p) \in B_k(V^k;W)$.
Finally, if $D^kF$ is also continuous, then we say $F$ is $\boldsymbol{k}$\textbf{-times continuously differentiable in} $\boldsymbol{U}$ and write $F \in C^k(U;W)$.
\end{defi}

Concretely, if $F \colon U \to W$ is $k$-times Fr\'{e}chet differentiable (in $U$), then one can show by induction that
\[
D^kF(p)[h_1,\ldots,h_k] = \partial_{h_1}\cdots \partial_{h_k}F(p) = \frac{d}{ds_1}\Big|_{s_1=0}\cdots \frac{d}{ds_k}\Big|_{s_k=0} F(p+s_1h_1+\cdots+s_kh_k),
\]
for all $p \in U$ and $h_1,\ldots,h_k \in V$.

\begin{defi}[$\cI$-differentiability]\label{def.Idiff}
Fix $a \aff \cM_{\sa}$.
A Borel measurable function $f \colon \R \to \C$ is called \textbf{$\boldsymbol{k}$-times (Fr\'{e}chet) $\boldsymbol{\cI}$-differentiable at $\boldsymbol{a}$} if there is an open set $U \subseteq \cI_{\sa}$ with $0 \in U$ such that
\begin{enumerate}[label=(\alph*),leftmargin=2\parindent]
    \item $f(a+b)-f(a) \in \cI$ for all $b \in U$ (i.e., when $b \in U$, $f(a+b)-f(a)$ is densely defined and bounded, and its unique bounded linear extension belongs to $\cI$), and
    \item the map $U \ni b \mapsto f_{a,\mathsmaller{\cI}}(b) \coloneqq f(a+b)-f(a) \in \cI$ is $k$-times Fr\'{e}chet differentiable (with respect to $\|\cdot\|_{\cI}$) at $0 \in U \subseteq \cI_{\sa}$.
\end{enumerate}
In this case, we write
\[
D_{\mathsmaller{\cI}}^kf(a) \coloneqq D^kf_{a,\mathsmaller{\cI}}(0) \in B_k(\cI_{\sa}^k;\cI)
\]
for the $k^{\text{th}}$ Fr\'{e}chet derivative of $f_{a,\mathsmaller{\cI}} \colon U \to \cI$ at $0 \in U$.
If $f$ is $k$-times $\cI$-differentiable at $a$ for every $a \aff \cM_{\sa}$, then we simply say $f$ is \textbf{$\boldsymbol{k}$-times $\boldsymbol{\cI}$-differentiable}.
\end{defi}

Suppose that $f \colon \R \to \C$ is Lipschitz and $f(a+c)-f(a) \in \cI$, for all $a \aff \cM_{\sa}$ and $c \in \cI_{\sa}$ (i.e., $f_{a,\mathsmaller{\cI}} \colon \cI_{\sa} \to \cI$ is defined everywhere).
We claim that if $f$ is $k$-times $\cI$-differentiable, then $f_{a,\mathsmaller{\cI}}$ is $k$-times Fr\'{e}chet differentiable everywhere --- not just at $0 \in \cI_{\sa}$.
Indeed, fix $b,c \in \cI_{\sa}$, and note that
\[
f_{a,\mathsmaller{\cI}}(b+c)-f_{a,\mathsmaller{\cI}}(b) = f(a+b+c)-f(a+b) = f_{a+b,\mathsmaller{\cI}}(c). \numberthis\label{eq.domissues}
\]
This is the case because \eqref{eq.domissues} is immediate from the definition on
\[
\dom(a) = \dom(a)\cap \dom(a+b+c)\cap \dom(a+b) \subseteq \dom(f(a)) \cap \dom(f(a+b+c)) \cap \dom(f(a+b)),
\]
which is dense in $H$.
(Note that we used \eqref{eq.dom}.)
In other words,
\[
f_{a,\mathsmaller{\cI}}(b+c) = f_{a+b,\mathsmaller{\cI}}(c) + f_{a,\mathsmaller{\cI}}(b),
\]
for all $c \in \cI_{\sa}$.
Since $c \mapsto f_{a+b,\mathsmaller{\cI}}(c)$ is $k$-times differentiable at $0 \in \cI_{\sa}$, we conclude that $f_{a,\mathsmaller{\cI}}$ is $k$-times differentiable at $b$ with
\[
D^kf_{a,\mathsmaller{\cI}}(b) = D^kf_{a+b,\mathsmaller{\cI}}(0) = D_{\mathsmaller{\cI}}^kf(a+b). 
\]
With this in mind, here is the main result of this section.

\begin{thm}[Derivatives of operator functions in ISNIs]\label{thm.deropfunc}
Suppose $(\cI,\|\cdot\|_{\cI}) \unlhd \cM$ is integral symmetrically normed, and fix $a \aff \cM_{\sa}$.
If $f \in OC^{[k]}(\R)$, then $f_{a,\mathsmaller{\cI}} \colon \cI_{\sa} \to \cI$ is defined everywhere, and $f_{a,\mathsmaller{\cI}} \in C^k(\cI_{\sa};\cI)$.
In particular, $f$ is $k$-times $\cI$-differentiable.
Moreover,
\[
D_{\mathsmaller{\cI}}^kf(a)[b_1,\ldots,b_k] = \sum_{\pi \in S_k}f^{[k]}(\underbrace{a,\ldots,a}_{\mathsmaller{k+1 \; \mathrm{times}}})\sh [b_{\pi(1)},\ldots,b_{\pi(k)}],
\]
for all $(b_1,\ldots,b_k) \in \cI_{\sa}^k$.

By the observation above, we therefore also have
\[
D^kf_{a,\mathsmaller{\cI}}(c)[b_1,\ldots,b_k] = D_{\mathsmaller{\cI}}^kf(a+c)[b_1,\ldots,b_k] = \sum_{\pi \in S_k}f^{[k]}(\underbrace{a+c,\ldots,a+c}_{k+1 \; \mathrm{times}})\sh [b_{\pi(1)},\ldots,b_{\pi(k)}],
\]
for all $c,b_1,\ldots,b_k \in \cI_{\sa}$.
\end{thm}
\begin{proof}
Fix $a \aff \cM_{\sa}$.
Notice that if $f \in OC^{[k]}(\R)$, then $f \in C^1(\R)$ and $f^{[1]} \in \ell^{\infty}(\R,\cB_{\R}) \iotimes \ell^{\infty}(\R,\cB_{\R})$.
In particular, by Corollary \ref{cor.Iperturb}, $f_{a,\mathsmaller{\cI}}(c) = f(a+c)-f(a) \in \cI$, for all $c \in \cI_{\sa}$.
In addition, we observe that if $f \in OC^{[k]}(\R)$, then the map
\[
\cI_{\sa} \ni c \mapsto I^{a+c,\ldots,a+c}f^{[k]} \in B_k(\cI^k;\cI)
\]
is continuous by the Continuous Perturbation Lemma (Lemma \ref{lem.Icont}).
Therefore, the claimed $k^{\text{th}}$ derivative map is, in fact, continuous.
Thus, to prove the theorem, it suffices to prove the claimed formula for $D_{\mathsmaller{\cI}}^kf(a)$.
We do so by induction on $k$.

Fix $c \in \cI_{\sa}$.
By Theorem \ref{thm.perturb},
\begin{align*}
    f_{a,\mathsmaller{\cI}}(c)-f_{a,\mathsmaller{\cI}}(0)-f^{[1]}(a,a)\sh c & = f(a+c)-f(a) - f^{[1]}(a,a)\sh c \\
    & = f^{[1]}(a+c,a)\sh c - f^{[1]}(a,a)\sh c \\
    & = \big(I^{a+c,a}f^{[1]}-I^{a,a}f^{[1]}\big)[c].
\end{align*}
Therefore, by the Continuous Perturbation Lemma (Lemma \ref{lem.Icont}),
\[
\frac{1}{\|c\|_{\cI}}\big\|f_{a,\mathsmaller{\cI}}(c)-f_{a,\mathsmaller{\cI}}(0)-f^{[1]}(a,a)\sh c\big\|_{\cI} \leq \big\|I^{a+c,a}f^{[1]}-I^{a,a}f^{[1]}\big\|_{B(\cI)} \to 0
\]
as $c \to 0$ in $\cI_{\sa}$.
This completes the proof when $k=1$.

Next, suppose $k \geq 2$ and that we have proven the claimed derivative formula when $f \in OC^{[k-1]}(\R)$.
To prove the formula for $f \in OC^{[k]}(\R)$, we set some notation and make some preliminary observations.
If $S$ is a set, $s \in S$, and $m \in \N_0$, then we write
\[
s_{(m)} \coloneqq (s,\ldots,s) \in S^m,
\]
where $s_{(0)}$ is the empty list.
Now, fix $b = (b_1,\ldots,b_{k-1}) \in \cI^{k-1}$ and $f \in OC^{[k]}(\R) \subseteq OC^{[k-1]}(\R)$.
By Theorem \ref{thm.perturb}, we have
\begin{align*}
    \delta(b,c) & \coloneqq f^{[k-1]}\big((a+c)_{(k)}\big)\sh b - f^{[k-1]}\big(a_{(k)}\big)\sh b \\
    & = \sum_{j=1}^k \big(f^{[k-1]}\big((a+c)_{(j)},a_{(k-j)}\big)\sh b - f^{[k-1]}\big((a+c)_{(j-1)},a_{(k-j+1)}\big)\sh b\big) \\
    & = \sum_{j=1}^k f^{[k]}\big((a+c)_{(j)},a_{(k+1-j)}\big)\sh [b_{j-},c,b_{j+}],
\end{align*}
using Notation \ref{nota.list}.
Next, by the inductive hypothesis, 
\[
D^{k-1}f_{a,\mathsmaller{\cI}}(c_0)[b] = D_{\mathsmaller{\cI}}^{k-1}f(a+c_0)[b] = \sum_{\tau \in S_{k-1}} f^{[k-1]}\big((a+c_0)_{(k)}\big)\sh b^{\tau},
\]
for all $c_0 \in \cI_{\sa}$, where $b^{\tau} = (b_{\tau(1)},\ldots,b_{\tau(k-1)})$ for $\tau \in S_{k-1}$.
Combining this inductive hypothesis with the expression for $\delta(b,c)$ above gives
\begin{align*}
    \e(b,c) & \coloneqq D^{k-1}f_{a,\mathsmaller{\cI}}(c)[b] - D^{k-1}f_{a,\mathsmaller{\cI}}(0)[b] - \sum_{\tau \in S_{k-1}}\sum_{j=1}^kf^{[k]}\big(a_{(k+1)}\big)\sh [b_{j-}^{\tau},c,b_{j+}^{\tau}] \\
    & = \sum_{\tau \in S_{k-1}} \big(f^{[k-1]}\big((a+c)_{(k)}\big)\sh b^{\tau} - f^{[k-1]}\big(a_{(k)}\big)\sh b^{\tau}\big) - \sum_{\tau \in S_{k-1}}\sum_{j=1}^kf^{[k]}\big(a_{(k+1)}\big)\sh [b_{j-}^{\tau},c,b_{j+}^{\tau}] \\
    & = \sum_{\tau \in S_{k-1}}\Bigg(\delta(b^{\tau},c) - \sum_{j=1}^kf^{[k]}\big(a_{(k+1)}\big)\sh [b_{j-}^{\tau},c,b_{j+}^{\tau}]\Bigg) \\
    & = \sum_{\tau \in S_{k-1}}\sum_{j=1}^k\big(f^{[k]}\big((a+c)_{(j)},a_{(k+1-j)}\big)\sh [b_{j-}^{\tau},c,b_{j+}^{\tau}]-f^{[k]}\big(a_{(k+1)}\big)\sh [b_{j-}^{\tau},c,b_{j+}^{\tau}]\big).
\end{align*}
It follows that\vspace{-0.7mm}
\[
\frac{\|\e(\cdot,c)\|_{B_{k-1}(\cI_{\sa}^{k-1};\cI)}}{\|c\|_{\cI}} \leq (k-1)!\sum_{j=1}^k\big\|I^{(a+c)_{(j)},a_{(k+1-j)}}f^{[k]}-I^{a_{(k+1)}}f^{[k]}\big\|_{B_k(\cI^k;\cI)} \to 0\vspace{-1.5mm}
\]
as $c \to 0$ in $\cI_{\sa}$ by the Continuous Perturbation Lemma.
Writing $\tilde{b} \coloneqq (b_0,b_1,\ldots,b_{k-1})$, this proves\vspace{-1.01mm}
\[
D_{\mathsmaller{\cI}}^kf(a)\big[\tilde{b}\big] = D^kf_{a,\mathsmaller{\cI}}(0)\big[\tilde{b}\big] = \sum_{\tau \in S_{k-1}}\sum_{j=1}^kf^{[k]}\big(a_{(k+1)}\big)\sh [b_{j-}^{\tau},b_0,b_{j+}^{\tau}] = \sum_{\pi \in S_k}f^{[k]}\big(a_{(k+1)}\big)\sh\tilde{b}^{\pi},\vspace{-1.01mm}
\]
as claimed.
This completes the proof.
\end{proof}

\begin{rem}\label{rem.dePS}
Let $H$ be a separable complex Hilbert space, $(\cM \subseteq B(H),\tau)$ be a semifinite von Neumann algebra, and $(E,\|\cdot\|_E)$ be a separable symmetric Banach function space.
In \cite{depagtersukochev}, it is proven (Theorem 5.16) that if $f \colon \R \to \R$ is a continuous function such that $f^{[1]}$ admits a decomposition as in Definition \ref{def.BOCIPTP} with only $\varphi_1(\cdot,\sigma),\varphi_2(\cdot,\sigma) \in BC(\R)$ (i.e., these functions are not assumed to be operator continuous) and if $a \in S(\tau)_{\sa}$, then the map $E(\tau)_{\sa} \ni b \mapsto f(a+b)-f(a) \in E(\tau)_{\sa}$ makes sense and is Gateaux differentiable at $0$ with Gateaux derivatives expressible as double operator integrals involving $f^{[1]}$.
In particular, this result applies when $E = L^p$ with $1 \leq p < \infty$.
It is noted, however, in the introduction of \cite{depagtersukochev} that Fr\'{e}chet differentiability does not in general hold in this setting.
This is why we must work in the space $(\mathcal{E}(\tau),\|\cdot\|_{\mathcal{E}(\tau)}) = (E(\tau) \cap \cM,\|\cdot\|_{E(\tau) \cap \cM})$ (e.g., $\mathcal{L}^p(\tau)$) instead of the space $(E(\tau),\|\cdot\|_{E(\tau)})$ (e.g., $L^p(\tau)$), to prove positive results about Fr\'{e}chet differentiability in this setting.
(Also, our method --- particularly the extra assumption of operator continuity in our decompositions --- allows us to assume only that $a \aff \cM_{\sa}$, i.e., we need not assume that $a$ is $\tau$-measurable.)
In short, the results in \cite{depagtersukochev} are, for good reason, of a different flavor than the results in the present paper.
\end{rem}

\appendix
\section{Proof of Theorem \ref{thm.Peller}}\label{sec.BesovPeller}

In this appendix, we provide a full proof of Theorem \ref{thm.Peller}.
We shall freely use basic facts about tempered distributions and their Fourier transforms.
In particular, we recall that, as a consequence of the Paley--Wiener Theorem, if $f \in \mathscr{S}'(\R^m)$ is such that $\supp \wh{f}$ is compact, then $f$ is a smooth function.

\subsection{Part I}\label{sec.PellerI}

First, we set some notation that we shall use to write an expression (Theorem \ref{thm.Pellerexp} below) that is key to the endeavor of proving Theorem \ref{thm.Peller}.
Write $\R_+ \coloneqq [0,\infty)$.

\begin{nota}
We define two families $(r_u)_{u \in \R_+}$ and $(\mu_u)_{u \in \R_+}$ of tempered distributions on $\R$ by requiring\vspace{-1.01mm}
\begin{align*}
    \wh{r}_u(\xi) & \coloneqq 1_{[0,u]}(|\xi|) + \frac{u}{|\xi|}1_{(u,\infty)}(|\xi|) = \begin{cases}
    1, & \text{if } |\xi| \leq u \\
    \frac{u}{|\xi|}, & \text{if } |\xi| > u
    \end{cases} \\[-0.8mm]
    \wh{\mu_u}(\xi) & \coloneqq \frac{|\xi|-u}{|\xi|}1_{(u,\infty)}(|\xi|) = 1- \wh{r}_u(\xi)
\end{align*}\vspace{-1.01mm}
for $u > 0$ and $\xi \in \R$;
and $r_0 \coloneqq \delta_0$, $\mu_0 \coloneqq 0$.
In other words, $\mu_u = \wch{1} - r_u = \delta_0 - r_u$, for all $u \geq 0$.
\end{nota}

\begin{prop}\label{prop.randmu}
Let $f \colon \R \to \C$ be a Borel measurable function.
Write $f \ast \mu \colon \R_+ \times \R \to \C$ for the map $(u,x) \mapsto (f \ast \mu_u)(x) = f(x) - (f \ast r_u)(x)$ when it makes sense.\vspace{-0.4mm}
\begin{enumerate}[label=(\roman*),font=\normalfont,leftmargin=2\parindent]
    \item If $u > 0$, then $r_u \in L^1(\R) \cap L^2(\R)$.
    Specifically, $r_u = u\, r_1(u \cdot)$, $\|r_1\|_{L^2} = \sqrt{2\pi^{-1}}$, and $\|r_1\|_{L^1} < 2 < \infty$;
    so that $\|r_u\|_{L^2} = \sqrt{2(\pi u)^{-1}}$ and $\|r_u\|_{L^1} < 2$. \label{item.rbound}\vspace{-0.4mm}
    \item If $f$ is bounded, then $f \ast \mu$ is bounded and Borel measurable with\vspace{-1.01mm}
    \[
    \|f \ast \mu\|_{\ell^{\infty}(\R_+ \times \R)} \leq \|f\|_{\ell^{\infty}(\R)}(1+\|r_1\|_{L^1}) \leq 3\|f\|_{\ell^{\infty}(\R)}.\vspace{-1.01mm}
    \]
    (And we can replace the $\ell$'s with $L$'s.)
    If in addition $f \in C(\R)$, then $f \ast \mu \in C((0,\infty) \times \R)$.\label{item.fastmuinfbound}\vspace{-0.4mm}
    \item If $f \in L^1(\R)$, then $\|f \ast \mu_u\|_{L^1} \leq \|f\|_{L^1}(1+\|r_1\|_{L^1}) \leq 3\|f\|_{L^1}$, for all $u \geq 0$, as well.\label{item.fastmu1bound}\vspace{-0.4mm}
    \item Fix $\sigma > 0$.
    Suppose $f$ is bounded and $\supp \wh{f} \subseteq [0,\sigma]$.
    Then $f \ast \mu_u \equiv 0$ when $u > \sigma$.
    In particular, $(f \ast \mu_{\boldsymbol{\cdot}})(x) \in C_c(\R_+)$, for all $x \in \R$.\label{item.fastmusupp}
\end{enumerate}
\end{prop}
\pagebreak
\begin{proof}
We take each item in turn but postpone the proof of \ref{item.fastmusupp} until just after Lemma \ref{lem.Pellerapprox}.

\ref{item.rbound} First, notice that $\|\wh{r}_1\|_{L^2} = 2$ by an easy calculation.
Therefore, $\|r_1\|_{L^2} = (2\pi)^{-\frac{1}{2}}\|\wh{r}_1\|_{L^2} = \sqrt{2\pi^{-1}}$ by Plancherel's Theorem.
Next, fix $u > 0$ and $\xi \in \R$.
Notice $\wh{r}_u(\xi) = \wh{r}_1(\xi/u)$, from which it follows, by Fourier Inversion on $L^2$, that $r_u = \cF^{-1}(\wh{r}_1(\cdot/u)) = u\, r_1(u\, \cdot)$.
In particular, we have $\|r_u\|_{L^2} = u^{-\frac{1}{2}}\|r_1\|_{L^2}$ and $\|r_u\|_{L^1} = \|r_1\|_{L^1}$, as claimed.

It now suffices to prove $\|r_1\|_{L^1} < 2$.
To this end, note $\int_{-1}^1 |r_1(x)| \, dx \leq \sqrt{2} \|r_1\|_{L^2} = 2\pi^{-\frac{1}{2}}$, by the previous paragraph.
Now, for almost every $x \in \R$,
\begin{align*}
    r_1(x) & = \frac{1}{2\pi}\int_{-\infty}^{\infty} \wh{r}_1(\xi) e^{ix\xi} \, d\xi = -\frac{1}{2\pi x^2} \int_{-\infty}^{\infty} \wh{r}_1(\xi) \frac{d^2}{d\xi^2}e^{ix\xi} \, d\xi \\
    & = -\frac{1}{2\pi x^2} \int_{-\infty}^{\infty} \frac{d^2}{d\xi^2}\wh{r}_1(\xi) e^{ix\xi} \, d\xi = -\frac{1}{2\pi x^2} \int_{|\xi|>1} \frac{2 }{|\xi|^3} e^{ix\xi} \, d\xi,
\end{align*}
using integration by parts, where all of the above are improper Riemann integrals.\footnote{This is not technically perfect since actually $r_1 = L^2\text{-}\lim_{R \to \infty} \frac{1}{2\pi}\int_{|\xi| \leq R} \wh{r}_1(\xi)e^{i\boldsymbol{\cdot}\xi} \, d\xi$.
One should really take a particular sequence $R_k \to \infty$ as $k \to \infty$ for the first couple of integrals.}
Now, notice
\[
\Bigg|\int_{|\xi|>1} \frac{2 }{|\xi|^3} e^{ix\xi} \, d\xi\Bigg| \leq 4\int_1^{\infty}\frac{1}{\xi^3} \, d\xi = 2.
\]
It follows that $\int_{|x|>1}|r_1(x)| \, dx \leq \frac{2}{\pi}\int_1^{\infty}\frac{1}{x^2} \, dx = \frac{2}{\pi}$.
We finally conclude $\|r_1\|_{L^1} \leq \frac{2}{\sqrt{\pi}} + \frac{2}{\pi} < 2$, as desired.

\ref{item.fastmuinfbound} Fix $u > 0$ and $x \in \R$.
Then, recalling $r_u = u \, r_1(u\cdot)$,
\[
f \ast r_u(x) = \int_{\R}f(x-y) \,r_u(y) \, dy = \int_{\R} f(x-u^{-1}t) \, r_1(t) \, dt.
\]
Measurability of $f \ast \mu$ follows from this identity and the fact that $f \ast \mu(0,\cdot) = 0$.
The bounds are also immediate from this identity (because $f \ast \mu = f-f \ast r_{\cdot}$) and the first part.
Finally, joint continuity of $(0,\infty) \times \R \ni (u,x) \mapsto f \ast r_u(x) \in \R$ follows from the continuity of $f$ and the Dominated Convergence Theorem (which applies because $f$ is bounded and $r_1 \in L^1(\R)$).

\ref{item.fastmu1bound} This is immediate from Young's Convolution Inequality (when $u > 0$), the fact that $f \ast \mu_0 = 0$ (when $u=0$), and the first part.
\end{proof}

In order to bound integral projective tensor norms, one must exhibit expressions for the functions in question as integrals that ``separate variables" in a particular way.
Here is one such expression, which we take the rest of the section to prove.

\begin{thm}\label{thm.Pellerexp}
Fix $\sigma > 0$.
If $f \in \ell^{\infty}(\R,\cB_{\R})$ satisfies $\supp \wh{f} \subseteq [0,\sigma]$, then
\[
f^{[k]}(\blambda) = i^k\sum_{j=1}^{k+1}\int_{\R_+^k} \Bigg(\prod_{m=1}^{j-1}e^{i\lambda_mu_m}\Bigg) (f \ast \mu_{|\vec{u}|})(\lambda_j)\,e^{-i\lambda_j|\vec{u}|}\Bigg( \prod_{m=j+1}^{k+1}e^{i\lambda_mu_{m-1}}\Bigg) \, d\vec{u}, \numberthis\label{eq.fkexp}
\]
for all $k \in \N$ and $\blambda \in \R^{k+1}$, where $|\vec{u}| \coloneqq \sum_{m=1}^ku_m$ and empty products are defined to be $1$.
\end{thm}
\begin{rem}
The expression in \eqref{eq.fkexp} was written in \cite{peller0} and \cite{peller1} in the cases $k=1$ and $k=2$, respectively, in a slightly different form.
The use of $\mu_u$ was inspired by \cite{peller2}, in which \eqref{eq.fkexp} is written down exactly as stated in the cases $k \in \{1,2\}$.
\end{rem}

For example,
\begin{align*}
    f^{[1]}(\lambda_1,\lambda_2) & = i\int_{\R_+}\big((f \ast \mu_u)(\lambda_1) \,e^{-i\lambda_1u} e^{i\lambda_2u} +  e^{i\lambda_1u}(f \ast \mu_u)(\lambda_2) \,e^{-i\lambda_2u}\big) \, du \;\text{ and} \numberthis\label{eq.expf1}\\
    f^{[2]}(\lambda_1,\lambda_2,\lambda_3) & = -\int_{\R_+^2} \big((f \ast \mu_{u+v})(\lambda_1) \,e^{-i\lambda_1(u+v)} e^{i\lambda_2u}e^{i\lambda_3v}  + e^{i\lambda_1u}(f \ast \mu_{u+v})(\lambda_2) \,e^{-i\lambda_2(u+v)} e^{i\lambda_3v} \\
    & \; \; \; \; \; \; \; \; \; \; \; \; \; \; \; \; \; \; \; \; \; \; \; \; \; \; + e^{i\lambda_1u}e^{i\lambda_2v}(f \ast \mu_{u+v})(\lambda_3) \,e^{-i\lambda_3(u+v)} \big)\,  du \, dv,
\end{align*}
for all $\lambda_1,\lambda_2,\lambda_3 \in \R$.
\pagebreak

Notice that Proposition \ref{prop.randmu} allows us to make sense of the expression \eqref{eq.fkexp} in the first place.
By item \ref{item.fastmusupp}, the integrand in \eqref{eq.fkexp} is bounded, continuous, and vanishes when $|\vec{u}| > \sigma$. 
Therefore, the integral above is really over $\{\vec{u} \in \R_+^k : |\vec{u}| \leq \sigma\}$, which has finite measure.
This, together with the continuity part of item \ref{item.fastmuinfbound} and the Dominated Convergence Theorem, also implies the right hand side of \eqref{eq.fkexp} is continuous in $\blambda$.

The expression \eqref{eq.fkexp} is proven, inspired by the sketch in \cite{peller1}, in the following steps.
\begin{enumerate}[label=Step \arabic*., leftmargin=3.36\parindent]
    \item Use an approximation procedure (Lemma \ref{lem.Pellerapprox}) to reduce to the case $f,\wh{f} \in L^1(\R)$.\makeatletter\def\@currentlabel{Step 1}\makeatother\label{step.1}
    \item Use an inductive argument to reduce to the case $k=1$.\makeatletter\def\@currentlabel{Step 2}\makeatother\label{step.2}
    \item Prove \eqref{eq.fkexp} when $k=1$ (i.e., prove \eqref{eq.expf1}) assuming $f,\wh{f} \in L^1(\R)$.\makeatletter\def\@currentlabel{Step 3}\makeatother\label{step.3}
\end{enumerate}
The approximation procedure in \ref{step.1} will also help us to prove Proposition \ref{prop.randmu}\ref{item.fastmusupp}.

\begin{conv}\label{conv.PW}
For this section, fix $\sigma > 0$ and $f \in \ell^{\infty}(\R,\cB_{\R})$ with $\supp \wh{f} \subseteq [0,\sigma]$.
\end{conv}

\begin{lem}\label{lem.Pellerapprox}
Fix $0 \leq \om \in C_c^{\infty}(\R)$ such that $\supp \om \subseteq [0,1]$ and $\int_{\R} \om(\xi) \, d\xi = 2\pi$.
Define 
\[
\om_n \coloneqq n \, \om(n\cdot) \; \text{ and } \; f_n \coloneqq \wch{\om}_nf,
\]
for all $n \in \N$.
Then
\begin{enumerate}[label=(\roman*),font=\normalfont,leftmargin=2\parindent]
    \item $\|f_n\|_{\ell^{\infty}(\R)} \leq \|f\|_{\ell^{\infty}(\R)}$ and $f_n \to f$ pointwise as $n \to \infty$,\label{item.Linf}
    \item $f_n \in L^1(\R) \cap L^{\infty}(\R)$,\label{item.L1}
    \item $\wh{f}_n \in \mathscr{S}(\R) \subseteq L^1(\R)$ and $\supp \wh{f}_n \subseteq \big[0,\sigma+\frac{1}{n}\big] \subseteq [0,\sigma+1]$, and\label{item.hatL1}
    \item $f_n \ast \mu \to f \ast \mu$ boundedly on $\R_+ \times \R$ as $n \to \infty$.\label{item.bddconv}
\end{enumerate}
\end{lem}
\begin{proof}
We take each item in turn.

\ref{item.Linf} Notice that for all $x \in \R$,
\[
\wch{\om}_n(x) = \wch{n\,\om(n\cdot)}(x) = \wch{\om}(n^{-1}x) \to \wch{\om}(0) = \frac{1}{2\pi}\int_{\R} \om(\xi) \, d\xi = 1
\]
as $n \to \infty$.
But also, since $\om \geq 0$, $\big|\wch{\om}_n(x)\big| = \big|\frac{1}{2\pi} \int_{\R}\om(\xi) e^{in^{-1}x\xi} \, d\xi \big| \leq \frac{1}{2\pi}\int_{\R}\om(\xi) \, d\xi = 1$, i.e., $\|\wch{\om}_n\|_{\ell^{\infty}(\R)} \leq 1$, as well.
This takes care of the first part.

\ref{item.L1} Of course, $\wch{\om}_n \in \mathscr{S}(\R) \subseteq L^1(\R)$, so that $\|f_n\|_{L^1} \leq \|f\|_{L^{\infty}}\|\wch{\om}_n\|_{L^1} < \infty$.

\ref{item.hatL1} By basic properties of Fourier transforms of tempered distributions, $\wh{f}_n = \cF(\wch{\om}_nf) = \om_n \ast \wh{f}$.
Since $\wh{f}$ has compact support, we have $\om_n \ast \wh{f} \in \mathscr{S}(\R)$ and
\[
\supp \wh{f}_n = \supp \big(\om_n \ast \wh{f}\big) \subseteq \overline{\supp \om_n + \supp \wh{f}} \subseteq \big[0,n^{-1}\big]+[0,\sigma] = \big[0,\sigma+n^{-1}\big],
\]
as claimed.

\ref{item.bddconv} By Proposition \ref{prop.randmu} and the first part, $\|f_n \ast \mu\|_{\ell^{\infty}(\R_+\times \R)} \leq 3\|f_n\|_{\ell^{\infty}(\R)} \leq 3\|f\|_{\ell^{\infty}(\R)}$, for all $n \in \N$.
Now, let $u > 0$ and $x \in \R$.
(The case $u=0$ is obvious.)
By the proof of Proposition \ref{prop.randmu}\ref{item.fastmuinfbound} and the Dominated Convergence Theorem,
\[
(f_n \ast r_u)(x) = \int_{\R} f_n(x-u^{-1}y) \,r_1(y) \, dy \to \int_{\R} f(x-u^{-1}y) \,r_1(y) \, dy = (f \ast r_u)(x)
\]
as $n \to \infty$.
Therefore, $(f_n \ast \mu_u)(x) = f_n(x) - (f_n \ast r_u)(x) \to f(x) - (f \ast r_u)(x) = (f \ast \mu_u)(x)$ as $n \to \infty$.
\end{proof}

\begin{proof}[Proof of Proposition \ref{prop.randmu}\ref{item.fastmusupp}]
Suppose first that $f,\wh{f} \in L^1(\R)$. Recall from Proposition \ref{prop.randmu}\ref{item.fastmu1bound} that $\|f\ast\mu_u\|_{L^1} \leq 3\|f\|_{L^1}$, so that $f \ast \mu_u \in L^1(\R)$.
Also,
\[
\cF( f \ast \mu_u) = \wh{f} \; \,\wh{\mu_u} \in L^1(\R).
\]
But $\supp \wh{\mu_u} = (-\infty,-u] \cup [u,\infty)$ when $u > 0$, and $\supp \wh{f} \subseteq [0,\sigma]$.
Therefore, if $u > \sigma$, then $\cF(f \ast \mu_u) \equiv 0$.
Therefore, by the Fourier Inversion Theorem,
\[
f \ast \mu_u = \cF^{-1}(\cF(f \ast \mu_u)) \equiv 0
\]
as well.
(Recall $f \ast \mu_u \in C(\R)$, so this equality is \textit{everywhere}.)
\pagebreak

Now, for general $f$ as in Proposition \ref{prop.randmu}\ref{item.fastmusupp}, let $(f_n)_{n \in \N}$ be as in Lemma \ref{lem.Pellerapprox}.
Since $f_n,\wh{f}_n \in L^1(\R)$ and $\supp \wh{f}_n \subseteq \big[0,\sigma+\frac{1}{n}\big]$, we know from the previous paragraph that $f_n \ast \mu_u \equiv 0$ when $u > \sigma + \frac{1}{n}$.
Now, suppose $u > \sigma$.
Then, choosing $n_1 \in \N$ such that $u > \sigma + \frac{1}{n}$, for all $n \geq n_1$, we know that $f_n \ast \mu_u \equiv 0$ whenever $n \geq n_1$.
Since $f_n \ast \mu \to f \ast \mu$ pointwise as $n \to \infty$, we conclude $f \ast \mu_u \equiv 0$ as well.
\end{proof}

We now begin the proof of Theorem \ref{thm.Pellerexp} in earnest.

\begin{nota}
Fix $\blambda = (\lambda_1,\ldots,\lambda_{k+1}) \in \R^{k+1}$ and $\vec{u} = (u_1,\ldots,u_k) \in \R_+^k$.
Define
\[
\e_{k,j}^f(\blambda,\vec{u}) \coloneqq \Bigg(\prod_{m=1}^{j-1}e^{i\lambda_mu_m}\Bigg) (f \ast \mu_{|\vec{u}|})(\lambda_j)\,e^{-i\lambda_j|\vec{u}|}\Bigg( \prod_{m=j+1}^{k+1}e^{i\lambda_mu_{m-1}}\Bigg)
\]
whenever $1 \leq j \leq k+1$, where $|\vec{u}| \coloneqq \sum_{m=1}^ku_m$.
\end{nota}

\begin{proof}[Proof of \ref{step.1}]
Suppose \eqref{eq.fkexp} holds when we also assume $f,\wh{f} \in L^1(\R)$.
For arbitrary $f$, let $(f_n)_{n \in \N}$ be as in Lemma \ref{lem.Pellerapprox}.
Since $f_n,\wh{f}_n \in L^1(\R)$, we know \eqref{eq.fkexp} holds for $f_n$ in place of $f$.
We must take $n \to \infty$ to obtain \eqref{eq.fkexp} for $f$.
To this end, first let $\lambda_1,\ldots,\lambda_{k+1} \in \R$ be distinct and $\blambda \coloneqq (\lambda_1,\ldots,\lambda_{k+1}) \in \R^{k+1}$.
Then, by the recursive definition of the $k^{\text{th}}$ divided difference, $f_n^{[k]}(\blambda) \to f^{[k]}(\blambda)$ as $n \to \infty$ because $f_n \to f$ pointwise as $n \to \infty$.
Second, $\e_{k,j}^{f_n} \to \e_{k,j}^f$ boundedly on $\R^{k+1} \times \R_+^k$ as $n \to \infty$ by Lemma \ref{lem.Pellerapprox}\ref{item.bddconv}.
Third, by Proposition \ref{prop.randmu}\ref{item.fastmusupp}, the integral $\sum_{j=1}^{k+1}\int_{\R_+^k} \e_{k,j}^{f_n}(\blambda,\vec{u}) \, d\vec{u}$ is really only over $\{\vec{u} \in \R_+^k : |\vec{u}| \leq \sigma+1\}$, for all $n \in \N$.
Therefore, by the assumption and the Dominated Convergence Theorem,
\[
f_n^{[k]}(\blambda) = i^k\sum_{j=1}^{k+1} \int_{\R_+^k}\e_{k,j}^{f_n}(\blambda,\vec{u}) \, d\vec{u} \to i^k\sum_{j=1}^{k+1}\int_{\R_+^k} \e_{k,j}^f(\blambda,\vec{u}) \, d\vec{u}
\]
as $n \to \infty$, for all $\blambda \in \R^{k+1}$.
We conclude that
\[
f^{[k]}(\blambda) = i^k\sum_{j=1}^{k+1}\int_{\R_+^k} \e_{k,j}^f(\blambda,\vec{u}) \, d\vec{u}
\]
whenever $\lambda_1,\ldots,\lambda_{k+1} \in \R$ are distinct.
Since $\{\blambda \in \R^{k+1} : \lambda_1,\ldots,\lambda_{k+1} \in \R$ are distinct$\}$ is dense in $\R^{k+1}$ and both sides of the above are continuous in $\blambda$, we are done.
\end{proof}

\begin{conv}
For the remainder of this section, assume in addition that $f,\wh{f} \in L^1(\R)$.
\end{conv}

By the reduction from \ref{step.1}, this assumption is appropriate.
Next comes the proof of \ref{step.2}, which is a bit painful and may be skipped on a first read.
We warm up with two easy lemmas.

\begin{lem}\label{lem.easy1}
If $u \geq 0$ and $h(\lambda) \coloneqq e^{i\lambda u}$, then $h^{[1]}(\lambda_1,\lambda_2) = i\int_0^ue^{i\lambda_1v}e^{i\lambda_2(u-v)} \, dv$.
\end{lem}
\begin{proof}
The result is obvious if $u=0$, so we assume $u > 0$.
By Proposition \ref{prop.divdiffCk},
\[
h^{[1]}(\lambda_1,\lambda_2) = \int_0^1h'(t\lambda_1+(1-t)\lambda_2) \, dt = i\int_0^1ue^{i(t\lambda_1+(1-t)\lambda_2)u}\,dt = i\int_0^ue^{i\lambda_1v}e^{i\lambda_2(u-v)} \, dv,
\]
where we substituted $v \coloneqq tu$.
\end{proof}

\begin{lem}\label{lem.easy2}
If $u > 0$ and $g(\lambda) \coloneqq (f \ast \mu_u)(\lambda) \,e^{-i\lambda u}$, then $(g \ast \mu_v)(\lambda) = (f \ast \mu_{u+v})(\lambda) \,e^{-i\lambda u}$ when $v > 0$.
\end{lem}
\begin{proof}
Note that $\wh{g}(\xi) = \cF(f \ast \mu_u)(\xi+u) = \wh{f}(\xi+u) \, \wh{\mu_u}(\xi+u)$, so that
\[
\cF (g \ast \mu_v)(\xi) = \wh{g}(\xi) \, \wh{\mu_v}(\xi) = \wh{f}(\xi+u) \, \wh{\mu_u}(\xi+u) \, \wh{\mu_v}(\xi).
\]
But
\begin{align*}
    \wh{\mu_u}(\xi+u) \, \wh{\mu_v}(\xi) & = \frac{\xi+u-u}{\xi+u}1_{(u,\infty)}(\xi+u) \, \frac{\xi-v}{\xi} 1_{(v,\infty)}(\xi) \\
    & = \frac{\xi-v}{\xi+u} 1_{(u+v,\infty)}(\xi+u) = \wh{\mu_{u+v}}(\xi+u),
\end{align*}
so that $\cF(g \ast \mu_u)(\xi) = \wh{f}(\xi+u)\, \wh{\mu_{u+v}}(\xi+u) = \cF((f \ast \mu_{u+v}) \, e^{-i\boldsymbol{\cdot}u})(\xi)$, for all $\xi \in \R$.
The result follows from the Fourier Inversion Theorem.
\end{proof}
\pagebreak

We are now ready for the proof of \ref{step.2}.

\begin{proof}[Proof of \ref{step.2}]
Assume for some $\ell \geq 1$ that \eqref{eq.fkexp} holds whenever $1 \leq k \leq \ell$ (and all relevant $f$).
Suppose that $1 \leq k \leq \ell$, and fix distinct $\lambda_1,\ldots,\lambda_{k+2} \in \R$.
Then
\begin{align*}
    f^{[k+1]}(\lambda_1,\ldots,\lambda_{k+2}) & = \frac{f^{[k]}(\lambda_1,\ldots,\lambda_{k+1}) - f^{[k]}(\lambda_1,\ldots,\lambda_k,\lambda_{k+2})}{\lambda_{k+1}-\lambda_{k+2}} \\
    & = i^k\sum_{j=1}^{k+1}\int_{\R_+^k} \frac{\e_{k,j}^f(\lambda_1,\ldots,\lambda_{k+1},\vec{u}) - \e_{k,j}^f(\lambda_1,\ldots,\lambda_k,\lambda_{k+2},\vec{u})}{\lambda_{k+1}-\lambda_{k+2}} \, d\vec{u}.
\end{align*}
We now examine each term in the above sum.
Define
\[
\delta_j(\vec{u}) \coloneqq \frac{\e_{k,j}^f(\lambda_1,\ldots,\lambda_{k+1},\vec{u}) - \e_{k,j}^f(\lambda_1,\ldots,\lambda_k,\lambda_{k+2},\vec{u})}{\lambda_{k+1}-\lambda_{k+2}}
\]
for ease of notation.

First, suppose $1 \leq j < k+1$.
Then, by definition of the $\e_{k,j}$'s and Lemma \ref{lem.easy1},
\begin{align*}
    \delta_j(\vec{u}) & = \prod_{m=1}^{j-1}e^{i\lambda_mu_m} (f \ast \mu_{|\vec{u}|})(\lambda_j)\,e^{-i\lambda_j|\vec{u}|} \prod_{m=j+1}^ke^{i\lambda_mu_{m-1}}\frac{e^{i\lambda_{k+1}u_k}-e^{i\lambda_{k+2}u_k}}{\lambda_{k+1}-\lambda_{k+2}} \\
    & = i\int_0^{u_k}\prod_{m=1}^{j-1}e^{i\lambda_mu_m} (f \ast \mu_{|\vec{u}|})(\lambda_j)\,e^{-i\lambda_j|\vec{u}|} \prod_{m=j+1}^ke^{i\lambda_mu_{m-1}} \, e^{i\lambda_{k+1}v}e^{i\lambda_{k+2}(u_k-v)} \, dv.
\end{align*}
Now, this allows us to write
\[
\int_{\R_+^k}\delta_j(\vec{u}) \, d\vec{u} = i\int_{\R_+^k}\int_0^{u_k} \prod_{m=1}^{j-1}e^{i\lambda_mu_m} (f \ast \mu_{|\vec{u}|})(\lambda_j)\,e^{-i\lambda_j|\vec{u}|} \prod_{m=j+1}^ke^{i\lambda_mu_{m-1}} \, e^{i\lambda_{k+1}v}e^{i\lambda_{k+2}(u_k-v)}\, dv\, d\vec{u}.
\]
We now manipulate this integral expression.
Changing the order of integration yields
\[
i\int_{\R_+^k}\int_0^{u_k} \boldsymbol{\cdot} \, dv\, d\vec{u} = i\int_{\R_+}\int_{\R_+^k}1_{\{u_k \geq v\}}(\vec{u}) \boldsymbol{\cdot} \, d\vec{u} \,dv.
\]
Changing variables as $(u_1,\ldots,u_k,v) \mapsto (u_1,\ldots,u_{k-1},u_k-v,v) =: (v_1,\ldots,v_{k+1}) = \boldsymbol{v}$ yields
\[
i\int_{\R_+}\int_{\R_+^k}1_{\{u_k \geq v\}}(\vec{u}) \boldsymbol{\cdot} \, d\vec{u} \,dv = i\int_{\R_+^{k+1}} \boldsymbol{\cdot} \, d\boldsymbol{v},
\]
where functions evaluated at $u_1,\ldots,u_k,v$ go to the same functions evaluated at $v_1,\ldots,v_{k-1},v_k+v_{k+1},v_{k+1}$.
(In particular, $|\vec{u}|$ goes to $|\boldsymbol{v}|$.)
This yields
\begin{align*}
    \int_{\R_+^k}\delta_j(\vec{u}) \, d\vec{u} & = i\int_{\R_+^{k+1}}\prod_{m=1}^{j-1}e^{i\lambda_mv_m} (f \ast \mu_{|\boldsymbol{v}|})(\lambda_j)\,e^{-i\lambda_j|\boldsymbol{v}|} \prod_{m=j+1}^ke^{i\lambda_mv_{m-1}}\, e^{i\lambda_{k+1}v_k}e^{i\lambda_{k+2}v_{k+1}} \, d\boldsymbol{v} \\
    & = i\int_{\R_+^{k+1}}\prod_{m=1}^{j-1}e^{i\lambda_mv_m} (f \ast \mu_{|\boldsymbol{v}|})(\lambda_j)\,e^{-i\lambda_j|\boldsymbol{v}|} \prod_{m=j+1}^{k+2}e^{i\lambda_mv_{m-1}}\, d\boldsymbol{v} \\
    & = i\int_{\R_+^{k+1}}\e_{k+1,j}^f(\lambda_1,\ldots,\lambda_{k+2},\boldsymbol{v})\, d\boldsymbol{v},
\end{align*}
which is one of the terms we wanted to see.

Second, for the $j=k+1$ term, notice that
\[
\delta_{k+1}(\vec{u}) = \prod_{m=1}^k e^{i\lambda_mu_m} \big(f \ast \mu_{|\vec{u}|} \,e^{-i\boldsymbol{\cdot} |\vec{u}|}\big)^{[1]}(\lambda_{k+1},\lambda_{k+2}).
\]
But $g(x) \coloneqq (f \ast \mu_{|\vec{u}|})(x) \,e^{-ix |\vec{u}|}$ satisfies $g,\wh{g} \in L^1(\R) \cap L^{\infty}(\R)$ and $\supp \wh{g} \subseteq [0,\sigma]$.
Therefore, by assumption and Lemma \ref{lem.easy2}, we have
\begin{align*}
    g^{[1]}(\lambda_{k+1},\lambda_{k+2}) & = i\int_{\R_+}\big((g \ast \mu_v)(\lambda_{k+1}) \,e^{-i\lambda_{k+1}v} e^{i\lambda_{k+2}v} +(g \ast \mu_v)(\lambda_{k+2}) \,e^{-i\lambda_{k+2}v} e^{i\lambda_{k+1}v}\big) \, dv \\
    & = i\int_{\R_+}\hspace{-0.5mm}\big((f \ast \mu_{|\vec{u}|+v})(\lambda_{k+1})\, e^{-i\lambda_{k+1}(|\vec{u}|+v)} e^{i\lambda_{k+2}v} \hspace{-0.5mm} +\hspace{-0.5mm}(f \ast \mu_{|\vec{u}|+v})(\lambda_{k+2})\,e^{-i\lambda_{k+2}(|\vec{u}|+v)} e^{i\lambda_{k+1}v}\big) \,dv
\end{align*}
assuming $|\vec{u}| > 0$.
Therefore, renaming $(u_1,\ldots,u_k,v)$ to $(v_1,\ldots,v_{k+1})$,
\begin{align*}
    \int_{\R_+^k}\delta_{k+1}(\vec{u}) \, d\vec{u} & = i\int_{\R_+^k}\int_{\R_+}\Bigg(\prod_{m=1}^k e^{i\lambda_mu_m}(f \ast \mu_{|\vec{u}|+v})(\lambda_{k+1})\, e^{-i\lambda_{k+1}(|\vec{u}|+v)} e^{i\lambda_{k+2}v} \\
    & \; \; \; \; \; \; \; \; \; \; \; \; \; \; \; \; \; \; \; \; \; \; \; \; \; \; +\prod_{m=1}^k e^{i\lambda_mu_m}(f \ast \mu_{|\vec{u}|+v})(\lambda_{k+2})\,e^{-i\lambda_{k+2}(|\vec{u}|+v)} e^{i\lambda_{k+1}v}\Bigg) \, dv \, d\vec{u} \\
    & = i\int_{\R_+^{k+1}}\Bigg(\prod_{m=1}^k e^{i\lambda_mv_m}(f \ast \mu_{|\boldsymbol{v}|})(\lambda_{k+1})\, e^{-i\lambda_{k+1}|\boldsymbol{v}|} e^{i\lambda_{k+2}v_{k+1}} \\
    & \; \; \; \; \; \; \; \; \; \; \; \; \; \; \; \; \; \; \; \; \; \; \; \; \; \; +\prod_{m=1}^{k+1} e^{i\lambda_mv_m}(f \ast \mu_{|\boldsymbol{v}|})(\lambda_{k+2})\,e^{-i\lambda_{k+2}|\boldsymbol{v}|} \Bigg) \, d\boldsymbol{v} \\
    & = i\int_{\R_+^{k+1}} \big(\e_{k+1,k+1}^f(\lambda_1,\ldots,\lambda_{k+2},\boldsymbol{v}) + \e_{k+1,k+2}^f(\lambda_1,\ldots,\lambda_{k+2},\boldsymbol{v}) \big) \, d\boldsymbol{v},
\end{align*}
which are the remaining terms we needed.

Finally, putting it all together, we have
\[
f^{[k+1]}(\lambda_1,\ldots,\lambda_{k+2}) = i^k\sum_{j=1}^{k+1}\int_{\R_+^k} \delta_j(\vec{u}) \, d\vec{u} = i^{k+1}\sum_{j=1}^{k+2}\int_{\R_+^{k+1}}\e_{k+1,j}^f(\lambda_1,\ldots,\lambda_{k+2},\boldsymbol{v})\, d\boldsymbol{v}
\]
when $\lambda_1,\ldots,\lambda_{k+2} \in \R$ are distinct.
But, since both sides of the equation are continuous in $(\lambda_1,\ldots,\lambda_{k+2})$, this completes the proof. 
\end{proof}

We are now left with \ref{step.3} (the easiest step), i.e., the base case of the induction in \ref{step.2}.

\begin{proof}[Proof of \ref{step.3}]
First, we claim that
\[
f^{[1]}(\lambda_1,\lambda_2) = \frac{i}{2\pi}\int_{\R_+^2} \wh{f}(u+v)\, e^{i\lambda_1u}e^{i\lambda_2v} \, du \, dv,
\]
for all $\lambda_1,\lambda_2 \in \R$.
(The integral above makes sense because $\wh{f}$ is compactly supported and belongs to $L^1(\R) \cap L^{\infty}(\R)$.)
Indeed, by the Fourier Inversion Theorem, continuity of $f$, and the fact that $\supp \wh{f} \subseteq \R_+$, we have $f(\lambda) = \frac{1}{2\pi}\int_{\R_+} \wh{f}(\xi) \, e^{i\lambda\xi} \, d\xi$, for all $\lambda \in \R$.
Therefore, if $\lambda_1,\lambda_2 \in \R$ are distinct, then
\begin{align*}
    f^{[1]}(\lambda_1,\lambda_2) & = \frac{1}{2\pi}\int_{\R_+} \wh{f}(\xi) \frac{e^{i\lambda_1\xi} - e^{i\lambda_2\xi}}{\lambda_1-\lambda_2} \, d\xi = \frac{i}{2\pi}\int_{\R_+}\int_0^{\xi} \wh{f}(\xi) \, e^{i\lambda_1v}e^{i\lambda_2(\xi-v)} \, dv \, d\xi \\
    & = \frac{i}{2\pi}\int_{\R_+}\int_v^{\infty} \wh{f}(\xi) \, e^{i\lambda_1v}e^{i\lambda_2(\xi-v)} \, d\xi \, dv = \frac{i}{2\pi}\int_{\R_+^2} \wh{f}(u+v) \, e^{i\lambda_1v}e^{i\lambda_2u} \, du \, dv,
\end{align*}
by Lemma \ref{lem.easy1} and the change of variables $u \coloneqq \xi -v$.
Swapping the roles of $u$ and $v$ in the above integral gives the desired expression.
As usual, the continuity of both sides in $(\lambda_1,\lambda_2)$ allows us to pass from distinct $\lambda_1,\lambda_2$ to arbitrary $\lambda_1,\lambda_2$.

Therefore, our goal is to show
\[
i\int_{\R_+}\big((f \ast \mu_u)(\lambda_1) \,e^{-i\lambda_1u} e^{i\lambda_2u} +(f \ast \mu_u)(\lambda_2) \,e^{-i\lambda_2u} e^{i\lambda_1u}\big) \, du = \frac{i}{2\pi}\int_{\R_+^2} \wh{f}(u+v) \, e^{i\lambda_1u}e^{i\lambda_2v} \, du \, dv,\pagebreak
\]
for all $\lambda_1,\lambda_2 \in \R$.
To this end, notice that for all $u \geq 0$, the function $g(\lambda) \coloneqq (f \ast \mu_u)(\lambda) \,e^{-i\lambda u}$ satisfies $g,\wh{g} \in L^1(\R) \cap L^{\infty}(\R)$ and $g \in C(\R)$.
Also, 
\[
\wh{g}(\xi) = \wh{f}(\xi+u) \,\wh{\mu_u}(\xi+u) = \wh{f}(\xi+u) \frac{\xi}{\xi+u}1_{(0,\infty)}(\xi)
\]
when $u > 0$.
Therefore, by the Fourier Inversion Theorem and the continuity of $g$,
\[
g(\lambda) = \frac{1}{2\pi}\int_{\R}\wh{g}(\xi)\,e^{i\lambda\xi} \, d\xi = \frac{1}{2\pi}\int_{\R_+}\wh{f}(\xi+u)\frac{\xi}{\xi+u}e^{i\lambda\xi} \, d\xi,
\]
for all $\lambda \in \R$ (when $u > 0$).
Therefore,
\begin{align*}
    i\int_{\R_+}(f \ast \mu_u)(\lambda_1)\, e^{-i\lambda_1u} e^{i\lambda_2u} \, du & = \frac{i}{2\pi}\int_{\R_+}\int_{\R_+}\wh{f}(\xi+u)\frac{\xi}{\xi+u}e^{i\lambda_1\xi}e^{i\lambda_2u} \, d\xi \, du  \\
    & = \frac{i}{2\pi}\int_{\R_+^2}\wh{f}(u+v)\frac{u}{u+v}e^{i\lambda_1u}e^{i\lambda_2v} \, du \, dv \; \text{ and}\\
    i\int_{\R_+}(f \ast \mu_u)(\lambda_2)\,e^{-i\lambda_2u} e^{i\lambda_1u} \, du & = \frac{i}{2\pi}\int_{\R_+}\int_{\R_+}\wh{f}(\xi+u)\frac{\xi}{\xi+u}e^{i\lambda_2\xi}e^{i\lambda_1u} \, d\xi \, du \\
    & = \frac{i}{2\pi}\int_{\R_+^2}\wh{f}(u+v)\frac{v}{u+v}e^{i\lambda_1u}e^{i\lambda_2v} \, dv \, du.
\end{align*}
Adding these together yields $\frac{i}{2\pi}\int_{\R_+^2} \wh{f}(u+v) \, e^{i\lambda_1u}e^{i\lambda_2v} \, du \, dv$, as desired.
This completes the proof.
\end{proof}

\subsection{Part II}\label{sec.PellerII}

We now use Theorem \ref{thm.Pellerexp} to prove Theorem \ref{thm.Peller}.

\begin{prop}\label{prop.fastmuOC}
If $v,\sigma > 0$ and $f \colon \R \to \C$ is as in Convention \ref{conv.PW}, then $f \ast \mu_v \in BOC(\R)$.
\end{prop}
\begin{proof}
First, suppose $f,\wh{f} \in L^1(\R)$ as well, and let $g \coloneqq f \ast \mu_v$. Then $\wh{g} = \wh{f} \, \wh{\mu}_v$ is compactly supported on $\R_+$.
By Theorem \ref{thm.Pellerexp} (really, only \ref{step.3} of its proof), we have
\[
g^{[1]}(\lambda_1,\lambda_2) = i\int_{\R_+}\big((g \ast \mu_u)(\lambda_1) \,e^{-i\lambda_1u} e^{i\lambda_2u} +(g \ast \mu_u)(\lambda_2) \,e^{-i\lambda_2u} e^{i\lambda_1u}\big) \, du,
\]
for all $\lambda_1,\lambda_2 \in \R$.
For arbitrary $f$ as in the statement and $f_n$ as in Lemma \ref{lem.Pellerapprox}, we have that
\[
(f_n \ast \mu_v)^{[1]}(\lambda_1,\lambda_2) = i\int_{\R_+}\big(((f_n \ast \mu_v) \ast \mu_u)(\lambda_1) \,e^{-i\lambda_1u} e^{i\lambda_2u} +((f_n \ast \mu_v)\ast\mu_u)(\lambda_2) \,e^{-i\lambda_2u} e^{i\lambda_1u}\big) \, du.
\]
As in the proof of \ref{step.1} of the proof of Theorem \ref{thm.Pellerexp}, we may take $n \to \infty$ when $\lambda_1 \neq \lambda_2$ to conclude
\[
(f \ast \mu_v)^{[1]}(\lambda_1,\lambda_2) = i\int_{\R_+}\big(((f \ast \mu_v) \ast \mu_u)(\lambda_1) \,e^{-i\lambda_1u} e^{i\lambda_2u} +((f \ast \mu_v)\ast\mu_u)(\lambda_2) \,e^{-i\lambda_2u} e^{i\lambda_1u}\big) \, du.
\]
Since the right hand side is continuous in $(\lambda_1,\lambda_2)$, we conclude this identity holds when $\lambda_1=\lambda_2$ as well.
This is a $\ell^{\infty}$-integral projective decomposition of $(f \ast \mu_v)^{[1]}$.
We conclude from Corollary \ref{cor.Iperturb} that if $H$ is a complex Hilbert space, $a \in C(H)_{\sa}$, and $c \in B(H)_{\sa}$, then
\[
\|(f \ast \mu_v)(a+c)-(f \ast \mu_v)(a)\| \leq \big\|(f\ast \mu_v)^{[1]}\big\|_{\ell^{\infty}(\R,\cB_{\R}) \iotimes \ell^{\infty}(\R,\cB_{\R})}\|c\|,
\]
so that $f \ast \mu_v \in OC(\R)$.
Since $f \ast \mu_v$ is bounded by Proposition \ref{prop.randmu}\ref{item.fastmuinfbound}, we are done.
\end{proof}

\begin{prop}\label{prop.Pellerprecursor}
Fix $\sigma > 0$.
If $k \in \N$, then there is a constant $a_k < \infty$ such that whenever $f \in \ell^{\infty}(\R,\cB_{\R})$ satisfies $\supp \wh{f} \subseteq \big[-\sigma,-\frac{\sigma}{4}\big] \cup \big[\frac{\sigma}{4},\sigma\big]$, we have
\[
\big\|f^{[k]}\big\|_{BOC(\R)^{\iotimes(k+1)}} \leq a_k\sigma^k\|f\|_{L^{\infty}}.
\]
\end{prop}
\begin{proof}
Let $f$ be as in the statement of the proposition and $\psi_1 \in C_c^{\infty}(\R)$ be a bump function such that $\psi_1 \equiv 1$ on $\big[\frac{\sigma}{4},\sigma\big]$ and $\supp \psi_1 \subseteq \big[\frac{\sigma}{8},2\sigma\big]$.
Writing $\psi_2(\xi) \coloneqq \psi_1(-\xi)$, $\chi_1 \coloneqq \cF^{-1}(\psi_1)$ and $\chi_2 \coloneqq \cF^{-1}(\psi_2)$, we have $f = \chi_1 \ast f +\chi_2 \ast f$ because $\wh{f} = \psi_1\wh{f}+\psi_2\wh{f}$.
But $f_1 \coloneqq \chi_1 \ast f$ satisfies the hypotheses of Theorem \ref{thm.Pellerexp}.
By Lemma \ref{lem.expOC}, Proposition \ref{prop.fastmuOC}, the fact that $BOC(\R)$ is an algebra, and the comments about when the integrand in \eqref{eq.fkexp} vanishes, Theorem \ref{thm.Pellerexp} gives a BOCIPD of $f_1^{[k]}$ from which we may conclude
\begin{align*}
    \big\|f_1^{[k]}\big\|_{BOC(\R)^{\iotimes(k+1)}} & \leq \sum_{j=1}^{k+1}\int_{\{\vec{u} \in \R_+^k : |\vec{u}| \leq \sigma\}} \|f_1 \ast \mu_{|\vec{u}|}\|_{\ell^{\infty}(\R)} \, d\vec{u} \\
    & \leq 3(k+1)\frac{\sigma^k}{k!}\|f_1\|_{L^{\infty}} \leq 3\|\chi_1\|_{L^1}(k+1)\frac{\sigma^k}{k!}\|f\|_{L^{\infty}}
\end{align*}
by the bounds from Proposition \ref{prop.randmu} and Young's Convolution Inequality.
Next, $x \mapsto f_2(-x)$ also satisfies the hypotheses of Theorem \ref{thm.Pellerexp}.
This allows us to conclude
\[
\big\|f_2^{[k]}\big\|_{BOC(\R)^{\iotimes(k+1)}} \leq 3\|\chi_2\|_{L^1}(k+1)\frac{\sigma^k}{k!}\|f\|_{L^{\infty}}
\]
as well.
Thus we may take $a_k \leq \frac{3}{k!}(\|\chi_1\|_{L^1}+\|\chi_2\|_{L^1})(k+1)$ in the statement of proposition.
\end{proof}

We now transfer this result into the desired statement about Besov spaces.
Recall from Definition \ref{def.Besov} that we have fixed $\varphi \in C_c^{\infty}(\R^m)$ such that $0 \leq \varphi \leq 1$ everywhere, $\supp \varphi \subseteq \{\xi \in \R^m : \|\xi\| \leq 2\}$, and $\varphi \equiv 1$ on $\{\xi \in \R^m : \|\xi\| \leq 1\}$.
We also defined $\varphi_j(\xi) \coloneqq \varphi(2^{-j}\xi) - \varphi(2^{-j+1}\xi)$, for all $j \in \Z$ and $\xi \in \R^m$.
It is easy to see that $0 \leq \varphi_j \leq 1$, $\supp \varphi_j \subseteq \{\xi \in \R^m :  2^{j-1} \leq \|\xi\| \leq 2^{j+1}\}$, $\varphi+\sum_{j=1}^n\varphi_j = \varphi(2^{-n}\cdot)$ for $n \in \N$ (so that $\varphi+\sum_{j = 1}^{\infty} \varphi_j \equiv 1$ everywhere), and $\sum_{j = -\infty}^{\infty} \varphi_j = 1_{\R^m \setminus \{0\}}$.
From these bump functions, we get the \textit{Littlewood--Paley sequence/decompositions} of a tempered distribution.
Indeed, if $f \in \mathscr{S}'(\R^m)$, then
\[
f = \wch{\varphi(2^{-n+1}\boldsymbol{\cdot})} \ast f +\sum_{j = n}^{\infty} \wch{\varphi}_j \ast f \numberthis\label{eq.ILPD}
\]
in the weak$^*$ topology of $\mathscr{S}'(\R^m)$, for all $n \in \Z$.
Therefore, the series $\sum_{j = -\infty}^{\infty}\wch{\varphi}_j \ast f$ converges if and only if $(\wch{\varphi(2^n\cdot)} \ast f)_{n \in \N}$ converges, if and only if $(\varphi(2^n\cdot) \wh{f})_{n \in \N}$ converges (all in the weak$^*$ topology).
In particular,
\[
f = \sum_{j = -\infty}^{\infty} \wch{\varphi}_j \ast f \numberthis\label{eq.HLPD}
\]
in the weak$^*$ topology if and only if $\operatorname{w}^*\text{-}\lim_{n \to \infty} \wch{\varphi(2^n\boldsymbol{\cdot})} \ast f = 0 \Leftrightarrow \operatorname{w}^*\text{-}\lim_{n \to \infty} \varphi(2^n\boldsymbol{\cdot}) \wh{f} = 0$.
The identity in \eqref{eq.ILPD} with $n=1$ is called the \textbf{inhomogeneous Littlewood--Paley decomposition} of $f$.
The identity in \eqref{eq.HLPD} (or at the very least, the formal series therein) is called the \textbf{homogeneous Littlewood--Paley decomposition} of $f$.
Sometimes these are also called the \textbf{Calder\'{o}n Reproducing Formulas}.
The proofs boil down to the weak$^*$ continuity of the Fourier transform and the fact that if $\eta \in \mathscr{S}(\R^m)$, then $\varphi(R^{-1}\cdot)\eta \to \eta$ in $\mathscr{S}(\R^m)$ as $R \to \infty$, which is a nice exercise to prove.

Note that if $\sum_{j=-\infty}^{\infty} \wch{\varphi}_j \ast f$ converges (in the weak$^*$ topology), then $P \coloneqq f-\sum_{j=-\infty}^{\infty} \wch{\varphi}_j \ast f \in \mathscr{S}'(\R^m)$ is easily seen to have the property that $\supp \wh{P} \subseteq \{0\}$.
Therefore, $P \in \C[\lambda_1,\ldots,\lambda_m]$ is a polynomial and
\[
f = \sum_{j=-\infty}^{\infty} \wch{\varphi}_j \ast f + P.
\]
This observation will come in handy later.
The most important fact about Besov spaces for us is that the inhomogeneous Littlewood--Paley series of the $k^{\text{th}}$ derivative of a function belonging to $\dot{B}_1^{k,\infty}(\R)$ converges uniformly.
To prove this, we use Bernstein's Inequality.

\begin{lem}[Bernstein's Inequality]\label{lem.Berns}
Suppose $\alpha \in \N_0^m$ and $1 \leq r \leq p \leq \infty$.
There is a constant $b_{\alpha,r,p} <\infty$ such that for all $R > 0$ and $u \in \mathscr{S}'(\R^m)$ with 
$\supp \wh{u} \subseteq \{\xi \in \R^m : \|\xi\| \leq R\}$, we have
\[
\big\|\partial^{\alpha} u \big\|_{L^p} \leq b_{\alpha,r,p} R^{|\alpha|+m(\frac{1}{r}-\frac{1}{p})}\|u\|_{L^r}.
\]
\end{lem}
\begin{proof}
Defining $u_R \coloneqq R^{-m} u(R^{-1}\cdot)$, we see that $\supp \wh{u}_R \subseteq \{\xi \in \R^m : \|\xi\| \leq 1\}$.
Supposing we know the desired inequality when $R=1$, we have $\|\partial^{\alpha} u_R \|_{L^p} \lesssim \|u_R\|_{L^r}$.
Since 
\[
\partial^{\alpha}u_R = R^{-|\alpha|}R^{-m}(\partial^{\alpha}u)(R^{-1}\cdot) \; \text{ and } \; \|u_R\|_{L^q} = R^{m(\frac{1}{q}-1)}\|u\|_{L^q}
\]
for all $q \in [1,\infty]$, we conclude
\[
R^{-|\alpha|+m(\frac{1}{p}-1)} \big\|\partial^{\alpha}u\|_{L^p} =  \big\|\partial^{\alpha} u_R \big\|_{L^p} \lesssim \|u_R\|_{L^r} = R^{m(\frac{1}{r}-1)}\|u\|_{L^r},
\]
whence the desired inequality follows.
Therefore, we can and do assume $R=1$.

Next, we notice there are really two inequalities in the one we would like to prove:
\[
\|u\|_{L^p} \lesssim \|u\|_{L^r} \; \text{ and } \; \big\|\partial^{\alpha} u \big\|_{L^p}  \lesssim \|u\|_{L^p}.
\]
To prove these it is key to notice $u = \wch{\varphi} \ast u$ and $\partial^{\alpha} u = \wch{\varphi} \ast \partial^{\alpha}u = (\partial^{\alpha}\wch{\varphi}) \ast u$, by taking Fourier transforms of both sides and recalling $\varphi \equiv 1$ on $\{\xi \in \R^m : \|\xi\| \leq 1\}$.
Then, by Young's Convolution Inequality,
\[
\|u\|_{L^p} = \|\wch{\varphi} \ast u \|_{L^p} \leq \|\wch{\varphi} \|_{L^q} \|u\|_{L^r},
\]
where $\frac{1}{q} = 1+\frac{1}{p} - \frac{1}{r} \in [0,1]$ (using $1 \leq r \leq p \leq \infty$).
By the same inequality,
\[
\big\|\partial^{\alpha} u \big\|_{L^p} = \big\|(\partial^{\alpha}\wch{\varphi}) \ast u \big\|_{L^p} \leq \big\|\partial^{\alpha} \wch{\varphi}\big\|_{L^1}\|u\|_{L^p}.
\]
This completes the proof.
\end{proof}

We actually learned from the proof that we can take
\[
b_{\alpha,p,p} \leq \big\|\partial^{\alpha} \wch{\varphi}\big\|_{L^1} \; \text{ and } \; b_{\vec{0},r,p} \leq \|\wch{\varphi}\|_{L^q}, \numberthis\label{eq.Bernconst}
\]
where $\frac{1}{q} = 1-\big(\frac{1}{r}-\frac{1}{p}\big)$.
In particular, we can take $b_{\alpha,r,p} \leq \|\partial^{\alpha} \wch{\varphi}\|_{L^1}\|\wch{\varphi} \|_{L^q}$.

\begin{prop}\label{prop.HBesovHLPD}
Fix $s \in \R$, $1 \leq p \leq \infty$, and $f \in \dot{B}_1^{s,p}(\R^m)$.
If $\alpha \in \N_0^m$ and $|\alpha| = s-\frac{m}{p}$, then $\sum_{j=-\infty}^{\infty}\wch{\varphi}_j \ast \partial^{\alpha}f = \sum_{j=-\infty}^{\infty}\partial^{\alpha}(\wch{\varphi}_j \ast f)$ is absolutely uniformly convergent. 
\end{prop}
\begin{proof}
Since the Fourier transform of $\wch{\varphi}_j \ast \partial^{\alpha}f$ is supported in $\{\xi \in \R^m : \|\xi\| \leq 2^{j+1}\}$, Bernstein's Inequality (Lemma \ref{lem.Berns}) gives
\begin{align*}
    \sum_{j=-\infty}^{\infty} \big\|\wch{\varphi}_j \ast \partial^{\alpha}f\big\|_{L^{\infty}} & \leq b_{\alpha,p,\infty}\sum_{j=-\infty}^{\infty}(2^{j+1})^{|\alpha|+\frac{m}{p}} \big\|\wch{\varphi}_j \ast f\big\|_{L^p} \\
    & = 2^sb_{\alpha,p,\infty}\sum_{j=-\infty}^{\infty}2^{js} \big\|\wch{\varphi}_j \ast f\big\|_{L^p} = 2^sb_{\alpha,p,\infty}\|f\|_{\dot{B}_1^{s,p}},
\end{align*}
which is finite.
\end{proof}

Let us record a special bound we learned in the proof about our case of interest.
When $(s,p,q) = (k,\infty,1)$ for $k \in \N_0$ and $m=1$, we get from \eqref{eq.Bernconst} that
\[
\sum_{j=-\infty}^{\infty} \big\|(\wch{\varphi}_j \ast f)^{(k)}\big\|_{L^{\infty}} = \sum_{j=-\infty}^{\infty} \big\|\wch{\varphi}_j \ast f^{(k)}\big\|_{L^{\infty}} \leq 2^k\big\|\wch{\varphi}^{(k)}\big\|_{L^1}\|f\|_{\dot{B}_1^{k,\infty}} =: b_k\|f\|_{\dot{B}_1^{k,\infty}} \numberthis\label{eq.seriesbound}
\]
for $f \in \dot{B}_1^{k,\infty}(\R)$.
In particular, there is a polynomial $P_k \in \C[\lambda]$ such that
\[
f^{(k)} = \sum_{j=-\infty}^{\infty} (\wch{\varphi}_j \ast f)^{(k)} + P_k \in C(\R) \cap \ell^{\infty}(\R,\cB_{\R})+\C[\lambda] \numberthis\label{eq.decompoffk}
\]
as tempered distributions.
Thus $f \in C^k(\R)$, i.e., we have $\dot{B}_1^{k,\infty}(\R) \subseteq C^k(\R)$.

\begin{proof}[Proof of Theorem \ref{thm.Peller}]
We know from \eqref{eq.seriesbound} and \eqref{eq.decompoffk} that if $f \in \dot{B}_1^{k,\infty}(\R)$, then
\[
\sum_{j \in \Z} \big\|(\wch{\varphi}_j \ast f)^{(k)} \big\|_{L^{\infty}} < \infty,
\]
and $f^{(k)}$ differs from the bounded continuous function $\sum_{j \in \Z}(\wch{\varphi}_j\ast f)^{(k)}$ by a polynomial $P_k$.
If $f \in PB^k(\R)$, then $f^{(k)}$ is itself bounded, so $P_k$ must also be bounded and therefore constant.
Write $C \in \C$ for this constant.
Now, fix $\blambda \coloneqq (\lambda_1,\ldots,\lambda_{k+1}) \in \R^{k+1}$.
Then Proposition \ref{prop.divdiffCk} (twice), the uniform convergence of the series, and the fact that $\rho_k(\Delta_k) = \frac{1}{k!}$ give
\begin{align*}
    f^{[k]}(\blambda) & = \int_{\Delta_k}f^{(k)}(\boldsymbol{t} \cdot \blambda) \,\rho_k(d\boldsymbol{t}) = \int_{\Delta_k}\Big(C+\sum_{j \in \Z}(\wch{\varphi}_j\ast f)^{(k)}(\boldsymbol{t} \cdot \blambda)\Big) \,\rho_k(d\boldsymbol{t})\\
    & = \frac{C}{k!} + \sum_{j \in \Z}\int_{\Delta_k}\big(\wch{\varphi}_j \ast f\big)^{(k)}(\boldsymbol{t} \cdot \blambda)\, \rho_k(d\boldsymbol{t}) = \frac{C}{k!} + \sum_{j \in \Z}(\wch{\varphi}_j \ast f)^{[k]}(\blambda).
\end{align*}
Next, fix $j \in \Z$.
Since $f \in \dot{B}_1^{k,\infty}(\R)$, we have that $\wch{\varphi}_j \ast f$ satisfies the hypotheses of Proposition \ref{prop.Pellerprecursor} with $\sigma = 2^{j+1}$.
Therefore, completeness of $BOC(\R)^{\iotimes(k+1)}$, Proposition \ref{prop.Pellerprecursor}, and the definition of $\|\cdot\|_{\dot{B}_1^{k,\infty}}$ give
\[
\big\|f^{[k]}\big\|_{BOC(\R)^{\iotimes(k+1)}} \leq \frac{|C|}{k!} + \sum_{j \in \Z} \underbrace{\big\|(\wch{\varphi}_j \ast f)^{[k]}\big\|_{BOC(\R)^{\iotimes(k+1)}}}_{\leq a_k(2^{j+1})^k \|\scalebox{0.7}{\wch{\varphi}}_j \ast f\|_{L^{\infty}}} \leq \frac{|C|}{k!} + 2^ka_k\|f\|_{\dot{B}_1^{k,\infty}}.
\]
Finally, recalling the definition of $C$ and using \eqref{eq.seriesbound} again, we get
\[
|C| \leq \inf_{t \in \R}\big|f^{(k)}(t)\big| + \sum_{j \in \Z}\big\|(\wch{\varphi}_j \ast f)^{(k)}\big\|_{L^{\infty}} \leq \inf_{t \in \R}\big|f^{(k)}(t)\big| + b_k\|f\|_{\dot{B}_1^{k,\infty}}.
\]
It follows that we may take $c_k \leq\frac{b_k}{k!} + 2^ka_k$ in the statement of the theorem.

For the second statement, suppose $f \in PB^1(\R)$.
Then, by the above, there is some $C \in \C$ such that
\[
f' = C+\sum_{j \in \Z} \wch{\varphi}_j \ast f' = C+\sum_{j \in \Z} (\wch{\varphi}_j \ast f)'.\numberthis\label{eq.f'series}
\]
Now, it is easy to see that
\[
\dot{B}_1^{1,\infty}(\R) \cap \dot{B}_1^{k,\infty}(\R) = \bigcap_{\ell=1}^k \dot{B}_1^{\ell,\infty}(\R).
\]
Therefore, if $f \in PB^1(\R) \cap \dot{B}_1^{k,\infty}(\R) \subseteq \bigcap_{\ell=1}^k \dot{B}_1^{\ell,\infty}(\R)$ as well, then $\sum_{j \in \Z} \|(\wch{\varphi}_j \ast f)^{(\ell)} \|_{L^{\infty}} < \infty$ whenever $1 \leq \ell \leq k$.
This ensures that we can differentiate the series in \eqref{eq.f'series} to conclude that
\[
f^{(\ell)} = \sum_{j \in \Z} \big(\wch{\varphi}_j \ast f\big)^{(\ell)}
\]
and thus $f \in PB^{\ell}(\R)$, whenever $1 < \ell \leq k$.
In addition, when $k \geq 2$, the previous paragraph's analysis gives
\[
\big\|f^{[k]}\big\|_{BOC(\R)^{\iotimes(k+1)}} \leq 2^ka_k\|f\|_{\dot{B}_1^{k,\infty}}.
\]
We have therefore proven the desired bound and that $PB^1(\R) \cap \dot{B}^{k,\infty}(\R) = PB^1(\R) \cap \cdots \cap PB^k(\R)$.
\end{proof}
\begin{rem}
Note that if we required $f \in \dot{B}_1^{1,\infty}(\R)$ and $f' = \sum_{j \in \Z} \wch{\varphi}_j \ast f'$, instead of only $f \in PB^1(\R)$, then we would get $\|f^{[1]}\|_{BOC(\R) \iotimes BOC(\R)} \leq 2a_1\|f\|_{\dot{B}_1^{1,\infty}}$ from this proof.
\end{rem}

\begin{ack}
\phantomsection
\addcontentsline{toc}{section}{Acknowledgements}
I acknowledge support from NSF grant DGE 2038238.
I am grateful to Bruce Driver for numerous helpful conversations and valuable guidance about writing.
I would also like to thank Edward McDonald for bringing \cite{doddssub2} to my attention and Junekey Jeon for bringing Lemma \ref{lem.compact} to my attention.
\end{ack}


\small
\begin{thebibliography}{1}
\addcontentsline{toc}{section}{References}
\setlength\itemsep{0em}
\bibitem{aleksandrovOL}
A. B. Aleksandrov and V. V. Peller.
``Operator Lipschitz functions" (Russian).
\textit{Uspekhi Matematicheskikh Nauk} \textbf{71} (2016), 3--106.
English translation: \textit{Russian Mathematical Surveys} \textbf{71} (2016), 605--702.

\bibitem{azamovetal}
N. A. Azamov, A. L. Carey, P. G. Dodds, and F. A. Sukochev. 
``Operator Integrals, Spectral Shift, and
Spectral Flow."
\textit{Canadian Journal of Mathematics} \textbf{61} (2009), 241--263.

\bibitem{bensharp}
C. Bennett and R. Sharpley.
\textit{Interpolation of Operators}.
Pure and Applied Mathematics, vol. 129.
Academic Press Inc., Boston, 1988.

\bibitem{birmansolomyak5}
M. S. Birman and M. Z. Solomyak.
\textit{Spectral Theory of Self-Adjoint Operators in Hilbert Space} (Russian).
Izdat. Leningrad. Univ., Leningrad, 1980.
English translation: Mathematics and its Applications (Soviet Series), vol. 5.
D. Reidel Publishing Company, Dordrecht, 1987.

\bibitem{cohn}
D. L. Cohn.
\textit{Measure Theory}, 2nd ed.
Birkh\"{a}user Advanced Texts.
Springer, New York, NY, 2013.

\bibitem{conwayfunc}
J. B. Conway.
\textit{A Course in Functional Analysis}, 2nd ed.
Graduate Texts in Mathematics, vol. 96.
Springer-Verlag, New York, NY, 1990.

\bibitem{conwayop}
J. B. Conway.
\textit{A Course in Operator Theory}.
Graduate Studies in Mathematics, vol. 21.
American Mathematical Society, Providence, RI, 2000.

\bibitem{cz}
M. M. Czerwi\'{n}ska and A. H. Kaminska.
``Geometric properties of noncommutative symmetric spaces of measurable operators and unitary matrix ideals."
\textit{Commentationes Mathematicae} \textbf{57} (2017), 45--122.

\bibitem{daletskiikrein}
Yu. L. Daletskii and S. G. Krein.
``Integration and differentiation of functions of Hermitian operators and applications to the theory of perturbations" (Russian).
\textit{Voronezh. Gos. Univ. Trudy Sem. Funktsional. Anal.} \textbf{1} (1956), 81--105.
English translation: \textit{Thirteen Papers on Functional Analysis and Partial Differential Equations}, pp. 1--30.
American Mathematical Society Translations: Series 2, vol. 47.
American Mathematical Society, Providence, RI, 1965.

\bibitem{dasilva}
R. C. da Silva.
``Lecture Notes on Noncommutative $L_p$-Spaces." 
Preprint,  \href{https://arxiv.org/abs/1803.02390}{arXiv:1803.02390} [math.OA] (2018).

\bibitem{depagtersukochev}
B. de Pagter and F. A. Sukochev.
``Differentiation of operator functions in non-commutative $L_p$-spaces."
\textit{Journal of Functional Analysis} \textbf{212} (2004), 28--75.

\bibitem{dixmierLp}
J. Dixmier.
``Formes lin\'{e}aires sur un anneau d'op\'{e}rateurs" (French).
\textit{Bulletin de la Soci\'{e}t\'{e} Math\'{e}matique de France} \textbf{81} (1953), 9--39.

\bibitem{dixmier}
J. Dixmier.
\textit{Von Neumann Algebras}.
North-Holland Publishing Company, Amsterdam, 1981.

\bibitem{doddsNKS}
P. G. Dodds and B. de Pagter.
``Normed K\"{o}the spaces: A non-commutative viewpoint."
\textit{Indagationes Mathematicae} \textbf{25} (2014), 206--249.

\bibitem{doddsNKD}
P. G. Dodds, T. K. Dodds, and B. de Pagter.
``Noncommutative K\"{o}the duality."
\textit{Transactions of the American Mathematical Society} \textbf{339} (1993), 717--750.

\bibitem{doddssub2}
P. G. Dodds, T. K. Dodds, F. A. Sukochev, and D. Zanin.
``Arithmetic--Geometric Mean and Related Submajorisation and Norm Inequalities for $\tau$-Measurable Operators: Part II."
\textit{Integral Equations and Operator Theory} \textbf{92} (2020), Article 32.

\bibitem{fackkosaki}
T. Fack and H. Kosaki.
``Generalized $s$-numbers of $\tau$-measurable operators."
\textit{Pacific Journal of Mathematics} \textbf{123} (1986), 269--300.

\bibitem{folland}
G. B. Folland.
\textit{Real Analysis: Modern Techniques and Their Applications}, 2nd ed.
Pure and Applied Mathematics.
John Wiley \& Sons, New York, NY, 1999.

\bibitem{kadisonringrose1}
R. V. Kadison and J. R. Ringrose.
\textit{Fundamentals of Operator Algebras, Volume I: Elementary Theory}.
Graduate Studies in Mathematics, vol. 15.
American Mathematical Society, Providence, RI, 1997.

\bibitem{kaltonsukochev}
N. J. Kalton and F. A. Sukochev.
``Symmetric norms and spaces of operators."
\textit{Journal f\"{u}r die reine und angewandte Mathematik} \textbf{621} (2008), 81--121.

\bibitem{kreininterp}
S. G. Krein, Ju. I. Petunin, and E. M. Semenov.
\textit{Interpolation of Linear Operators}.
Translations of Mathematical Monograph, vol. 54,
American Mathematical Society, Providence, RI, 1982.

\bibitem{lemerdymcdonald}
C. Le Merdy and E. McDonald.
``Necessary and sufficient conditions for $n$-times Fr\'{e}chet differentiability on $\cS^p$, $1 < p < \infty$."
\textit{Proceedings of the American Mathematical Society} \textbf{149} (2021), 3881--3887.

\bibitem{lemerdyskripka}
C. Le Merdy and A. Skripka.
``Higher order differentiability of operator functions in Schatten norms."
\textit{Journal of the Institute of Mathematics of Jussieu} \textbf{19} (2020), 1993--2016.

\bibitem{murrayvonneumann}
F. J. Murray and J. von Neumann.
``On Rings of Operators."
\textit{Annals of Mathematics}, Second Series \textbf{37} (1936), 116--229.

\bibitem{nelson}
E. Nelson.
``Notes on Non-commutative Integration."
\textit{Journal of Functional Analysis} \textbf{15} (1974), 103--116.

\bibitem{nikitopoulosMOI}
E. A. Nikitopoulos.
``Multiple operator integrals in non-separable von Neumann algebras."
\textit{Journal of Operator Theory} \textbf{89} (2023), 361--427.

\bibitem{peetre}
J. Peetre.
\textit{New Thoughts on Besov Spaces}.
Duke University Mathematics Series, vol. 1.
Duke University Press, Durham, NC, 1976.
\pagebreak

\bibitem{peller0}
V. V. Peller.
``Hankel operators in the perturbation theory of unitary and self-adjoint operators" (Russian).
\textit{Funktsional. Anal. i Prilozhen.} \textbf{19} (1985), 37--51.
English translation: \textit{Functional Analysis and Its Applications} \textbf{19} (1985), 111--123.

\bibitem{peller1}
V. V. Peller.
``Multiple operator integrals and higher operator derivatives."
\textit{Journal of Functional Analysis} \textbf{233} (2006), 515--544.

\bibitem{peller2}
V. V. Peller.
``Multiple operator integrals in perturbation theory."
\textit{Bulletin of Mathematical Sciences} \textbf{6} (2016), 15--88.

\bibitem{sawano}
Y. Sawano.
\textit{Theory of Besov Spaces}.
Developments in Mathematics, vol. 56.
Springer, Singapore, 2018.

\bibitem{skripka}
A. Skripka and A. Tomskova.
\textit{Multilinear Operator Integrals: Theory and Applications}.
Lecture Notes in Mathematics, vol. 2250.
Springer, Cham, 2019.

\bibitem{takesaki}
M. Takesaki.
\textit{Theory of Operator Algebras I}.
Springer-Verlag, New York, NY, 1979.

\bibitem{terp}
M. Terp.
``$L^p$ spaces associated with von Neumann algebras."
Notes, Matematisk Institute, Copenhagen University, Copenhagen, 1981.

\bibitem{triebel1}
H. Triebel.
\textit{Theory of Function Spaces}.
Monographs in Mathematics, vol. 78.
Birkh\"{a}user Verlag, Basel, 1983.

\bibitem{vakhania}
N. N. Vakhania, V. I. Tarieladze, and S. A. Chobanyan.
\textit{Probability Distributions on Banach Spaces}.
Mathematics and its Applications (Soviet Series), vol. 14.
D. Reidel Publishing Company, Dordrecht, 1987.

\bibitem{zaanen}
A. C. Zaanen.
\textit{Integration}, 2nd ed.
North-Holland Publishing Company, Amsterdam, 1967.
\end{thebibliography}
\end{document}